\documentclass[english]{article}
\usepackage{mystyle_arxiv}
\bibliography{references.bib}

\title{Homogenization of Interaction Energy of Dislocation Loops}
\author{Pascal Steinke}
\date{\today}

\begin{document}
	
	\maketitle
	\begin{abstract}
		We consider the sum of interaction energies of an array of dislocation loops placed on a 
		homogenization lattice as the lattice spacing and diameter of the loops 
		tend to zero. Under suitable well-separation and $\lp^{ 2 
		}$-boundedness assumptions, a formula for the limit of the total interaction energy 
		is derived in terms of the weak limit, the generated $\Hm$-measure and 
		the generated Wigner measure of the associated Burgers vectors and 
		oriented surface areas.
	\end{abstract}
	
	\tableofcontents
	\section{Introduction}

Dislocations are a type of crystallographic defect in the atomic structure of metals. 
Already in 
\cite{volterra_07_equilibrium_of_elastic_bodies_with_multiple_components}, 
Volterra described a procedure in which an elastic body is cut, displaced 
across the cut surface, and then allowed to respond elastically.
This is essentially the slip-surface picture we will adopt below, and it  
gave a reasonable prediction of the force required to initiate plastic slip.
From an experimental side, the works \cite{muegge_mineralogie} and \cite{ewing_rosenhain_1899} make early contributions relevant to plasticity of crystal structures, and, in particular, observe slip phenomena. 
By using x-rays, the crystallographic structure of metals was established, and subsequently the first discrepancies between theoretical computations and experimental results of yield strength were discovered. 
To explain this phenomenon, \cite{orowan_34_kristallplastizitaetI}, \cite{orowan_34_kristallplastizitaetIII}, \cite{polanyi_34_gitterstoerung} and \cite{taylor_34_mechanism_of_plastic_deformation_of_crystals} first described the concept of an edge dislocation. In \cite{burgers_39_proceedings_I} and \cite{burgers_39_proceedings_II}, the description of a screw dislocation was added. 
Peierls \cite{peierls_40_size_of_a_dislocation} and Nabarro \cite{nabarro_47_dislocations} subsequently refined this picture by coupling the elastic description of a dislocation to the periodic structure of the crystal lattice, leading to the classical Peierls--Nabarro model, adopted by a number of mathematical treatments, and to the notion of a lattice resistance to glide.
During the following years, more evidence was provided for the existence of dislocations (see for example \cite{dash_57_observation_of_dislocations_in_silicon} and \cite{gilman_johnston_57_origin_and_growth_of_glide_bands_in_lithium_fluoride_crystals}).

Consider the elastic strain $ F $ and the associated strain energy
\begin{equation*}
	\frac{1}{2}
	\int_{ \Omega }
	\eltensor F
	\colon 
	F
	\dd{ x },
\end{equation*}
where $ \eltensor $ is the elastic tensor. If $ F $ is the strain of two disjoint dislocations, the energy splits into diagonal terms and off-diagonal (cross) terms. The diagonal terms are what we call the self-energy, and the off-diagonal terms are what we refer to as the interaction energy.
The self-energy is the energetic cost of creating a dislocation, which is a local defect of the crystal lattice. 
Actually, the integral is not well-defined, and by cutting out a cylinder of diameter $ \eps > 0 $, we note that it diverges logarithmically in $ \eps  $, which can be thought of as the atomic lattice spacing, see for example the textbook \cite[Sct.~6]{anderson_hirth_lothe_theory_of_dislocations}. Moreover in the dilute regime, this is the dominant energy contribution.

Besides the self-energy, the \emph{interaction energy} can be interpreted in the sense that the presence of already existing dislocations may favor or impede the formation and movement of other dislocations. Moreover two dislocations might repel or attract each other, see \cite[Sct.~4]{anderson_hirth_lothe_theory_of_dislocations}. We also want to point out that this energy can be negative, and we will in fact show that rapid oscillations in the orientations of a large number of dislocations can form a large total negative interaction energy, see also \Cref{cor:negative_interaction}.

The interaction energy of separate dislocation loops does not share the same logarithmic divergence of the self-energy as $ \eps \to 0 $, and is thus of lower order. 
In the present work, we focus solely on the interaction energy, and thus aim to complement the work \cite{conti_garroni_ortiz_the_line_tension_approximation_as_the_diluate_limit_of_linear_elastic_dislocations}, which identifies the line-tension approximation as the $ \Gamma $-limit of regularized linear elastic energy. 
The natural question is to investigate the interaction energy of dislocations in a three-dimensional setting.
Our aim is to contribute to the understanding of interaction energies of dislocations in three dimensions, under the geometric constraint that the dislocations are placed on a lattice whose spacing we will send to zero.

There have been a vast number of papers treating dislocations from a variational perspective. The main guiding question has been how to arrive at a continuum model after starting from a discrete dislocation model.
In the early work \cite{ortiz_plastic_yielding_as_phase_transition}, the author proposes a phase-field model which develops a statistical mechanics theory for forest hardening, which is supplemented by \cite{cuitino_koslowski_ortiz_a_phase_field_theory_of_dislocation_dynamics_strain_hardening_and_hysteresis_in_ductile_single_crystals} to explain a wider range by phenomena like strain hardening. Moreover it has been extended to the multiphase case in \cite{koslowski_ortiz_2004_multi-phase_field_model_of_planar_dislocation_networks}.

Building on this model, the authors prove in 
\cite{garroni_mueller_gamma_limit_of_a_phase_field_model_of_dislocations} 
the $ \Gamma $-convergence of an appropriate phase-field model in two 
dimensions with a discrete array of obstacles in the dilute regime, so that 
the leading order term of the energy is given by a cell-formula. If the 
number of obstacles scales critically like $ (\eps \abs{ \log ( \eps ) 
})^{-1}$, 
they 
show in 
\cite{garroni_mueller_a_variational_model_for_dislocations_in_the_line_tension_limit}
that a line-tension model is obtained in the $ \Gamma $-limit. 
See also 
\cite{conti_garroni_mueller_singular_kernels_multiscale_decomposition_of_microstructure_and_dislocation_models}
for the vectorial case, and 
\cite{conti_garroni_mueller_23_derivation_of_strain_gradient_plasticity_from_a_generalized_peierls_nabarro_model}
for higher-order terms which lead to self-energy and interaction energy in 
the $ \Gamma $-limit.
Also from a variational perspective, the authors considered in 
\cite{cermelli_leoni_05_renormalized_energy_and_forces_on_dislocations} the 
case of point defects in the plane which model dislocations.
In \cite{ponsiglione_2007_elastic_energy_screw_dislocations_from_discrete_to_continuous} the author showed $ \Gamma $-convergence for an elastic model with a core-cutoff in a three-dimensional cylinder, which is essentially a reduction to the two-dimensional case.
In 
\cite{garroni_leoni_ponsiglione_2010_gradient_theory_for_plasticity_via_homogenization_of_discrete_dislocations},
the authors develop a gradient theory for plasticity, starting from a 
discrete model for dislocations. The limiting energy in the critical $ \eps 
\log( \eps ) $-regime is given by the elastic energy of the macroscopic 
strain, and a line-tension energy of its curl. 
In the work \cite{ginster_19_plasticity_as_Gamma_without_separation}, the 
author also considered the derivation of plasticity from a two-dimensional 
model of dislocation energy, but critically managed to remove the 
assumption of well-separation. Moreover in 
\cite{ginster_19_strain_gradient_plast_mixed_growth} he considered the case 
of subquadratic growth at the cores, removing the necessity for a 
core-cutoff.
We also note the works 
\cite{scardia_zeppieri_2012_line_tension_model_for_plasticity_as_gamma_limit_of_nonlinear_dislocation_energy,
	mueller_scardia_zeppieri_14_geometric_rigidity_straing_gradient_pl}
on the case of non-linear elasticity.

In three dimensions, the literature is sparser. The authors showed in \cite{conti_garroni_ortiz_the_line_tension_approximation_as_the_diluate_limit_of_linear_elastic_dislocations} that both for a core-cutoff and core-mollification model, the elastic energies converge after a $ \log ( \eps ) $-rescaling to a line-tension energy. 
A first step to a higher-order version of this result, and the 
higher-dimensional equivalent of 
\cite{garroni_leoni_ponsiglione_2010_gradient_theory_for_plasticity_via_homogenization_of_discrete_dislocations},
has been done in 
\cite{fortuna_garroni_25_homogenization_of_line_tension_energies}, where 
the authors show a  homogenization result for the line-tension energies, 
which are represented by rectifiable $1 $-currents. 
Moreover the authors identified in 
\cite{fonseca_ginster_wojtowytsch_2021_on_the_motion_of_curved_dislocations_in_3d}
that the Peach--Köhler force of a single dislocation approaches the mean 
curvature of the dislocation, at 
least in the setting of simplified elasticity.

The result in \cite{conti_garroni_ortiz_the_line_tension_approximation_as_the_diluate_limit_of_linear_elastic_dislocations} provides a $ \Gamma $-limit which only sees the self-energy of the dislocations. Hence the purpose of the current paper is to study the limit of the total elastic interaction energy in three dimensions of an array of dislocations, while neglecting their self-energy. 
We consider a dislocation density model.
Our setup is that we place dislocation loops $ b \otimes \tau \hm^{ 1 } \llcorner_{ \gamma } $, where $ b $ is the Burgers vector, $ \gamma $ the dislocation line, and $ \tau $ its tangent vector, at each point of a homogenization lattice $ \lattice_{ \rho }$, see also \Cref{fig:homogenization}. We assume that the loops are well-separated in the sense that their diameter is much smaller than their distance $ \rho $, see also assumption \ref{item:wellSeperated}. The first observation, made in \Cref{sct:interaction_energy_via_oriented_surface_area}, is that to leading order, we can express the total interaction energy of these dislocations as a sum
\begin{equation}
	\label{eq:discrete_sum}
	\frac{1}{2}
	\sum_{ \substack{x, y \in \lattice_{ \rho } \\ x \neq y }}
	\of_{ \rho } ( x ) \colon M ( x - y ) \of_{ \rho } ( y ).
\end{equation}
Here $ \of_{ \rho } ( x ) \coloneqq \int_{ S } b \otimes n \dd{ \hm^{ 2 } } $ is the oriented surface area of the corresponding slip surface $ S $ with oriented normal $ n $ of the dislocation loop placed at the lattice point $ x \in \lattice_{ \rho } $. 
The kernel $ M $ is $ - 3 $-homogeneous, since it is precisely given by $ \eltensor (\diff^{ 2 } K) \eltensor^{ \top } $, where $ \eltensor $ is the elasticity tensor and $ K $ Green's function of the operator $\mathcal{L}( u ) = - \divg ( \eltensor \diff u ) $.
We note that the relevant quantity $ \of_{\rho } ( x ) $ can be uniquely identified by the Burgers vector and the oriented dislocation line $ \gamma $, as long as $ \partial S = \gamma $ holds (in the sense of currents). This is a consequence of Stokes Theorem.

The important relation between the atomic lattice spacing $ \eps > 0 $, the diameter of the loops $ r > 0 $ and the lattice spacing of the homogenization lattice $ \rho > 0 $ is summarized as
\begin{equation*}
	\eps \ll r \ll \rho.
\end{equation*}
This key geometric assumption is essential for our work and can not be removed. We believe that the periodicity in the model might be removed, as often demonstrated in the modern framework of (stochastic) homogenization.

We want to note that the interaction energy $ \of_{ \rho } ( x ) \colon M ( x - y ) \of_{ \rho } ( y ) $ is akin to what we would expect as an interaction energy from a physical perspective. 
Namely if $ \gamma_{ 1 } $ and $ \gamma_{ 2 } $ denote the dislocation lines for the two loops, with tangent vectors $ \tau_{ 1 } $ respectively $ \tau_{ 2 } $, then the interaction energy between them should be of the form
\begin{equation*}
	\int_{ \gamma_{ 1 } }
	\int_{ \gamma_{ 2 } }
	b_1 \otimes \tau_{ 1 } \colon K ( x - y ) b_2 \otimes \tau_{ 2 }
	\dd{ \hm^{ 1 } ( y ) } 
	\dd{ \hm^{ 1 } ( x ) }.
\end{equation*}
See also \Cref{rmk:alt_derivation} for a note on why these formulas coincide.
If $ \gamma_1 $ and $ \gamma_2 $ are, however, close to each other, then it is difficult to split this interaction energy from the self-energy.

The energy in (\ref{eq:discrete_sum}) describes the interaction of an array of tensors with a kernel of critical scaling. The analogous but simpler question of an array of dipoles has been studied in \cite{james_mueller_internal_variables_and_fine_scale_oscillations_in_micromagnetics}. In that case the kernel is determined by the Maxwell equations and given by the second derivative of the fundamental solution of the Laplace operator $ - \Delta $. In that setting, it has been noted in \cite{james_mueller_internal_variables_and_fine_scale_oscillations_in_micromagnetics} and \cite{firoozye_93_homogenization_on_lattices} that in order to compute a limiting energy, we have to make a case distinction depending on the scale of oscillations, which are referred to as \emph{long-range} and \emph{short-range} oscillations, see \Cref{def:weak_long} and \Cref{def:weak_short}. 

Returning to the energy in (\ref{eq:discrete_sum}), we first explore the case of long-range oscillations in \Cref{sct:weak_long}. The strategy in this case is to use the distributional convergence of $ \fourier (\chi_{ \R^{ 3 } \setminus B_r ( 0 ) } M ) $ to $ \fourier (M) - \fint_{ \Sph^{ 2 } } \fourier M $ and approximate the sum (\ref{eq:discrete_sum}) by a double integral. This approximation, however, is only valid due to the assumption that no short-range oscillations occur. Passing to the limit then reveals that the limiting energy can be expressed in terms of the weak $\lp^{ 2 } $-limit of the oriented surface areas $ \of_{ \rho } $, and the $ \Hm $-measure generated by their oscillations, as introduced in \cite{tartar_H_measures_a_new_approach}.

In the case of short-range oscillations, this integral approximation is no longer valid, and instead, inspired by \cite{firoozye_93_homogenization_on_lattices}, we use a Fourier series approach. It then turns out that the associated discrete Wigner measure, whose continuous version was originally introduced for dispersive partial differential equations (mainly the Schrödinger equation) in \cite{wigner_1932_quantum_correction,lions_paul_93_on_wigner_measures}, captures the total interaction energy. 
The definition of Wigner measure is similar to those of the $ \Hm $-measure, the difference being that a scale has to be fixed beforehand. 
The advantage of the Wigner measure is that it can capture the magnitude of the oscillations. 
For a more thorough discussion, see \cite{gerard_markowich_mauser_97_hom_limits_and_wigner_transforms}.
Passing to the limit of the Wigner measures is non-trivial due to the lack of continuity of the associated Fourier series of $ M $ at zero. In the case of short-range oscillations however, we can show that the Wigner measures exhibit no concentration in zero, which lets us pass to the limit.

To obtain a limiting energy for general sequences of arrays of dislocation loops, in \Cref{sct:scale_separation} we split the sequence through a convolution into its long- and short-range oscillatory parts. By showing that the total energy splits additively when choosing the mollification parameter suitably, we are thus able to combine the result of the long- and short-range case to get that the limit of the total interaction energy can be written as 
\begin{equation*}
	\frac{1}{2}
	\left( \inner*{ \fourier  \of }{\Psi \fourier  \of }_{ \lp^{2 } ( \R^{ 3 } ) } 
	+ 
	\int_{ \R^{ 3 } \times \Sph^{2 } } \Psi ( \nu ) \cdot \dd{ \mu_{ \Hm } ( x , \nu ) }
	+
	\int_{ U^{ \ast } } \idfs{ M } ( \xi ) \cdot \dd{ \mu_{ \mathrm{W } } ( \xi ) } 
	\right).
\end{equation*}
Here $ \Psi = \fourier M - \fint_{ \Sph^{ 2 } } \fourier M + S $, where  $ S $ is a lattice sum, is $ 0 $-homogeneous,  $\of $ is the weak limit of $ \of_{ \rho } $, $ \mu_{ \Hm } $ the generated $ \Hm $-measure of the long-range oscillatory part, and $ \mu_{ \mathrm{W}} $ the limit of the Wigner measures of the short-range oscillations. See \Cref{sct:main_result} for a more thorough explanation of these notions, and \Cref{thm:main_theorem} for the main result. Let us also note that given an array of dislocation loops, the associated $ \Hm $-measures and Wigner measures can be explicitly computed, see also \Cref{ex:h_measures_of_simple_form}.

At this point, we turn to addressing consequences of our main result. 
\Cref{cor:negative_interaction} states that even though the weak-limit of 
the oriented surface areas $ \of_{ \rho } $ might be zero, we can create a 
negative limit of the total interaction energy by suitably chosen 
oscillations. 
To this end, we study the indefiniteness of the operator $ \mu \mapsto \int_{ \R^{ 3 } \times \Sph^{ 2 } } \Psi ( \nu ) \cdot \dd{ \mu_{ \Hm } ( x , \nu ) } $ on the space of $ \Hm $-measures in \Cref{sct:relaxation}. 
We do this in the case when the elasticity tensor $ \eltensor $ is 
isotropic and the homogenization lattice $ \lattice_{ 1 } $ has cubic 
symmetry in order to explicitly compute the lattice sum $ S $ and thus 
understand the structure of the map $ \Psi $.
Moreover we show in \Cref{ex:rotations_energy} how our result implies that 
given a constant configuration of dislocation loops, it is energetically 
more favourable to have fine-scale oscillations of the directions of the 
dislocation loops.

We will introduce our model and show how the interaction energy can be expressed in terms of the slip surface in \Cref{sct:model_and_main_result}. Our main results can then be found in \Cref{sct:main_result}. The core of the mathematical work then happens in \Cref{sct:from_loops_to_surfaces}, where we prove our main result in several steps as described above. We finish the main part in \Cref{sct:relaxation} by showing the relaxation result, and we have moved some of the proofs which we deemed to be of lesser interest to the reader to \Cref{sct:appendix}.

\subsection{Notation}

\begin{itemize}
	\item For $ f \in \lp^{ 1 } ( \R^{ n } ) $, we denote its Fourier transform by $ \fourier f $, which is given by
	\begin{equation*}
		\fourier f ( \xi )
		\coloneqq
		\int_{ \R^{ n } } f ( x ) \exp ( - 2 \pi i \inner*{x}{\xi } )
		\dd{ x },
	\end{equation*}
	and extend this definition to $ \lp^{ 2 } ( \R^{  n } ) $ via Plancherel's 
	identity. We sometimes also write $ \hat{ f } $ for the Fourier transform 
	when notationally convenient.
	\item We define the inverse discrete Fourier series $ \idfs{f} $ by equation (\ref{eq:inverse_discrete_fourier_series}).
	\item For $ f \in \cont^{ 1 } ( \R^{ n } ; \R^{ m } ) $, we define its Jacobian by $ \diff f ( x )_{ i, j } \coloneqq \partial_{ x_j } f_{ i } $.
	\item If $ A, B \in \R^{ 3 \times 3 } $, we denote by $ A \colon B \coloneqq \sum_{ i j } A_{ i j } B_{ i j } $ their inner product. Similarly, we define for $ v, w \in \R^{ 3 } $ their inner product by $ v \cdot w = \inner*{ v }{w} = \sum_{ i } v_i w_i $, and the same notation is also used for $  S, T \in \R^{ (3\times 3) \times (3 \times 3 ) } $.
	\item We denote by $ \eltensor $ the elastic tensor, which is an element of $ \R^{ (3 \times 3 ) \times (3 \times 3 ) } $, see also \Cref{sct:derivation_of_energy}.
	\item We denote by $ \lattice_{ 1 } $ our Bravais homogenization lattice, defined in \Cref{sct:main_result}.
	\item We denote by $ E_3 $ the identity matrix in three dimensions.
\end{itemize}

\section{Model}
\label{sct:model_and_main_result}

\subsection{Derivation of the Energy}
\label{sct:derivation_of_energy}

Let us first assume that $\disdens \in \ccinf\left(\R^3; \R^{3 \times 3}\right)$ is a smooth dislocation density with $ \divg \disdens = 0 $. Let $ \eltensor \in \mathrm{Lin}\left(\R^{3\times3}; \R^{3 \times 3}\right)$ be a symmetric elastic tensor, which means that $\eltensor F \colon G =F \colon \eltensor G$ for the standard scalar product on $\R^{3\times3}$, $\eltensor W = 0$ if $W^{ \top } = - W$ and $\eltensor F \colon F \ge c |F|^2$  for all $F\in \R^{ 3 \times 3 }$ with  $F^{ \top } = F$ and some $c > 0$.
The \emph{strain field} associated with $\disdens$ is the unique  solution  in $\lp^2\left(\R^{ 3 }; \R^{3 \times 3}\right)$ of the system
\begin{align}
	\label{eq:curl_for_strain_field}
	\cur F &= \disdens, \\
	\label{eq:div_for_strain_field}
	\divg \eltensor F &= 0,
\end{align}
see also \cite[Thm.~4.1]{conti_garroni_ortiz_the_line_tension_approximation_as_the_diluate_limit_of_linear_elastic_dislocations} for existence and uniqueness, and \Cref{lemma:existence_of_strain_field} for a simpler proof in our setting.
Here $\cur$ and $\divg$ are taken rowwise.
We call $\sigma := \eltensor F$ the associated stress field. The elastic energy induced by $\disdens $ is defined as
\begin{equation*}
	\energy \coloneqq \frac12 \int_{\R^3}  \eltensor F \colon F  \dd{x}.
\end{equation*}

If $\disdens = \disdens_1 + \disdens_2$ and $F_i$ is the strain field of $\disdens_i$, then we have $\energy= \energy_{11}+ \energy_{22} + \energy_{12} + \energy_{21}$,
where $\energy_{11} + \energy_{22}$ is the combined self-energy of the dislocation densities $\disdens_1$ and $\disdens_2$,
and 
\begin{equation} 
	\label{eq:def_interaction_energy_primal}
	\interactionEnergy := \energy_{12} + \energy_{21} = \frac12  \int_{\R^3} \eltensor F_1 \colon F_2 + \eltensor F_2 \colon F_1 \dd{x}  = \int_{\R^3} \eltensor F_1 \colon F_2 \dd{ x }
\end{equation}
is the interaction energy. 
Our goal is now to express the interaction energy in terms of the slip surfaces associated to $ \disdens_1 $ respectively $ \disdens_2 $.

To this end assume that there exist $ G_{ i } \in \lp^{ 2 } \left( \R^{ 3 } ; \R^{ 3 \times 3 } \right) $ such that
\begin{equation}  \label{eq:curlG_rho}  \cur G_i = \disdens_i \quad \text{and}   \quad  \dist(\spt G_1, \spt G_2) > 0.
\end{equation}

\begin{example}
	\label{ex:dislocation_density_and_slip_surface}
	For us the following example is the most important one. 
	Assume that $\gamma$ is an oriented and closed $\cont^1$-curve (the dislocation line) and $S$ is an oriented surface with $\partial S = \gamma$ (the slip surface).
	More precisely this means that $\cur  n \mathcal H^2|_S = \tau \mathcal H^1|_\gamma$ where
	$n$ is the unit normal of $S$ and $\tau$ is the unit tangent vector of $\gamma$.
	Let $b \in \mathbb{R}^3$ be a constant Burgers vector. 
	Then 
	\begin{equation}
		\label{eq:curl_of_surface_is_curve} 
		\cur \left[ (b\otimes n) \mathcal H^2|_S\right] =  (b \otimes \tau) \mathcal H^1|_\gamma.
	\end{equation}
	If we consider $ b \otimes [S, n , 1 ] $ as a vector-valued 2-current on $ \R^{ 3 } $, and $ b \otimes [\gamma, \tau , 1 ] $ as a vector-valued 1-current on $ \R^{ 3 } $, then equation (\ref{eq:curl_of_surface_is_curve}) can equivalently be written as
	\begin{equation*}
		\partial ( b \otimes [S, n , 1 ] )
		=
		b \otimes [\gamma, \tau , 1 ]
	\end{equation*}
	in the sense of currents.

	\begin{figure}
		\tikzset{every picture/.style={line width=0.75pt}} 
		\hspace{2cm}
		\begin{tikzpicture}[x=0.75pt,y=0.75pt,yscale=-1,xscale=1]
			
			\draw  [color={rgb, 255:red, 65; green, 117; blue, 5 }  ,draw opacity=1 ][fill={rgb, 255:red, 74; green, 144; blue, 226 }  ,fill opacity=1 ] (100.54,95) .. controls (120.54,85) and (237.04,40) .. (205.5,79) .. controls (173.96,118) and (202.54,175) .. (204.54,150.5) .. controls (205.68,136.53) and (180.82,135.24) .. (156.82,136.43) .. controls (138.72,137.33) and (121.11,139.64) .. (115.5,139) .. controls (102.46,137.5) and (80.54,105) .. (100.54,95) -- cycle ;
			\draw    (140.54,93.67) -- (140.89,65.22) ;
			\draw [shift={(140.92,63.22)}, rotate = 90.71] [color={rgb, 255:red, 0; green, 0; blue, 0 }  ][line width=0.75]    (10.93,-3.29) .. controls (6.95,-1.4) and (3.31,-0.3) .. (0,0) .. controls (3.31,0.3) and (6.95,1.4) .. (10.93,3.29)   ;
			\draw  [color={rgb, 255:red, 65; green, 117; blue, 5 }  ,draw opacity=1 ][fill={rgb, 255:red, 74; green, 144; blue, 226 }  ,fill opacity=1 ] (442.54,66.5) .. controls (462.54,58) and (481.54,103.5) .. (461.54,123.5) .. controls (441.54,143.5) and (393.04,171) .. (413.04,201) .. controls (433.04,231) and (442.54,146) .. (422.54,116) .. controls (402.54,86) and (422.54,75) .. (442.54,66.5) -- cycle ;
			\draw    (419.75,185.32) -- (396.67,170.26) ;
			\draw [shift={(395,169.17)}, rotate = 33.13] [color={rgb, 255:red, 0; green, 0; blue, 0 }  ][line width=0.75]    (10.93,-3.29) .. controls (6.95,-1.4) and (3.31,-0.3) .. (0,0) .. controls (3.31,0.3) and (6.95,1.4) .. (10.93,3.29)   ;
			\draw    (201.78,84.28) -- (216.44,63.2) ;
			\draw [shift={(217.58,61.56)}, rotate = 124.82] [color={rgb, 255:red, 0; green, 0; blue, 0 }  ][line width=0.75]    (10.93,-3.29) .. controls (6.95,-1.4) and (3.31,-0.3) .. (0,0) .. controls (3.31,0.3) and (6.95,1.4) .. (10.93,3.29)   ;
			\draw    (429.58,72.22) -- (454.21,60.05) ;
			\draw [shift={(456,59.17)}, rotate = 153.7] [color={rgb, 255:red, 0; green, 0; blue, 0 }  ][line width=0.75]    (10.93,-3.29) .. controls (6.95,-1.4) and (3.31,-0.3) .. (0,0) .. controls (3.31,0.3) and (6.95,1.4) .. (10.93,3.29)   ;
			
			\draw (123.13,84.23) node [anchor=north west][inner sep=0.75pt]  [font=\footnotesize]  {$n_{1}$};
			\draw (379.73,170.17) node [anchor=north west][inner sep=0.75pt]    {$n_{2}$};
			\draw (219.92,61.4) node [anchor=north west][inner sep=0.75pt]  [font=\footnotesize]  {$\tau _{1}$};
			\draw (160.58,140) node [anchor=north west][inner sep=0.75pt]  [font=\footnotesize,color={rgb, 255:red, 65; green, 117; blue, 5 }  ,opacity=1 ]  {$\gamma _{1}$};
			\draw (90.25,75.07) node [anchor=north west][inner sep=0.75pt]  [font=\footnotesize,color={rgb, 255:red, 74; green, 144; blue, 226 }  ,opacity=1 ]  {$S_{1}$};
			\draw (447.92,40.07) node [anchor=north west][inner sep=0.75pt]  [font=\footnotesize]  {$\tau _{2}$};
			\draw (472.25,76.07) node [anchor=north west][inner sep=0.75pt]  [font=\footnotesize,color={rgb, 255:red, 74; green, 144; blue, 226 }  ,opacity=1 ]  {$S_{2}$};
			\draw (392.92,92.07) node [anchor=north west][inner sep=0.75pt]  [font=\footnotesize,color={rgb, 255:red, 65; green, 117; blue, 5 }  ,opacity=1 ]  {$\gamma _{2}$};

		\end{tikzpicture}
		\caption{An illustration of two dislocation loops $ \gamma_{ 1 } $ and $ \gamma_2 $ with their respective slip surfaces $ S_1 $ and $ S_2 $.}
	\end{figure}
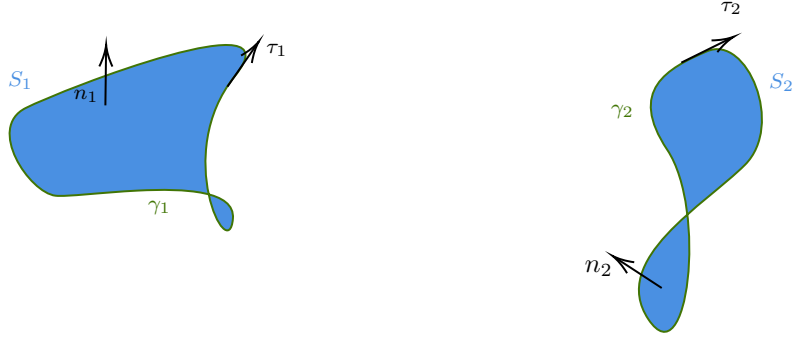
	
\end{example}

We will now express the interaction energy $ \interactionEnergy $  in terms of the slip surfaces $G_1$ and $G_2$. 
\begin{lemma}
	\label{lem:interaction_energy_via_surface}
	Let $ (\disdens_i)_{ i =1,2} \subseteq \ccinf ( \R^{ 3 } ; \R^{ 3 \times 3 } ) $ be dislocation densities and  $ (G_{ i })_{ i = 1,2 } $ be corresponding slip surfaces as in equation (\ref{eq:curlG_rho}) with disjoint support. Then
	\begin{equation}
		\label{eq:precise_interaction_energy}
		\interactionEnergy
		= 
		\int_{\R^3}  \int_{\R^3}   
		\sum_{i,j,i',j'}     M_{iji'j'}(x-y) (G_1)_{i'j'}(y) (G_2)_{ij}(x)\dd{ x } \dd{ y }
	\end{equation}
	for some $- 3$ homogeneous smooth function $ M \colon \R^{ 3 } \setminus \{ 0 \} \to \R^{ (3 \times 3) \times ( 3 \times 3 ) } $ which is symmetric and zero on skew-symmetric matrices in the sense that for all $ x \in \R^{ 3 } \setminus \{ 0 \} $, we have
	\begin{alignat*}{2}
		M ( x ) A \colon B & = A \colon M ( x ) B \quad &&\text{for all }A, B \in \R^{ 3 \times 3 },
		\\
		M( x ) S & = 0 &&\text{for all }S\in \R^{ 3 \times 3 }_{ \mathrm{skew}}. 
	\end{alignat*}
\end{lemma}

\begin{proof}
	For the associated strain fields as in equations (\ref{eq:curl_for_strain_field}) and (\ref{eq:div_for_strain_field}), we first write
	$$ F_i = G_i + H_i.$$
	Then $\cur H_i = 0$ by equations (\ref{eq:curl_for_strain_field}) and (\ref{eq:curlG_rho}) and thus there exist $u_i \in \dot{\ha}^{1} ( \R^3 ; \R^{ 3 } ) $ such that 
	$$ \diff u_i = H_i.$$
	Moreover $\divg \eltensor H_i = - \divg \eltensor G_i$ by equation (\ref{eq:div_for_strain_field}). Thus $u_i$ is the unique solution to the partial differential equation
	\begin{equation*}
		\mathcal{L} u_{ i }
		\coloneqq
		- \divg \left(\eltensor \diff u_{ i } \right)
		=
		\divg \eltensor G_{ i }
	\end{equation*}
	in the homogeneous space $ \dot{\ha}^{1} ( \R^{ 3 } ; \R^{ 3 }) $.
	Therefore we get
	\begin{align*}
		\interactionEnergy & = \int_{\R^3} \eltensor F_1 \colon F_2 \dd{ x } \\
		&= \int_{\R^3} \eltensor F_1 \colon G_2 \dd{x} + \int_{\R^3} \eltensor F_1\colon \diff u_2 \dd{ x } \\
		&= \int_{\R^3} \eltensor F_1 \colon G_2 \dd{ x }
	\end{align*}
	since $  F_2 = G_2 + \diff u_2 $ and $ \divg \eltensor F_{ 1 } = 0 $ by equation (\ref{eq:div_for_strain_field}). 
	Using again that $F_1 = G_1 + \diff u_1$ and that $\dist(\mathrm{spt} G_1, \mathrm{spt} G_2) > 0$, this proves 
	\begin{equation}
		\label{eq:interaction_energy_via_G}
		\interactionEnergy
		=
		\int_{ \R^{ 3 } }
		\eltensor \diff u_{ 1 } \colon G_{ 2 }
		\dd{ x }.
	\end{equation}
	Define the fundamental solution $K_{ij}$ by 
	\begin{equation*}  
		{\mathcal L}^{-1} (\delta_0 e_j )= \sum_{i=1}^3 K_{ij} e_i.
	\end{equation*}
	The inverse operator is a priori not well-defined since $ \mathcal{L} $ is not injective, but we choose the representative which is $-1$-homogeneous, see also \Cref{lem:symmetry_properties} (the inverse Fourier transform of a $ - 2 $ homogeneous function is $ - 1 $ homogeneous since we are in three dimensions).  
	Moreover, for $x \notin \spt G_1$
	$$  u_1(x) = \int_{\R^3}  \sum_{i,j,k} e_i  {K_{ij}}(x-y)  \partial_k  (\eltensor G_1)_{jk}(y) \dd{y }. $$
	Thus
	\begin{align*}(\diff u_1)_{il}(x)&  =  \sum_{j,k} \int_{\R^3}  \partial_l  K_{ij}(x-y) \,   \partial_k(\eltensor G_1)_{jk}(y) \dd{y}, 
		\shortintertext{or, after relabeling indices,}
		(\diff u_1)_{kl}(x) 
		& =  
		\sum_{k',l'} 
		\int_{\R^3}   \partial_l  K_{kk'}(x-y) \partial_{l'}(\eltensor G_1)_{k'l'}(y) \dd{y}.
	\end{align*}
	It follows that
	\begin{align*} (\eltensor \diff u_1)_{ij} 
		&= 
		\sum_{kl} \eltensor_{ijkl} ( \diff u_1)_{kl}  \\
		& = 
		\sum_{k,l} \sum_{k',l'} \int_{\R^3}  \eltensor_{ijkl}  \partial_l K_{kk'}(x-y) \, \partial_{l'} (\eltensor G_1)_{k'l'}(y) \dd{y}   \\
		& = 
		\sum_{k,l} \sum_{k',l'} \sum_{i',j'}  \int_{\R^3}  \eltensor_{ijkl}  \partial_l  K_{kk'}(x-y)  \eltensor_{k'l'i'j'} \, \partial_{l'} (G_1)_{i'j'}(y) \dd{y } 
	\end{align*}
	Finally we get from equation \eqref{eq:interaction_energy_via_G} and the assumption that $ G_{ 1 } $ and $ G_{ 2 } $ have compact, disjoint supports that
	\begin{align*}
		\notag 
		\interactionEnergy
		& =    
		\int_{\R^3} \int_{\R^3} \sum_{i,j,i',j',k,l,k',l'}  
		\eltensor_{ijkl} \partial_l  K_{kk'}(x-y) \eltensor_{k'l'i'j'}  \partial_{l'} (G_1)_{i'j'}(y) 
		(G_2)_{ij}(x) 
		\dd{ x }\dd{ y } 
		\\
		& = 
		\int_{\R^3}  \int_{\R^3}   
		\sum_{i,j,i',j'}     M_{iji'j'}(x-y) (G_1)_{i'j'}(y) (G_2)_{ij}(x) \dd{ x } \dd{ y }
	\end{align*}
	with 
	\begin{equation}  \label{eq:formula_M}
		M_{iji'j'}(z) \coloneqq
		\sum_{k,l, k', l'}   \eltensor_{ijkl}  \partial_{l l'}^{ 2 } K_{kk'}(z)   \eltensor_{k'l'i'j'}.
	\end{equation}
	In particular $M$ is homogeneous of degree $-3$ since $ K $ is homogeneous of degree $ - 1 $. 
	Moreover, we know that $K$ is symmetric because its Fourier transform is symmetric, see \Cref{lem:symmetry_properties}.
	Thus it follows from the properties of the elastic tensor that $M$ is symmetric, i.e. $M_{iji'j'} =M_{i'j'ij}$ for all $ i, j, i', j' \in \{1,2,3\}$, and that $ M $ is zero on skew-symmetric matrices, i.e. $ M_{ i j i' j' } = M_{ j i i' j' } $ for all $ i, j, i', j' \in \{1,2,3\}$.
\end{proof}

Using the representation of the interaction energy (\ref{eq:precise_interaction_energy}), we now define the interaction energy of non-regular dislocation densities. As in \Cref{ex:dislocation_density_and_slip_surface}, let $ \gamma_{ 1 } , \gamma_{ 2 } $ be closed $ \hm^{ 1 } $-rectifiable curves representing dislocation loops with Burgers vectors $ b_{ 1 } $ and $ b_{ 2 } $ respectively. Let $ S_{ 1 } $ and $ S_{ 2 } $ respectively be slip surfaces with disjoint support of the two dislocation loops with unit normals $ n_{ 1 } $ and $ n_{ 2 } $ respectively. 
Motivated by equation (\ref{eq:precise_interaction_energy}), we define the interaction energy of the two dislocation loops as
\begin{equation}
	\label{eq:interaction_energy_non_regular}
	\interactionEnergy
	\coloneqq
	\int_{ S_{ 1 } }
	\int_{ S_{ 2 } }
	b_{1}(x) \otimes n_{1} ( x ) 
	\colon
	M ( x - y )
	b_{2 }(y) \otimes n_{2} ( y ) 
	\dd{ \hm^{ 2 } ( y ) }
	\dd{ \hm^{ 2 } ( x ) }.
\end{equation}
Note that by the proof of \Cref{lem:interaction_energy_via_surface}, the definition does not depend on the choice of $ S_{ 1 } $ and $ S_{ 2 } $, but only on the dislocation lines and the Burgers vectors.
We also remark that if $ S_{ 1 } $ and $ S_{ 2 } $ are far away from each other relative to their diameter, then the double integral in formula (\ref{eq:interaction_energy_non_regular}) should not depend up to leading order on the individual interaction of points on $ S_{ 1 } $ with points on $ S_{ 2 } $, but only on the mean of their interactions. If $ S_{ 1 } $ is centered at $ x_{ 0 } $ and $ S_{ 2 } $ at $ y_{ 0 } $, we thus propose the approximation
\begin{equation}
	\label{eq:interaction_energy_approximation}
	\interactionEnergy
	\approx
	\left(  \int_{ S_{ 1 } }
	b_{ 1 } (x)
	\otimes
	n_{ 1 } ( x ) 
	\dd{ \hm^{ 2 } ( x ) }
	\right)
	\colon
	M(x_{ 0 } - y_{ 0 } )
	\left( 
	\int_{ S_{ 2 } } b_{ 2 } \otimes n_2 ( y ) \dd{ \hm^{ 2 } ( y ) }
	\right)
\end{equation}
which will be made rigorous in \Cref{prop:energy_asymptotics}.

Assuming that we have a countable collection of dislocation densities $ ( \disdens_{ n } )_{ n \in \N } $, we define the total interaction energy as
\begin{equation}
	\label{eq:def_total_interaction_energy}
	\interactionEnergy ( \disdens )
	\coloneqq
	\frac{ 1 }{ 2 }
	\sum_{ n \neq m \in \N }
	\interactionEnergy_{ n m }
\end{equation}
for the interaction energy $ \interactionEnergy_{nm} $ as in equation 
(\ref{eq:interaction_energy_non_regular}). Note that we will assume that 
the dislocation loops are confined to a bounded domain $ \Omega $ so that 
the above sum is finite and always well-defined.

\begin{remark}
	\label{rmk:alt_derivation}
	Another way of arriving at equality 
	(\ref{eq:interaction_energy_non_regular}) would be to use the formula 
	for the interaction energy given in 
	\cite[Eq.~(4.4)]{anderson_hirth_lothe_theory_of_dislocations}, which 
	was first obtained in \cite{blin_energy_of_dislocation_in_a_crystal}. 
	We, however, wanted to give a self-contained derivation of the interaction energy. 
	Moreover the derivation in \cite[Sct.~4.5]{anderson_hirth_lothe_theory_of_dislocations}
	also yields a formula for the interaction energy via the one-dimensional dislocation lines, see  \cite[Eq.~(4.40)]{anderson_hirth_lothe_theory_of_dislocations}. 
	The authors show that the interaction energy can also be written as
	\begin{equation*}
		\int_{ \gamma_{ 1 } }
		\int_{ \gamma_{ 2 } }
		b_{ 1 } \otimes \tau_{ 1 } ( x )\colon 
		N ( x- y ) 
		b_{ 2 } \otimes \tau_{ 2 } ( y )
		\dd{ \hm^{ 1 } ( y ) }
		\dd{ \hm^{ 1 } ( x ) }
	\end{equation*}
	for a $ -1 $-homogeneous kernel $ N $ by applying Stokes Theorem to the right hand side of (\ref{eq:interaction_energy_non_regular}).
\end{remark}

\subsection{Formula for the Kernel}
In this section, we discuss a more explicit formula for the kernel $ M $ used to define the interaction energy
(\ref{eq:interaction_energy_non_regular}).
\begin{lemma}
	\label{lem:symmetry_properties}
	The Fourier transform of the Green's function $  K $ of the operator $ \mathcal{L} ( u ) = - \divg \eltensor \diff u  $ can for $ \xi \in \R^{ 3 } \setminus \{ 0 \} $  be written as
	\begin{align*}
		\fourier K ( \xi ) & = \frac{1}{4 \pi^2} A^{ - 1 } ( \xi )
		\shortintertext{with}
		A_{ki}(\xi)  &= \sum_{l}   [ \eltensor (e_i \otimes \xi)]_{kl} 
		\xi_l.
	\end{align*}
	As a consequence, $ \fourier K $ is symmetric, positive-definite and $ 
	- 2 $-homogeneous. In the case of isotropic elasticity, where the 
	elasticity tensor $ \eltensor $ is given for some $ a > -1/3 $ by
	\begin{equation} 
		\label{eq:def_isotropic_elasticity_tensor}
		\eltensor A \coloneqq \frac{A+A^{ \top }}{2} + a \tr (A) E_{ 3 } ,
	\end{equation} 
	we have 
	\begin{equation}
		\label{eq:fourier_K_formula}
		\fourier K ( \xi )
		=
		\frac{ 1 }{2 \pi^2 \abs{ \xi }^{ 2 } }
		\left(
		E_{ 3 }
		-
		\frac{ 1+2a}{ 2+2a }
		\frac{ \xi }{ \abs{ \xi } } \otimes \frac{ \xi }{ \abs{ \xi } } 
		\right)
		\quad
		\text{for all }
		\xi \in \R^{ 3 } \setminus \{0\}.
	\end{equation}
	The Fourier transform of the kernel $ M $ defined in equation (\ref{eq:formula_M}) is also symmetric and negative-semidefinite as a bilinear form on $ \R^{ 3 \times 3 } \times \R^{ 3 \times 3 } $ and satisfies $ \hat{ M } ( \xi ) = \hat{ M } ( -\xi) $ for all $ \xi \in \R^{ 3 } \setminus \{ 0 \} $. Moreover $ \hat{M} ( \xi ) F = 0 $ for all $ F \in \R^{ 3 \times 3 }_{ \mathrm{skew} } $ and all $ \xi \in \R^{ 3 } \setminus \{ 0 \} $.
\end{lemma}
\begin{proof}
	Rewriting the equation
	\begin{equation}
		\label{eq:elliptic_equation}
		- \divg \eltensor \diff u = f
	\end{equation}
	in Fourier space yields
	\begin{equation*} 
		4 \pi^2
		\sum_{ l = 1 }^{ 3 }  \xi_l  ( \eltensor (\mathcal F  u(\xi) \otimes \xi))_{kl} = (\mathcal F f_k)(\xi)
	\end{equation*}
	for all $ k \in \{ 1, 2 , 3 \} $.
	For $f = \delta_0 e_j$ we have $\mathcal F f =  e_j $. Moreover if $ u $ solves equation (\ref{eq:elliptic_equation}) for this choice of $ f $, then $ u_i =  K_{ij}$.  
	Hence
	\begin{align}
		\notag
		4 \pi^2
		\sum_{l}  \xi_l  \left( \eltensor  \sum_i \fourier K_{ij} (\xi) e_i  \otimes \xi \right)_{kl} &= \delta_{kj},
		\shortintertext{which is equal to}
		\label{eq:greens_fct}
		4\pi^2
		\sum_{l,i} \mathcal F K_{ij}(\xi) ( \eltensor(e_i \otimes \xi))_{kl} \xi_l & =  \delta_{kj}.
	\end{align}
	If we define the matrix-valued function $ A: \R^{ 3 } \to\R^{3 \times 3}$ by
	\begin{equation} 
		\label{eq:def_of_A}
		A_{ki}(\xi)  = \sum_{l}   [ \eltensor (e_i \otimes \xi)]_{kl} \xi_l,  
	\end{equation}
	then equation (\ref{eq:greens_fct}) is equivalent to
	\begin{equation*}
		\sum_i  \fourier K_{ij}(\xi)  A_{ki}(\xi) = 
		\frac{1}{4\pi^2}\delta_{kj}.
	\end{equation*}
	Assuming for now that $ A $ is invertible for every $ \xi \neq 0 $, we thus deduce 
	\begin{equation}  
		\label{eq:greens_function_as_inverse}   
		\fourier K(\xi) = 
		\frac{1}{4 \pi^2}A^{-1}(\xi) \text{ for }\xi \neq 0.
	\end{equation}
	In coordinate-free notation $A$ can be expressed through its action on $ a \in \R^{ 3 } $ via
	\begin{equation*}
		A(\xi) a = \eltensor (a \otimes \xi) \xi.
	\end{equation*}
	This implies that $ A $ is symmetric since $ \eltensor $ is symmetric and
	\begin{equation*}
		\inner*{ A(\xi) a }{b} = \inner*{ \eltensor (a \otimes \xi)}{  b \otimes \xi}.
	\end{equation*}
	Moreover
	\begin{equation*}
		\inner*{A a}{ a } 
		= 
		\inner*{ \eltensor ( a\otimes \xi)}{ a \otimes \xi } 
		\geq 
		c \abs{  \frac{ a \otimes \xi  + \xi \otimes a}{2} }^2 
		\gtrsim 
		\abs{ a }^{ 2 } 
		\abs{ \xi }^{ 2 }.
	\end{equation*}
	It follows that $A$ is positive definite, and hence invertible, for $\xi \ne 0$. 
	We thus deduce from (\ref{eq:greens_function_as_inverse}) that also $ \fourier K $ is symmetric and positive definite.
	
	We now want to study the kernel $ M $ defined by equation (\ref{eq:formula_M}). Its Fourier transform is given by
	\begin{equation}
		\label{eq:fourier_transform_of_M}
		(\fourier M_{iji'j'})(\xi) 
		= 
		- 4 \pi^2
		\sum_{k,l,k',l'} \eltensor_{ijkl}  (\fourier K_{kk'})(\xi) \xi_l \xi_{l'}  \eltensor_{k'l'i'j'}. 
	\end{equation}
	Thus, for $A, B \in \mathbb{R}^{3 \times 3}$ we get
	\begin{align}  
		\notag
		\inner*{\fourier M(\xi) A}{ B }
		& =  
		- 4 \pi^2
		\sum_{k,l,k',l'} (\eltensor A)_{k'l'} \xi_{l'} (\fourier K_{kk'})(\xi) (\eltensor B)_{kl}  \xi_l  
		\\
		\label{eq:fourer_M_on_matrices}
		& =  
		- 4 \pi^2 \inner*{ \fourier K(\xi) a }{ b }    
	\end{align}
	with
	\begin{equation*} 
		a_{k'}  =  \sum_{l'}  (\eltensor A)_{k'l'}  \xi_{l'}, \quad  b_k = \sum_l   ( \eltensor B)_{kl} \xi_l.
	\end{equation*}
	Since a bilinear form is determined by its actions on matrices and $ \eltensor $ is zero on skew-symmetric matrices, we see that
	$$ \fourier M(\xi)  A = \fourier M(\xi) \,  \mathrm{sym} \, A$$
	and that $\fourier M(\xi)$ is a negative-semidefinite bilinear  form on $\R^{3 \times 3}$ for all $ \xi \in \R^{ 3 } \setminus \{ 0 \}$. The symmetry of $ \fourier M $ follows immediately from the symmetry of $ \fourier K $.
	
	Lastly assume that $ \eltensor $ is isotropic in the sense of (\ref{eq:def_isotropic_elasticity_tensor}). We recall the definition of $ A $ from equation (\ref{eq:def_of_A}) to compute
	\begin{align*}
		A_{ k i } ( \xi )
		&=
		\frac{ \abs{ \xi }^{ 2 } }{ 2 } \left(
		E_3 + \frac{\xi}{\abs{ \xi } } \otimes \frac{\xi}{\abs{ \xi } } 
		\right)_{ k i }
		+
		\sum_{ l }
		\left(
		a \tr ( e_{ i } \otimes \xi ) E_3
		\right)_{ k l }
		\xi_{ l }
		\\
		& =
		\frac{ \abs{ \xi }^{ 2 } }{ 2 } \left(
		E_3 + \frac{\xi}{\abs{ \xi } } \otimes \frac{\xi}{\abs{ \xi } } 
		\right)_{ k i }
		+
		a \xi_{ i } \xi_{ k },
	\end{align*}
	thus
	\begin{equation*}
		A ( \xi ) = \frac{ \abs{ \xi }^{ 2 } }{ 2 } \left(
		E_3 + (1+2a) \frac{\xi}{\abs{ \xi } } \otimes \frac{\xi}{\abs{ \xi } } 
		\right).
	\end{equation*}
	This matrix is invertible since $ a \geq -1/3 $ and its inverse is
	\begin{equation}
		\label{eq:A_inverse_isotropic}
		A^{ - 1 } ( \xi )
		=
		\frac{ 2 }{ \abs{ \xi }^{ 2 } }
		\left(
		E_{ 3 }
		-
		\frac{ 1+2a}{ 2+2a }
		\frac{ \xi }{ \abs{ \xi } } \otimes \frac{ \xi }{ \abs{ \xi } } 
		\right),
	\end{equation}
	which finishes the proof.
\end{proof}

\section{Main Result}
\label{sct:main_result}
We want to place dislocation loops on a \emph{Bravais lattice} $ \lattice_{ 1 } $. For 3 linearly independent vectors $ v_{1}, v_{ 2 }, v_{ 3 } \in \R^{ 3 } $, this lattice is given by
\begin{align}
	\notag
	\lattice_{ 1 }
	& \coloneqq
	\left\{ \sum_{ i = 1 }^{ 3 } \nu^{ i } v_{ i } \, \colon \, \nu_{ i } \in \Z \text{ for all } i\in \{1,2,3\}\right\}
	\shortintertext{with unit cell}
	\label{eq:unit_cell}
	U  &\coloneqq
	\left\{ \sum_{ i = 1 }^{ 3 } \alpha^{ i } v_{ i } \, \colon \, \alpha_{ i } \in [0,1] \text{ for all } i \in \{1,2,3\}
	\right\}.
\end{align}
For simplicity, we will assume that the $ v_{i } $ are chosen such that $ \lm^{ 3 } ( U ) = 1 $. 
The corresponding lattice of size $ \rho > 0 $ is then defined by the rescaling $ \lattice_{ \rho } \coloneqq \rho \lattice_{ 1 } $. 
We want to emphasize that the lattice $ \lattice_{ \rho } $ (which we will sometimes refer to as the \emph{homogenization lattice}) is a different lattice than the atomic lattice, see also \Cref{fig:homogenization}.

Our material is confined to $ \Omega \subseteq \R^{ 3 } $ open and bounded.
Around each lattice point $ x \in \lattice_{ \rho } \cap \Omega $ we place a $ 1 $-rectifiable oriented dislocation loop $ \gamma_{ \rho } ( x ) $ together with a Burgers vector $ b_{ \rho } ( x ) $.  If moreover $ x \in \lattice_{ \rho } \setminus \Omega $, then no dislocation loop is placed at this point.
Let $S_{ \rho } ( x ) $ be a corresponding \emph{slip surface}, which is a  $2$-rectifiable oriented surface with  $ \mathrm{diam} ( S_{ \rho } ( x ) ) \leq r $ (up to nullsets) and $ \partial S_{ \rho } ( x ) = \gamma_{ \rho } ( x ) $. We have illustrated our setting in \Cref{fig:homogenization}.

Later on, we will also allow generalized arrays of dislocation loops, which
may take the form 
\begin{equation*}
	\sum_{ i = 1 }^{ n }  \lambda_{ i } b_i \otimes \tau_i \hm^{1} 
	\llcorner_{ 
		\gamma_{ i } }.
\end{equation*}
Here $ \lambda_{ i } \in \R $ for $ i \in \{1, \ldots, n \} $. We note that 
this is not a single loop any more, and that the Burgers vector may not lie 
in the crystalline lattice due to the multiplication with $ \lambda $. All 
the statements below also apply to these generalized arrays of dislocations.
\begin{figure}
	\centering
	
	\tikzset{every picture/.style={line width=0.75pt}} 
	
	\begin{tikzpicture}[x=0.75pt,y=0.75pt,yscale=-0.8,xscale=0.85]
		
		\draw  [color={rgb, 255:red, 65; green, 117; blue, 5 }  ,draw opacity=1 ][fill={rgb, 255:red, 74; green, 144; blue, 226 }  ,fill opacity=1 ] (242.69,290.42) .. controls (252.63,268.16) and (246.66,288.83) .. (253.95,289.62) .. controls (261.24,290.42) and (260.89,292.43) .. (261.24,306.32) .. controls (261.6,320.21) and (251.3,325.4) .. (242.02,315.06) .. controls (232.74,304.73) and (232.74,312.68) .. (242.69,290.42) -- cycle ;
		\draw  [color={rgb, 255:red, 65; green, 117; blue, 5 }  ,draw opacity=1 ][fill={rgb, 255:red, 74; green, 144; blue, 226 }  ,fill opacity=1 ] (165.14,280.88) .. controls (144.6,299.96) and (178.4,276.11) .. (159.18,294.39) .. controls (139.96,312.68) and (160.51,297.57) .. (165.81,314.27) .. controls (171.11,330.96) and (190.99,329.37) .. (185.03,317.45) .. controls (179.06,305.52) and (185.69,261.8) .. (165.14,280.88) -- cycle ;
		\draw  [color={rgb, 255:red, 65; green, 117; blue, 5 }  ,draw opacity=1 ][fill={rgb, 255:red, 74; green, 144; blue, 226 }  ,fill opacity=1 ] (75.01,292.01) .. controls (68.38,265.78) and (75.67,294.39) .. (94.23,289.62) .. controls (112.79,284.85) and (97.55,283.26) .. (94.89,305.52) .. controls (92.24,327.78) and (77,320.63) .. (71.7,308.7) .. controls (66.4,296.78) and (81.64,318.24) .. (75.01,292.01) -- cycle ;
		\draw  [color={rgb, 255:red, 65; green, 117; blue, 5 }  ,draw opacity=1 ][fill={rgb, 255:red, 74; green, 144; blue, 226 }  ,fill opacity=1 ] (75.01,199) .. controls (86.94,190.25) and (94.89,186.28) .. (96.88,204.56) .. controls (98.87,222.85) and (92.91,214.1) .. (90.26,217.28) .. controls (87.6,220.46) and (87.6,211.72) .. (75.67,213.31) .. controls (63.75,214.9) and (63.08,207.74) .. (75.01,199) -- cycle ;
		\draw  [color={rgb, 255:red, 65; green, 117; blue, 5 }  ,draw opacity=1 ][fill={rgb, 255:red, 74; green, 144; blue, 226 }  ,fill opacity=1 ] (148.58,195.82) .. controls (166.47,186.28) and (149.24,167.2) .. (177.07,188.66) .. controls (204.91,210.13) and (182.38,199.79) .. (189.67,210.13) .. controls (196.96,220.46) and (185.34,232.81) .. (173.1,217.28) .. controls (160.86,201.75) and (130.68,205.36) .. (148.58,195.82) -- cycle ;
		\draw  [color={rgb, 255:red, 65; green, 117; blue, 5 }  ,draw opacity=1 ][fill={rgb, 255:red, 74; green, 144; blue, 226 }  ,fill opacity=1 ] (241.36,86.11) .. controls (263.23,94.85) and (261.9,84.52) .. (265.22,94.85) .. controls (268.53,105.19) and (287.75,125.86) .. (262.57,120.29) .. controls (237.38,114.73) and (251.3,157.66) .. (244.67,112.34) .. controls (238.05,67.03) and (219.49,77.36) .. (241.36,86.11) -- cycle ;
		\draw  [fill={rgb, 255:red, 0; green, 0; blue, 0 }  ,fill opacity=1 ] (80.46,301.58) .. controls (80.46,298.83) and (82.31,296.61) .. (84.6,296.61) .. controls (86.89,296.61) and (88.74,298.83) .. (88.74,301.58) .. controls (88.74,304.32) and (86.89,306.54) .. (84.6,306.54) .. controls (82.31,306.54) and (80.46,304.32) .. (80.46,301.58) -- cycle ;
		\draw  [fill={rgb, 255:red, 0; green, 0; blue, 0 }  ,fill opacity=1 ] (80.46,202.2) .. controls (80.46,199.46) and (82.31,197.23) .. (84.6,197.23) .. controls (86.89,197.23) and (88.74,199.46) .. (88.74,202.2) .. controls (88.74,204.95) and (86.89,207.17) .. (84.6,207.17) .. controls (82.31,207.17) and (80.46,204.95) .. (80.46,202.2) -- cycle ;
		\draw  [fill={rgb, 255:red, 0; green, 0; blue, 0 }  ,fill opacity=1 ] (163.3,202.2) .. controls (163.3,199.46) and (165.15,197.23) .. (167.44,197.23) .. controls (169.73,197.23) and (171.58,199.46) .. (171.58,202.2) .. controls (171.58,204.95) and (169.73,207.17) .. (167.44,207.17) .. controls (165.15,207.17) and (163.3,204.95) .. (163.3,202.2) -- cycle ;
		\draw  [fill={rgb, 255:red, 0; green, 0; blue, 0 }  ,fill opacity=1 ] (163.3,301.58) .. controls (163.3,298.83) and (165.15,296.61) .. (167.44,296.61) .. controls (169.73,296.61) and (171.58,298.83) .. (171.58,301.58) .. controls (171.58,304.32) and (169.73,306.54) .. (167.44,306.54) .. controls (165.15,306.54) and (163.3,304.32) .. (163.3,301.58) -- cycle ;
		\draw  [fill={rgb, 255:red, 0; green, 0; blue, 0 }  ,fill opacity=1 ] (246.14,301.58) .. controls (246.14,298.83) and (248,296.61) .. (250.28,296.61) .. controls (252.57,296.61) and (254.43,298.83) .. (254.43,301.58) .. controls (254.43,304.32) and (252.57,306.54) .. (250.28,306.54) .. controls (248,306.54) and (246.14,304.32) .. (246.14,301.58) -- cycle ;
		\draw  [color={rgb, 255:red, 65; green, 117; blue, 5 }  ,draw opacity=1 ][fill={rgb, 255:red, 74; green, 144; blue, 226 }  ,fill opacity=1 ] (77.66,88.49) .. controls (110.8,86.11) and (96.88,87.7) .. (102.18,103.6) .. controls (107.49,119.5) and (90.26,133.81) .. (80.98,117.11) .. controls (71.7,100.42) and (79.65,136.99) .. (75.01,122.68) .. controls (70.37,108.37) and (44.53,90.88) .. (77.66,88.49) -- cycle ;
		\draw  [color={rgb, 255:red, 65; green, 117; blue, 5 }  ,draw opacity=1 ][fill={rgb, 255:red, 74; green, 144; blue, 226 }  ,fill opacity=1 ] (160.51,93.26) .. controls (190.33,75.77) and (176.41,86.11) .. (183.7,103.6) .. controls (190.99,121.09) and (168.46,83.72) .. (181.71,113.93) .. controls (194.97,144.14) and (159.18,100.42) .. (163.82,120.29) .. controls (168.46,140.17) and (130.68,110.75) .. (160.51,93.26) -- cycle ;
		\draw  [fill={rgb, 255:red, 0; green, 0; blue, 0 }  ,fill opacity=1 ] (80.46,102.83) .. controls (80.46,100.08) and (82.31,97.86) .. (84.6,97.86) .. controls (86.89,97.86) and (88.74,100.08) .. (88.74,102.83) .. controls (88.74,105.57) and (86.89,107.8) .. (84.6,107.8) .. controls (82.31,107.8) and (80.46,105.57) .. (80.46,102.83) -- cycle ;
		\draw  [fill={rgb, 255:red, 0; green, 0; blue, 0 }  ,fill opacity=1 ] (163.3,102.83) .. controls (163.3,100.08) and (165.15,97.86) .. (167.44,97.86) .. controls (169.73,97.86) and (171.58,100.08) .. (171.58,102.83) .. controls (171.58,105.57) and (169.73,107.8) .. (167.44,107.8) .. controls (165.15,107.8) and (163.3,105.57) .. (163.3,102.83) -- cycle ;
		\draw  [fill={rgb, 255:red, 0; green, 0; blue, 0 }  ,fill opacity=1 ] (246.14,102.83) .. controls (246.14,100.08) and (248,97.86) .. (250.28,97.86) .. controls (252.57,97.86) and (254.43,100.08) .. (254.43,102.83) .. controls (254.43,105.57) and (252.57,107.8) .. (250.28,107.8) .. controls (248,107.8) and (246.14,105.57) .. (246.14,102.83) -- cycle ;
		\draw   (60.74,136.6) -- (71.1,130.65) -- (71.1,133.63) -- (91.82,133.63) -- (91.82,130.65) -- (102.18,136.6) -- (91.82,142.55) -- (91.82,139.58) -- (71.1,139.58) -- (71.1,142.55) -- cycle ;
		\draw    (280.46,67.03) -- (264,83.96) ;
		\draw [shift={(261.9,86.11)}, rotate = 314.2] [fill={rgb, 255:red, 0; green, 0; blue, 0 }  ][line width=0.08]  [draw opacity=0] (8.93,-4.29) -- (0,0) -- (8.93,4.29) -- cycle    ;
		\draw    (147.91,230) -- (172.63,212.64) ;
		\draw [shift={(175.09,210.92)}, rotate = 144.92] [fill={rgb, 255:red, 0; green, 0; blue, 0 }  ][line width=0.08]  [draw opacity=0] (8.93,-4.29) -- (0,0) -- (8.93,4.29) -- cycle    ;
		\draw  [color={rgb, 255:red, 65; green, 117; blue, 5 }  ,draw opacity=1 ][fill={rgb, 255:red, 74; green, 144; blue, 226 }  ,fill opacity=1 ] (246,190.25) .. controls (251.3,182.3) and (254.61,183.1) .. (261.24,193.43) .. controls (267.87,203.77) and (267.21,202.18) .. (264.56,217.28) .. controls (261.9,232.39) and (258.59,239.54) .. (242.02,214.1) .. controls (225.45,188.66) and (240.7,198.2) .. (246,190.25) -- cycle ;
		\draw  [fill={rgb, 255:red, 0; green, 0; blue, 0 }  ,fill opacity=1 ] (246.14,202.2) .. controls (246.14,199.46) and (248,197.23) .. (250.28,197.23) .. controls (252.57,197.23) and (254.43,199.46) .. (254.43,202.2) .. controls (254.43,204.95) and (252.57,207.17) .. (250.28,207.17) .. controls (248,207.17) and (246.14,204.95) .. (246.14,202.2) -- cycle ;
		\draw   (167.44,301.58) -- (188.15,295.12) -- (188.15,298.35) -- (229.57,298.35) -- (229.57,295.12) -- (250.28,301.58) -- (229.57,308.03) -- (229.57,304.8) -- (188.15,304.8) -- (188.15,308.03) -- cycle ;
		\draw   (260.02,201.83) .. controls (260.02,197.64) and (262.68,194.24) .. (265.96,194.24) .. controls (269.23,194.24) and (271.89,197.64) .. (271.89,201.83) .. controls (271.89,206.02) and (269.23,209.42) .. (265.96,209.42) .. controls (262.68,209.42) and (260.02,206.02) .. (260.02,201.83)(257.67,201.83) .. controls (257.67,196.34) and (261.38,191.89) .. (265.96,191.89) .. controls (270.53,191.89) and (274.24,196.34) .. (274.24,201.83) .. controls (274.24,207.32) and (270.53,211.77) .. (265.96,211.77) .. controls (261.38,211.77) and (257.67,207.32) .. (257.67,201.83) ;
		\draw    (274.24,201.83) -- (360.5,201.79) ;
		\draw   (416.65,247.09) .. controls (416.65,245.81) and (417.49,244.77) .. (418.52,244.77) .. controls (419.55,244.77) and (420.39,245.81) .. (420.39,247.09) .. controls (420.39,248.37) and (419.55,249.41) .. (418.52,249.41) .. controls (417.49,249.41) and (416.65,248.37) .. (416.65,247.09) -- cycle ;
		\draw   (540.64,153.96) .. controls (540.64,152.68) and (541.48,151.64) .. (542.51,151.64) .. controls (543.54,151.64) and (544.38,152.68) .. (544.38,153.96) .. controls (544.38,155.25) and (543.54,156.29) .. (542.51,156.29) .. controls (541.48,156.29) and (540.64,155.25) .. (540.64,153.96) -- cycle ;
		\draw   (532.04,200.92) .. controls (532.04,199.63) and (532.88,198.59) .. (533.91,198.59) .. controls (534.95,198.59) and (535.78,199.63) .. (535.78,200.92) .. controls (535.78,202.2) and (534.95,203.24) .. (533.91,203.24) .. controls (532.88,203.24) and (532.04,202.2) .. (532.04,200.92) -- cycle ;
		\draw   (506.63,153.96) .. controls (506.63,152.68) and (507.47,151.64) .. (508.5,151.64) .. controls (509.53,151.64) and (510.37,152.68) .. (510.37,153.96) .. controls (510.37,155.25) and (509.53,156.29) .. (508.5,156.29) .. controls (507.47,156.29) and (506.63,155.25) .. (506.63,153.96) -- cycle ;
		\draw   (493.92,200.92) .. controls (493.92,199.63) and (494.76,198.59) .. (495.79,198.59) .. controls (496.82,198.59) and (497.66,199.63) .. (497.66,200.92) .. controls (497.66,202.2) and (496.82,203.24) .. (495.79,203.24) .. controls (494.76,203.24) and (493.92,202.2) .. (493.92,200.92) -- cycle ;
		\draw   (450.57,200.92) .. controls (450.57,199.63) and (451.4,198.59) .. (452.44,198.59) .. controls (453.47,198.59) and (454.3,199.63) .. (454.3,200.92) .. controls (454.3,202.2) and (453.47,203.24) .. (452.44,203.24) .. controls (451.4,203.24) and (450.57,202.2) .. (450.57,200.92) -- cycle ;
		\draw   (527.93,247.41) .. controls (527.93,246.13) and (528.77,245.08) .. (529.8,245.08) .. controls (530.83,245.08) and (531.67,246.13) .. (531.67,247.41) .. controls (531.67,248.69) and (530.83,249.73) .. (529.8,249.73) .. controls (528.77,249.73) and (527.93,248.69) .. (527.93,247.41) -- cycle ;
		\draw   (490.56,247.41) .. controls (490.56,246.13) and (491.39,245.08) .. (492.43,245.08) .. controls (493.46,245.08) and (494.3,246.13) .. (494.3,247.41) .. controls (494.3,248.69) and (493.46,249.73) .. (492.43,249.73) .. controls (491.39,249.73) and (490.56,248.69) .. (490.56,247.41) -- cycle ;
		\draw   (405.34,153.5) .. controls (405.34,152.21) and (406.18,151.17) .. (407.21,151.17) .. controls (408.24,151.17) and (409.08,152.21) .. (409.08,153.5) .. controls (409.08,154.78) and (408.24,155.82) .. (407.21,155.82) .. controls (406.18,155.82) and (405.34,154.78) .. (405.34,153.5) -- cycle ;
		\draw   (471.87,154.43) .. controls (471.87,153.14) and (472.71,152.1) .. (473.74,152.1) .. controls (474.77,152.1) and (475.61,153.14) .. (475.61,154.43) .. controls (475.61,155.71) and (474.77,156.75) .. (473.74,156.75) .. controls (472.71,156.75) and (471.87,155.71) .. (471.87,154.43) -- cycle ;
		\draw   (412.07,201.38) .. controls (412.07,200.1) and (412.91,199.06) .. (413.94,199.06) .. controls (414.97,199.06) and (415.81,200.1) .. (415.81,201.38) .. controls (415.81,202.67) and (414.97,203.71) .. (413.94,203.71) .. controls (412.91,203.71) and (412.07,202.67) .. (412.07,201.38) -- cycle ;
		\draw   (441.22,154.89) .. controls (441.22,153.61) and (442.06,152.57) .. (443.09,152.57) .. controls (444.12,152.57) and (444.96,153.61) .. (444.96,154.89) .. controls (444.96,156.18) and (444.12,157.22) .. (443.09,157.22) .. controls (442.06,157.22) and (441.22,156.18) .. (441.22,154.89) -- cycle ;
		\draw   (453.56,246.94) .. controls (453.56,245.66) and (454.39,244.62) .. (455.43,244.62) .. controls (456.46,244.62) and (457.29,245.66) .. (457.29,246.94) .. controls (457.29,248.23) and (456.46,249.27) .. (455.43,249.27) .. controls (454.39,249.27) and (453.56,248.23) .. (453.56,246.94) -- cycle ;
		\draw   (434.87,108.4) .. controls (434.87,107.12) and (435.71,106.08) .. (436.74,106.08) .. controls (437.77,106.08) and (438.61,107.12) .. (438.61,108.4) .. controls (438.61,109.69) and (437.77,110.73) .. (436.74,110.73) .. controls (435.71,110.73) and (434.87,109.69) .. (434.87,108.4) -- cycle ;
		\draw   (472.62,107.47) .. controls (472.62,106.19) and (473.46,105.15) .. (474.49,105.15) .. controls (475.52,105.15) and (476.36,106.19) .. (476.36,107.47) .. controls (476.36,108.76) and (475.52,109.8) .. (474.49,109.8) .. controls (473.46,109.8) and (472.62,108.76) .. (472.62,107.47) -- cycle ;
		\draw   (546.62,107.94) .. controls (546.62,106.65) and (547.46,105.61) .. (548.49,105.61) .. controls (549.52,105.61) and (550.36,106.65) .. (550.36,107.94) .. controls (550.36,109.22) and (549.52,110.26) .. (548.49,110.26) .. controls (547.46,110.26) and (546.62,109.22) .. (546.62,107.94) -- cycle ;
		\draw   (397.5,108.4) .. controls (397.5,107.12) and (398.33,106.08) .. (399.36,106.08) .. controls (400.4,106.08) and (401.23,107.12) .. (401.23,108.4) .. controls (401.23,109.69) and (400.4,110.73) .. (399.36,110.73) .. controls (398.33,110.73) and (397.5,109.69) .. (397.5,108.4) -- cycle ;
		\draw   (509.62,108.4) .. controls (509.62,107.12) and (510.46,106.08) .. (511.49,106.08) .. controls (512.52,106.08) and (513.36,107.12) .. (513.36,108.4) .. controls (513.36,109.69) and (512.52,110.73) .. (511.49,110.73) .. controls (510.46,110.73) and (509.62,109.69) .. (509.62,108.4) -- cycle ;
		\draw    (399.36,110.73) -- (407.21,151.17) ;
		\draw    (407.21,155.82) -- (413.94,199.06) ;
		\draw    (413.94,203.71) -- (418.52,244.77) ;
		\draw    (420.39,247.09) -- (453.56,246.94) ;
		\draw    (457.29,246.94) -- (490.56,247.41) ;
		\draw    (494.3,247.41) -- (527.93,247.41) ;
		\draw    (401.23,108.4) -- (434.87,108.4) ;
		\draw    (438.61,108.4) -- (472.62,107.47) ;
		\draw    (476.36,107.47) -- (509.62,108.4) ;
		\draw    (513.36,108.4) -- (546.62,107.94) ;
		\draw    (409.08,153.5) -- (441.22,154.89) ;
		\draw    (436.74,110.73) -- (443.09,152.57) ;
		\draw    (444.96,154.89) -- (471.87,154.43) ;
		\draw    (475.61,154.43) -- (506.63,153.96) ;
		\draw    (510.37,153.96) -- (540.64,153.96) ;
		\draw    (474.49,109.8) -- (473.74,152.1) ;
		\draw    (511.49,110.73) -- (508.5,151.64) ;
		\draw    (548.49,110.26) -- (542.51,151.64) ;
		\draw    (415.81,201.38) -- (450.57,200.92) ;
		\draw    (443.09,157.22) -- (452.44,198.59) ;
		\draw    (454.3,200.92) -- (493.92,200.92) ;
		\draw    (497.66,200.92) -- (532.04,200.92) ;
		\draw    (508.5,156.29) -- (495.79,198.59) ;
		\draw    (542.51,156.29) -- (533.91,198.59) ;
		\draw    (452.44,203.24) -- (455.43,244.62) ;
		\draw    (495.79,203.24) -- (492.43,245.08) ;
		\draw    (533.91,203.24) -- (529.8,245.08) ;
		\draw   (420.7,255.37) -- (428.8,251.34) -- (428.8,253.35) -- (445.01,253.35) -- (445.01,251.34) -- (453.11,255.37) -- (445.01,259.41) -- (445.01,257.39) -- (428.8,257.39) -- (428.8,259.41) -- cycle ;
		\draw   (370.52,184.98) .. controls (370.52,113.84) and (417.83,56.17) .. (476.19,56.17) .. controls (534.55,56.17) and (581.86,113.84) .. (581.86,184.98) .. controls (581.86,256.12) and (534.55,313.8) .. (476.19,313.8) .. controls (417.83,313.8) and (370.52,256.12) .. (370.52,184.98)(360.21,184.98) .. controls (360.21,108.15) and (412.13,45.86) .. (476.19,45.86) .. controls (540.24,45.86) and (592.17,108.15) .. (592.17,184.98) .. controls (592.17,261.82) and (540.24,324.11) .. (476.19,324.11) .. controls (412.13,324.11) and (360.21,261.82) .. (360.21,184.98) ;
		\draw [color={rgb, 255:red, 65; green, 117; blue, 5 }  ,draw opacity=1 ]   (438.71,234.92) .. controls (449.12,241.73) and (491.73,140.09) .. (492.2,125.33) ;
		\draw [color={rgb, 255:red, 208; green, 2; blue, 27 }  ,draw opacity=1 ]   (469.95,131) -- (431.24,130.46) ;
		\draw [shift={(429.24,130.44)}, rotate = 0.8] [color={rgb, 255:red, 208; green, 2; blue, 27 }  ,draw opacity=1 ][line width=0.75]    (10.93,-3.29) .. controls (6.95,-1.4) and (3.31,-0.3) .. (0,0) .. controls (3.31,0.3) and (6.95,1.4) .. (10.93,3.29)   ;
		
		\draw (203.86,306.71) node [anchor=north west][inner sep=0.75pt]  [font=\footnotesize]  {$\rho $};
		\draw (63.54,149.51) node [anchor=north west][inner sep=0.75pt]  [font=\footnotesize,rotate=-359.44]  {$r\ \ll \ \rho $};
		\draw (276.7,36.21) node [anchor=north west][inner sep=0.75pt]  [font=\footnotesize,color={rgb, 255:red, 65; green, 117; blue, 5 }  ,opacity=1 ]  {$\gamma _{\rho }( x)$};
		\draw (235.01,100.41) node [anchor=north west][inner sep=0.75pt]  [font=\footnotesize]  {$x$};
		\draw (132.89,234.95) node [anchor=north west][inner sep=0.75pt]  [font=\footnotesize,color={rgb, 255:red, 74; green, 144; blue, 226 }  ,opacity=1 ]  {$S_{\rho }( y)$};
		\draw (178.01,164) node [anchor=north west][inner sep=0.75pt]    {$y$};
		\draw (420.42,269.11) node [anchor=north west][inner sep=0.75pt]  [font=\tiny]  {$\varepsilon \ \ll r$};
		\draw (94.98,16.88) node [anchor=north west][inner sep=0.75pt]   [align=left] {Homogenization lattice};
		\draw (431.88,10.92) node [anchor=north west][inner sep=0.75pt]   [align=left] {Atomic lattice};
		\draw (416.56,121.63) node [anchor=north west][inner sep=0.75pt]  [color={rgb, 255:red, 208; green, 2; blue, 27 }  ,opacity=1 ]  {$b$};

	\end{tikzpicture}
	\caption{Dislocation loops $ \gamma_{ \rho } ( x ) $ are placed on the 
		homogenization lattice $ \lattice_{ \rho } $ with corresponding slip 
		surfaces $ S_{ \rho } ( x ) $ such that $ \partial S_{ \rho } = 
		\gamma_{ \rho } $. Their diameter $ r $  is much smaller than their 
		distance $ \rho $, and both are on a larger scale than the atomic 
		lattice spacing $ \eps $, so that $ \eps \ll r \ll \rho $. }
	\label{fig:homogenization}
\end{figure}
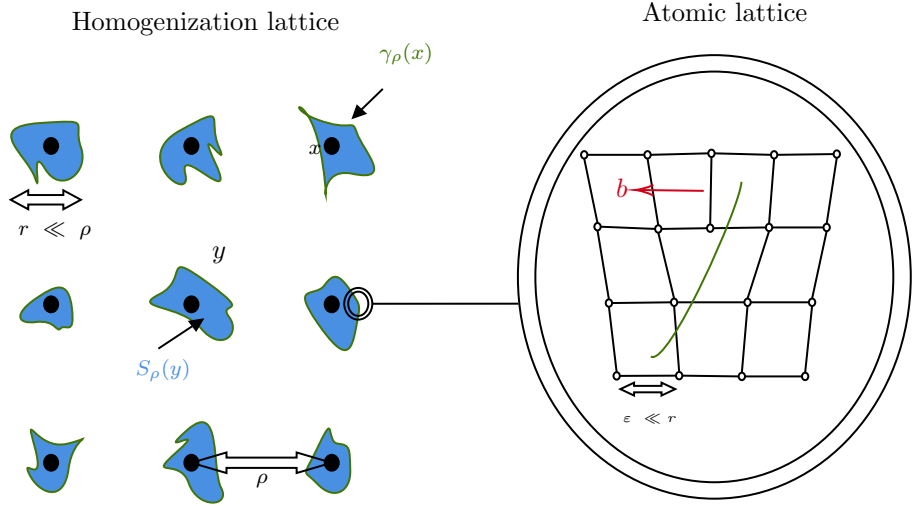

As we have seen in equation (\ref{eq:interaction_energy_approximation}), the oriented surface area function $ \of \colon \lattice_{ \rho } \to \R^{ 3 \times 3 } $ defined by
\begin{equation}
	\label{eq:def_or_area_function}
	\of_{ \rho } ( x ) 
	\coloneqq 
	\int_{ S_{ \rho } ( x ) }b_{ \rho } ( z ) \otimes n ( z ) \dd{ \hm^{ 2 } ( z ) }
\end{equation}
should dictate the leading order term in the interaction energy. Given any function $ f \colon \lattice_{ \rho } \to \R^{ 3 \times 3 } $ on the lattice, we define the corresponding piecewise constant extension as
\begin{align}
	\label{eq:def_extension_function}
	\widetilde{f} ( y ) & \coloneqq \frac{ 1 }{ \rho^{ 3 } } f ( x ) \quad \text{for all } y \in x + \rho U.	
\end{align}
Due to the unit cell (\ref{eq:unit_cell}) having volume 1, we thus have for each $ x \in \lattice_{\rho } $ that
\begin{equation*}
	\of_{ \rho } ( x ) 
	=
	\int_{ x + \rho U }
	\widetilde{\of_{ \rho }} ( y )
	\dd{ y }.
\end{equation*}
We make two assumptions on the array of dislocations.
\begin{enumerate}[labelindent=0pt,labelwidth=\widthof{\ref{item:wellSeperated}},label=(A\arabic*),ref=(A\arabic*),itemindent=0em,leftmargin=!]
	\item \label{item:l2BoundOnArea}
	The dislocation density stays $ \lp^{ 2 } $-bounded in the sense that
	\begin{equation*}
		\sup_{ \rho > 0 }
		\sum_{ x \in \lattice_{ \rho } }
		\frac{ 1 }{ \rho^{ 3 } }
		\left( \int_{ S_{ \rho } ( x ) } \abs{ b ( z ) } \dd{ \hm^{ 2 } ( z ) }
		\right)^{ 2 } 
		< 
		\infty.
	\end{equation*}
	\item \label{item:wellSeperated}
	The dislocation loops are well-separated in the sense that as $ \rho \to 0 $,
	\begin{equation*}
		r(\rho) \coloneqq 
		\esssup_{ x \in \lattice_{ \rho } } \esssup_{ y \in S_{ \rho } ( x ) } \abs{ x - y }
		\in o ( \rho ).
	\end{equation*}
\end{enumerate} 
We comment on these assumptions in \Cref{rmk:assumptions}. 

The sequence of arrays of dislocation loops will be denoted by
\begin{equation}
	\label{eq:disdens_def}
	\disdens_{ \rho } ( x )
	\coloneqq
	\left( b_{ \rho } ( x ) \otimes \tau \hm^{ 1 } \llcorner_{ \gamma_{ \rho } ( x ) } \right)_{ x \in \lattice_{ \rho } }.
\end{equation}
The first step in analyzing $ \lim_{ \rho \to 0 } \interactionEnergy ( 
\disdens_{ \rho }) $ is the following Proposition which yields an 
asymptotic representation of the total interaction energy via a lattice sum 
over the oriented surface areas.
\begin{proposition}
	\label{prop:energy_asymptotics}
	Let $ \disdens_{ \rho } = \left( b_{ \rho } ( x ) \otimes \tau \hm^{ 1 } \llcorner_{ \gamma_{ \rho } ( x ) } \right)_{ x \in \lattice_{ \rho } } $ be a sequence of arrays of dislocation loops satisfying assumptions \ref{item:l2BoundOnArea} and \ref{item:wellSeperated}. 
	Then there exists a non-relabelled subsequence of $ \rho \to 0 $ such 
	that
	\begin{equation}
		\label{eq:asymptotic_representations_of_energy}
		\lim_{ \rho \to 0 }
		\interactionEnergy(\disdens_{ \rho })
		=
		\lim_{ \rho \to 0 }
		\frac{ 1 }{ 2 }
		\sum_{ \substack{x, y \in \lattice_{ \rho } \\ x \neq y } }
		\of_{ \rho } ( x ) 
		\colon 
		M ( x - y ) \of_{ \rho } ( y ).
	\end{equation}
\end{proposition}

In order to compute the limiting energy $ \lim_{ \rho \to 0 } \interactionEnergy ( \disdens_{ \rho } ) $, we employ different strategies depending on the type of oscillations of the sequence $ \disdens_{ \rho } $, respectively the oriented surface area function $ \of_{ \rho } $. We distinguish two different types of oscillation of $ \of_{ \rho } $. 
The oscillations of weak-long type encode that the oscillations of the oriented surface area happen on a scale which is larger than the scale $ \rho $ of the homogenization lattice. An example of this is the sequence 
\begin{equation}
	\label{eq:example_weak_long}
	\widetilde{ \of_{ \rho} } ( x ) =
	\sin ( \rho^{-1/2} x_1) \psi ( x ) \quad \text{for a suitable test function } \psi. 
\end{equation}
The definition has been used before in the context of micromagnetics.
\begin{definition}[\cite{james_mueller_internal_variables_and_fine_scale_oscillations_in_micromagnetics}]
	\label{def:weak_long}
	We say that the sequence $ (\disdens_{ \rho } )_{ \rho > 0 } $ satisfies the weak-long oscillation condition if the corresponding oriented surface area function $ \of_{ \rho } $ satisfies
	\begin{equation}
		\label{eq:weak_long}
		\limsup_{ \rho \to 0 }
		\sup_{ \abs{ \xi } \leq \rho }
		\norm{ \widetilde{\of_{ \rho }} ( \xi + \cdot ) - \widetilde{ \of_{ \rho } } ( \cdot ) }_{ \lp^{ 2 } ( \R^{ 3 } ) } = 0.
	\end{equation}
\end{definition}
We first analyze the limit of the energies in the weak-long case. 
For the definition of $ \Hm $-measures, see \Cref{sct:H_measures}.
\begin{proposition}
	\label{prop:weak_long_energy_convergence}
	Let $ \disdens_{ \rho } = \left( b_{ \rho } ( x ) \otimes \tau \hm^{ 1 } \llcorner_{ \gamma_{ \rho } ( x ) } \right)_{ x \in \lattice_{ \rho } }$ be a sequence of arrays of dislocation loops which satisfies assumptions \ref{item:l2BoundOnArea},  \ref{item:wellSeperated} and only has long-range oscillations (\ref{eq:weak_long}).
	Then there exists a non-relabelled subsequence of $ \rho \to 0 $, a function $ \of \in \lp^{ 2 } \left( \R^{ 3} ; \R^{ 3 \times 3 } \right) $ and an $ \Hm $-measure $ \mu_{ \Hm } $ generated by $ (\widetilde{ \of_{ \rho } } - \of )_{ \rho } $ such that $ \widetilde{ \of_{ \rho } } \rightharpoonup \of $ in $ \lp^{ 2 } ( \R^{ 3 } ; \R^{ 3 \times 3}) $ as $ \rho \to 0 $ and
	\begin{equation}
		\label{eq:main_thm}
		\lim_{ \rho \to 0 }
		\interactionEnergy(\disdens_{ \rho })
		=
		\frac{ 1 }{ 2 } 
		\left( \inner*{ \fourier  \of }{\Psi \fourier  \of }_{ \lp^{2 } ( \R^{ 3 } ) } 
		+ 
		\int_{ \R^{ 3 } \times \Sph^{2 } } \Psi ( \nu ) \cdot \dd{ \mu_{ \Hm } ( x , \nu ) } 
		\right).
	\end{equation}
	Here $ \Psi \colon \R^{ 3 } \setminus \{ 0 \}  \to \R^{ (3 \times 3 ) \times ( 3 \times 3 )} $  is a smooth, $ 0 $-homogeneous function which only depends on the lattice $ \lattice_{ 1 } $ and the elastic tensor $ \eltensor $, and $ \fourier $ denotes the Fourier transform.
\end{proposition}
The proof of \Cref{prop:weak_long_energy_convergence} is the content of \Cref{sct:weak_long}.

The weak-short counterpart encodes that oscillations occur on the size of 
the lattice, and that the local average is zero, see equation 
(\ref{eq:weak_short_def}). An example of this is the sequence $ 
\widetilde{\of_{ \rho }} ( x ) = \sin ( \rho^{ - 1} x_1 ) \psi ( x ) $, for 
a suitable test function $ \psi $. The reason why we only consider short 
oscillations when they cancel each other out is that we will later consider 
an additive decomposition $ \of_{ \rho } ( x ) = \of_{ \rho }^{ \mathrm{S} 
} ( x ) + \of_{ \rho }^{ \mathrm{L} } $ into a weak-short and weak-long 
oscillating part, and the non-cancelling terms will be automatically 
absorbed into $ \of_{ \rho }^{ \mathrm{L} } $.

For the formal definition we fix some radially symmetric function $ \phi \in \ccinf ( \R^{ 3 } ;[0,1]) $ which is radially decreasing, has support in $ B_{ 1 } ( 0 ) $ and satisfies $ \phi ( 0 ) = 1 $. We define $ \Phi \coloneqq \fourier^{-1 } \phi $ and its dilation by $ \Phi_{ \eps } ( x ) \coloneqq \eps^{ - 3 } \Phi ( x/\eps ) $. 

\begin{definition}[\cite{firoozye_93_homogenization_on_lattices}]
	\label{def:weak_short}
	We say that the sequence $ ( \disdens_{ \rho })_{ \rho > 0 }$ satisfies the weak-short oscillation condition if the corresponding oriented surface area function $ \of_{ \rho } $ satisfies
	\begin{equation}
		\label{eq:weak_short_def}
		\lim_{ \eps \to 0 }
		\lim_{ \rho \to 0 }
		\norm{ \Phi_{ \rho/\eps } \ast \widetilde{\of_{\rho}} }_{ \lp^{ 2 } ( \R^{ 3} ) }
		=
		0.
	\end{equation} 
\end{definition}
Assuming that the sequence $ \disdens_{ \rho } $ has only oscillations of weak-short type, we arrive at the following convergence result. Here $ U^{ \ast } $ denotes the unit cell of the reciprocal lattice $ \lattice_{ 1 }^{ \ast } $, and $ \idfs{ M } $ denotes the inverse discrete Fourier series $ \sum_{ x \in \lattice_{ 1 } \setminus \{0\}} M ( x ) \exp ( -2\pi i x \cdot \xi ) $. The \emph{Wigner measure} $ \mu_{ \mathrm{W}} $ associated to $ \of_{ \rho } $ is the weak-limit of the measures $ \mu_{ \rho } \coloneqq \rho^{ -3 }\idfs{ \of_{ \rho } ( \rho \cdot )} \otimes \overline{ \idfs{ \of_{ \rho } ( \rho \cdot )}} \dd{ \lm^{ 3 } }$, see \Cref{sct:weak_short} for more in-depth definitions.
\begin{proposition}
	\label{prop:weak_short_energy_convergence}
	Let $ \disdens_{ \rho } =  \left( b_{ \rho } ( x ) \otimes \tau \hm^{ 1 } \llcorner_{ \gamma_{ \rho } ( x ) } \right)_{ x \in \lattice_{ \rho } } $ be a sequence of arrays of dislocation loops which satisfies assumptions \ref{item:l2BoundOnArea},  \ref{item:wellSeperated} and only has short-range oscillations (\ref{eq:weak_short_def}).
	Then there exists a non-relabelled subsequence of $ \rho \to 0 $ and a Wigner measure $ \mu_{ \mathrm{W}} $ generated by $ (\of_{ \rho })_{ \rho } $ such that
	\begin{equation}
		\lim_{ \rho \to 0 }
		\interactionEnergy(\disdens_{ \rho })
		=
		\frac{ 1 }{ 2 } \int_{ U^{ \ast } } \idfs{ M } ( \xi ) \cdot \dd{ \mu_{ \mathrm{W } } ( \xi ) }.
	\end{equation}
\end{proposition}
The proof of \Cref{prop:weak_short_energy_convergence} is the content of \Cref{sct:weak_short}.

By decomposing the sequence $ \disdens_{ \rho } $ into its weak-long and 
weak-short oscillating part and using 
\Cref{prop:weak_long_energy_convergence} and 
\Cref{prop:weak_short_energy_convergence}, we thus obtain our main result. 
Here we use the notation $ T_z ( a ) \coloneqq a + z $ for a translation by 
$ z $, and denote the corresponding push-forward by $ (T_z )_{ \#} $. 
Consequently, we have
\begin{equation*}
	(T_z)_\# ( b \otimes \tau \hm^{ 1 } \llcorner_{ \gamma } )
	=
	b \otimes \tau \hm^{ 1 } \llcorner_{ \gamma + z }.
\end{equation*}
\begin{theorem}
	\label{thm:main_theorem}
	Let $ \disdens_{ \rho } =  \left( b_{ \rho } ( x ) \otimes \tau \hm^{ 1 } \llcorner_{ \gamma_{ \rho } ( x ) } \right)_{ x \in \lattice_{ \rho } } $ be a sequence of arrays of dislocation loops satisfying assumptions \ref{item:l2BoundOnArea} and \ref{item:wellSeperated}. 
	Then there exists a non-relabelled subsequence of $ \rho \to 0 $ and 
	$ \rho \mapsto \delta ( \rho ) $ such that $ \rho \ll \delta ( \rho ) 
	\to 0 $ and such that we can decompose $ \disdens_{ \rho } $ into 
	generalized dislocation arrays
	\begin{align*}
		\disdens_{ \rho } ( x )
		& =
		\disdens_{ \rho }^{ \mathrm{L}} ( x ) 
		+
		\disdens_{ \rho }^{ \mathrm{S} } ( x ) 
		\\
		& \coloneqq
		\sum_{ y \in \lattice_{ \rho }}
		\rho^{ 3 } \Phi_{ \rho / \delta ( \rho ) } ( x - y )
		(T_{ x- y } )_{\#} \disdens_{ \rho } ( y )  
		+
		\left(\disdens_{ \rho } ( x ) - 
		\sum_{ y \in \lattice_{ \rho } }
		\rho^{ 3 } \Phi_{ \rho / \delta ( \rho ) } ( x - y ) 
		(T_{ x- y } )_{\#} \disdens_{ \rho } ( y )   \right).
	\end{align*} 
	Moreover, $\disdens_{ \rho }^{ \mathrm{ L } } $ satisfies the weak-long 
	condition (\ref{eq:weak_long}) and $ \disdens_{ \rho }^{ \mathrm{S}} $ 
	satisfies the weak-short condition (\ref{eq:weak_short_def}). Further 
	the mixed interaction energy vanishes in the sense that
	\begin{equation*}
		\lim_{ \rho \to 0}
		\sum_{ \substack{x, y \in \lattice_{ \rho } \\ x \neq y } }
		\interactionEnergy ( \disdens_{ \rho }^{ \mathrm{S} } ( x ) , \disdens_{ \rho }^{ \mathrm{ L }  } ( y ) )
		=
		0.
	\end{equation*}
	Finally, we have
	\begin{equation*}
		\lim_{ \rho \to 0 }
		\interactionEnergy ( \disdens_{ \rho } )
		=
		\frac{1}{2}
		\left( \inner*{ \fourier  \of }{\Psi \fourier  \of }_{ \lp^{2 } ( \R^{ 3 } ) } 
		+ 
		\int_{ \R^{ 3 } \times \Sph^{2 } } \Psi ( \nu ) \cdot \dd{ \mu_{ \Hm } ( x , \nu ) }
		+
		\int_{ U^{ \ast } } \idfs{ M } ( \xi ) \cdot \dd{ \mu_{ \mathrm{W } } ( \xi ) } 
		\right),
	\end{equation*}
	where $ \of $ and $ \mu_{ \Hm } $ are generated by $ \disdens_{ \rho }^{ \mathrm{ L }} $ as in \Cref{prop:weak_long_energy_convergence} and the Wigner measure $ \mu_{ \mathrm{W}} $ is generated by $ \disdens_{ \rho }^{ \mathrm{ S }} $ as in \Cref{prop:weak_short_energy_convergence}.
\end{theorem}
The proof of \Cref{thm:main_theorem} can be found in \Cref{sct:scale_separation}.

We close the paper by showing that we can actually generate negative interaction energy "out of nowhere" through oscillations if the lattice has cubic symmetry and the elastic tensor is isotropic. In the case of anisotropic elasticity or with different lattice symmetries, we still expect this to hold. However, our proof relies on explicit computations, which are much easier under these assumptions.
\begin{definition}
	\label{def:point_group_and_cubic_symmetry}
	Given a Bravais lattice $ \lattice_{ 1 } $, we define its \emph{point 
		group} $ P_1 $ as 
	\begin{equation*}
		P_1 \coloneqq \{ R \in \orth ( 3 ) \colon R \lattice_{ 1 } = \lattice_{ 1 } \}.
	\end{equation*}
	We say that $ \lattice_{ 1 } $ has \emph{cubic symmetry} if 
	\begin{equation}
		\label{eq:cubic_symmetry}
		\{ R \in \orth( 3 ) \colon R ( [-1,1]^3) = [-1,1]^3\}
		\subseteq P_1
	\end{equation}
\end{definition}

\begin{corollary}
	\label{cor:negative_interaction}
	Assume that $ \lattice_{ 1 } $ has cubic symmetry and that $ \eltensor $ is isotropic. Then we find a sequence of arrays of dislocation loops $ (\disdens_{ \rho })_{ \rho > 0 } $ which satisfies assumptions \ref{item:l2BoundOnArea} and \ref{item:wellSeperated}. Additionally, it satisfies
	\begin{align*}
		\liminf_{ \rho \to 0 }
		\interactionEnergy ( \disdens_{ \rho } )
		&<0
		\shortintertext{and}
		\widetilde{ \of_{ \rho } }
		&\rightharpoonup 0.
	\end{align*}
\end{corollary}
The proof of \Cref{cor:negative_interaction} is the content of \Cref{sct:relaxation}.

\begin{remark}
	\label{rmk:assumptions}
	We want to remark on the assumptions \ref{item:l2BoundOnArea} and \ref{item:wellSeperated}.
	
	The bound \ref{item:l2BoundOnArea} is a stronger version of the $ \lp^{ 2 } $-boundedness of $ \widetilde{ \of_{ \rho } } $ since
	\begin{equation}
		\label{eq:l2_bound_weaker}
		\sum_{ x \in \lattice_{ \rho } } \frac{1}{\rho^{ 3 } }
		\left(\int_{ S_{ \rho } ( x ) } \abs{ b_{ \rho } ( x' ) } \dd{ \hm^{ 2 } ( x'  ) } \right)^{ 2 }
		\geq
		\sum_{ x \in \lattice_{ \rho } }
		\frac{1}{\rho^{ 6 }} 
		\int_{ x + \rho U } 
		\rho^{ 6 } 
		\abs{ \widetilde{ \of_{ \rho} } ( y ) }^{ 2 }
		\dd{ y }
		=
		\int_{ \R^{ 3 } }
		\abs{ \widetilde{ \of_{ \rho} } ( y ) }^{ 2 }
		\dd{ y }.
	\end{equation}
	From a physical and variational perspective, it would be more natural to assume a priori only boundedness of the total interaction energy $ \interactionEnergy_{ \rho } $ and deduce the $ \lp^{ 2 } $-boundedness of the sequence $ \disdens_{ \rho } $. However, since we subtracted the self-energy from the total elastic energy, we cannot expect any coercivity and have to introduce it artificially. 
	In fact, looking at the approximation of the interaction energy (\ref{eq:interaction_energy_approximation}), we see that from the boundedness of the total interaction energy, we may at best expect a bound for the oriented surface area of the dislocation loops. But an estimate of the form
	\begin{equation}
		\label{eq:non_estimate}
		\hm^{ 2 } ( S )
		\leq
		C 
		\abs{ \int_{ S } n 
			\dd{ \hm^{ 2 } } }
	\end{equation}
	can not hold since the right hand side might be zero, while the left hand side is not, see Figure \ref{fig:or_surf_area_zero} for an example of such a surface whose boundary is given by a closed curve which violates (\ref{eq:non_estimate}). Furthermore, since the oriented surface area only depends on $ \partial S $, the choice of the surface $ S $ does not affect the failure of inequality (\ref{eq:non_estimate}). 
	The precise scaling of the norm comes from the passage from the discrete to the continuum setting as will become apparent in the proof of \Cref{prop:energy_asymptotics}.
	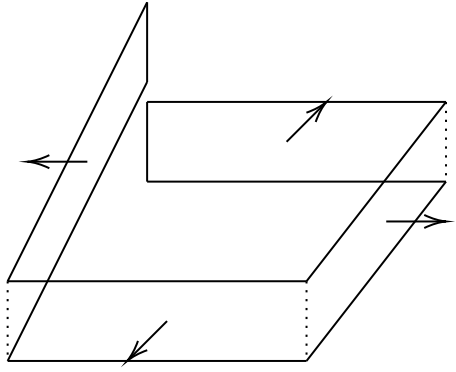
\begin{figure}
		
		
		\tikzset{every picture/.style={line width=0.75pt}} 
		
		\begin{tikzpicture}[x=0.75pt,y=0.75pt,yscale=-1,xscale=1]
			\path (65,0); 
			
			\draw    (220,200) -- (370,200) ;
			\draw    (220,160) -- (370,160) ;
			\draw  [dash pattern={on 0.84pt off 2.51pt}]  (220,160) -- (220,200) ;
			\draw  [dash pattern={on 0.84pt off 2.51pt}]  (370,160) -- (370,200) ;
			\draw    (370,160) -- (440,70) ;
			\draw    (370,200) -- (440,110) ;
			\draw  [dash pattern={on 0.84pt off 2.51pt}]  (440,70) -- (440,110) ;
			\draw    (290,110) -- (440,110) ;
			\draw [fill={rgb, 255:red, 65; green, 117; blue, 5 }  ,fill opacity=1 ]   (290,70) -- (440,70) ;
			\draw    (290,70) -- (290,110) ;
			\draw    (290,60) -- (220,200) ;
			\draw    (290,20) -- (220,160) ;
			\draw    (290,20) -- (290,60) ;
			\draw    (410,130) -- (438,130) ;
			\draw [shift={(440,130)}, rotate = 180] [color={rgb, 255:red, 0; green, 0; blue, 0 }  ][line width=0.75]    (10.93,-3.29) .. controls (6.95,-1.4) and (3.31,-0.3) .. (0,0) .. controls (3.31,0.3) and (6.95,1.4) .. (10.93,3.29)   ;
			\draw    (360,90) -- (378.59,71.41) ;
			\draw [shift={(380,70)}, rotate = 135] [color={rgb, 255:red, 0; green, 0; blue, 0 }  ][line width=0.75]    (10.93,-3.29) .. controls (6.95,-1.4) and (3.31,-0.3) .. (0,0) .. controls (3.31,0.3) and (6.95,1.4) .. (10.93,3.29)   ;
			\draw    (300,180) -- (281.41,198.59) ;
			\draw [shift={(280,200)}, rotate = 315] [color={rgb, 255:red, 0; green, 0; blue, 0 }  ][line width=0.75]    (10.93,-3.29) .. controls (6.95,-1.4) and (3.31,-0.3) .. (0,0) .. controls (3.31,0.3) and (6.95,1.4) .. (10.93,3.29)   ;
			\draw    (260,100) -- (232,100) ;
			\draw [shift={(230,100)}, rotate = 360] [color={rgb, 255:red, 0; green, 0; blue, 0 }  ][line width=0.75]    (10.93,-3.29) .. controls (6.95,-1.4) and (3.31,-0.3) .. (0,0) .. controls (3.31,0.3) and (6.95,1.4) .. (10.93,3.29)   ;

		\end{tikzpicture}

		
		\caption{An example of a non-trivial oriented surface whose oriented surface area is zero and its boundary is a closed curve.}
		
		\label{fig:or_surf_area_zero}
	\end{figure}
	
	The well-separatedness assumption \ref{item:wellSeperated} serves to differentiate the microscopic and macroscopic scales of the problem. Furthermore it enables us to treat the interaction of the dislocation loops as the interaction of two point forces. If the dislocation loops were too close to each other, then the interatomic forces will play a role in their interaction energy. 
\end{remark}

\begin{remark}
	\label{rmk:literature_comparison}
	We want to put our result into the context of the literature. In \cite{conti_garroni_ortiz_the_line_tension_approximation_as_the_diluate_limit_of_linear_elastic_dislocations} the authors consider a core-cutoff and core-mollification model for dilute dislocations represented by measures
	\begin{equation*}
		\mu = \sum_{ i } b_{ i } \otimes \tau \hm^{ 1 } \llcorner_{ \gamma_{ i } },
	\end{equation*}
	where the Burgers vectors have been rescaled to be in a fixed 
	atomic lattice $ \mathcal{B} $. They consequently show $ \Gamma 
	$-convergence  of the elastic energy, after a core-cutoff 
	respectively a core mollification, to a line tension energy $ F_{ 0 
	} $ under the topology of weak convergence of measures. Moreover 
	the energy has been rescaled by $ \abs{\log \eps } $.
	If $ \mu $ is like in our model of the form
	\begin{equation*}
		\label{eq:dilute_measure}
		\mu = \sum_{ x \in \lattice_{ \rho } \cap \Omega }
		b_{ \rho } ( x ) \otimes 
		\tau \hm^{ 1 } \llcorner_{ \gamma_{ \rho } ( x ) }
	\end{equation*}
	with dislocation loops of size $ r \ll \rho $ and unit length Burgers vector $ b_{ \rho } $, then
	\begin{equation}
		\label{eq:mass_mu}
		\abs{ \mu }
		\approx
		r \# ( \lattice_{ \rho } \cap \Omega )
		\approx
		\frac{r}{\rho^{ 3 } }.
	\end{equation}
	The $ \lp^{ 2 } $-norm of the associated surface area (see \ref{item:l2BoundOnArea}) is 
	\begin{equation}
		\label{eq:l2_area}
		\sum_{ x \in \lattice_{ \rho } \cap Q }
		\frac{1}{\rho^{ 3 } }
		(r^{ 2 })^{ 2 }
		\approx 
		\frac{r^{ 4 }}{ \rho^{ 6 } }
		=
		\left( \frac{r}{\rho} \right)^{ 3 } \frac{r}{ \rho^{ 3 } },
	\end{equation}
	which tends to zero if $ \abs{ \mu } $ stays bounded because of the well-separatedness assumption \ref{item:wellSeperated} which yields $ r \ll \rho $. Thus it follows from our computations that no interaction energy should survive in the limit $ \eps \to 0 $, which coincides with the result that the $ \Gamma $-limit in \cite{conti_garroni_ortiz_the_line_tension_approximation_as_the_diluate_limit_of_linear_elastic_dislocations} is the line-tension energy and does not consider any interaction energy of the dislocations.
	
	By considering a different scaling of the measure, but with the 
	same setup as in 
	\cite{conti_garroni_ortiz_the_line_tension_approximation_as_the_diluate_limit_of_linear_elastic_dislocations},
	we could hope to recover a higher-order $ \Gamma $-limit. The 
	two-dimensional analogue is the paper  
	\cite{garroni_leoni_ponsiglione_2010_gradient_theory_for_plasticity_via_homogenization_of_discrete_dislocations}.
	This would lead to the consideration of measures of the form as in 
	equation (\ref{eq:dilute_measure}).
	However we now say that $ \mu_{ \eps } 
	\to \mu $ if 
	\begin{equation*}
		\frac{1}{ \abs{\log ( \eps ) } }
		\mu_{ \eps }
		\rightharpoonup^{ \ast }
		\mu
	\end{equation*}
	and the elastic core-cutoff/core-mollification energy is rescaled 
	by $ \abs{ \log ( \eps ) }^{ 2 } $. 
	We then hope to recover a $ 
	\Gamma $-limit which both considers the self- and interaction 
	energy. 
	In order to compare this to our result, we note that if the radius 
	of the dislocation loops $ r $ and their distance $ \rho $ are 
	chosen such that $ ( \rho / r  )^{ 3 } = \abs{ \log( \eps ) } $, 
	then both $ \mu / \abs{ \log( \eps ) } $ and the $ \lp^{ 2 } $-norm 
	of the associated surface area may stay bounded, see the above 
	equations (\ref{eq:mass_mu}) and (\ref{eq:l2_area}). 
	This gives rise to potentially non-trivial limits for both 
	energies. It will be an upcoming task to compare these energies.
\end{remark}

\section{Dislocation Loops, Micromagnetics and Oscillations}
\label{sct:from_loops_to_surfaces}

The aim of this section is to first approximate the total interaction energy (\ref{eq:def_total_interaction_energy}) in terms of the oriented surface areas of the dislocation loops, see \Cref{prop:energy_asymptotics}. 
Afterwards we make a quick remark about the case when $ (\of_{ \rho } )_\rho $ converges strongly in $ \lp^{ 2 } ( \R^{ 3 } ; \R^{ 3 \times 3 } ) $.

Next we consider the case of $ (\disdens_{ \rho })_\rho  $ only having 
long-range oscillations as in (\ref{eq:weak_long}). We first pass to a 
double integral by piecewise constant extensions and by applying the 
Fourier transform. Due to the critical regularity of the kernel $ M $ 
however, we have to be careful at this step and crucially use that only 
long-range oscillations occur. We use ideas from 
\cite{james_mueller_internal_variables_and_fine_scale_oscillations_in_micromagnetics}.

\subsection{Interaction Energy via the Oriented Surface Area}

\label{sct:interaction_energy_via_oriented_surface_area}

\begin{proof}[Proof of \Cref{prop:energy_asymptotics}]
	Given the setting of \Cref{prop:energy_asymptotics}, we can use the definition of the interaction energy (\ref{eq:interaction_energy_non_regular}) and the total interaction energy (\ref{eq:def_total_interaction_energy}) to deduce 
	\begin{equation*}
		\interactionEnergy ( \disdens_{ \rho } )
		=
		\frac{ 1 }{ 2 }
		\sum_{ \substack{x, y \in \lattice_{ \rho } \\ x \neq y } }
		\int_{ S_{ \rho } ( x ) }
		\int_{ S_{ \rho } ( y ) }
		b_{ \rho } ( x' ) \otimes n ( x' ) 
		\colon 
		M ( x' - y' ) b_{ \rho } ( y' ) \otimes n ( y' )
		\dd{ \hm^{ 2 } ( y' ) }
		\dd{ \hm^{ 2 } ( x' ) }.
	\end{equation*}
	By approximating $ M ( x' - y' ) $ via $ M ( x - y ) $ and using the definition (\ref{eq:def_or_area_function}) of $ \of_{ \rho } $, we get that the total interaction energy is approximated by
	\begin{align}
		\label{eq:interaction_energy_two_loops_approx}
		\interactionEnergy ( \disdens_{ \rho } ) 
		& \approx  
		\frac{1 }{ 2 }
		\sum_{ \substack{x, y \in \lattice_{ \rho } \\ x \neq y } }
		\left( \int_{ S_{ \rho } ( x ) } b_{\rho } ( x' ) \otimes n ( x' ) \dd{ \hm^{ 2 } ( x' ) } 
		\right)
		\colon
		M ( x - y ) 
		\left( \int_{ S_{ \rho } ( y ) } b_{\rho } ( y' ) \otimes n ( y' ) \dd{ \hm^{ 2 } ( y' ) } 
		\right)
		\\
		\notag
		& =
		\frac{ 1 }{ 2 }
		\sum_{ \substack{x, y \in \lattice_{ \rho } \\ x \neq y } }
		\of_{ \rho } ( x ) \colon M ( x - y ) \of_{ \rho } ( y ),
	\end{align}
	which is what we wanted to show.
	It remains to prove that the error vanishes in the limit $ \rho \to 0 $.
	By the well-separatedness condition \ref{item:wellSeperated}, the definition of $ r > 0 $ there within and the $-4$-homogeneity of $ \diff M $, we have for all $ x \neq y \in \lattice_{ \rho } $ and $ x' \in S_{ \rho } ( x ) , y' \in S_{ \rho } ( y ) $ that
	\begin{align*}
		\abs{ M ( x - y ) - M ( x' - y' ) }
		& \leq
		\sup_{ \xi \in [x-y, x'-y']}
		\abs{ \diff M ( \xi ) }
		\abs{ x - y - ( x' - y' ) }
		\\
		& \lesssim
		(\abs{ x - y } - 2r)^{ -4 } r.
	\end{align*}
	Again using the well-separatedness condition \ref{item:wellSeperated}, we have for all $ x, y \in \lattice_{ \rho } $ with $ x \neq y $ that $ \abs{x - y } - 2r \gtrsim \abs{ x - y } $. Therefore the error we make in equation (\ref{eq:interaction_energy_two_loops_approx}) is bounded from above up to a constant by
	\begin{equation}
		\label{eq:error_latice_form}
		r \sum_{ \substack{x, y \in \lattice_{ \rho } \\ x \neq y } }
		\abs{ x - y }^{ -4 } 
		\int_{ S_{ \rho } ( x ) }\abs{ b_{ \rho } ( x' ) } \dd{ \hm^{ 2 } ( x' ) }  
		\int_{ S_{ \rho } ( y ) }\abs{ b_{ \rho } ( y' ) } \dd{ \hm^{ 2 } ( y' ) }.
	\end{equation}
	Applying the Cauchy--Schwarz inequality and Young's integral inequality to the term (\ref{eq:error_latice_form}) yields that the error can be estimated up to a constant by
	\begin{equation*}
		r 
		\sum_{ x \in \lattice_{ \rho } \setminus \{0\}}
		\abs{x}^{ - 4 }
		\sum_{ x \in \lattice_{ \rho } } 
		\left( \int_{ S_{ \rho } ( x ) } \abs{ b_{ \rho } ( x' ) } \dd{ \hm^{ 2 } ( x' ) } \right)^{ 2 }
		\lesssim
		\frac{ r }{ \rho } \sum_{ x \in \lattice_{ \rho } } \frac{1}{ \rho^{ 3 } }\left( \int_{ S_{ \rho } ( x ) } \abs{ b_{ \rho } ( x' ) } \dd{ \hm^{ 2 } ( x' ) } \right)^{ 2 },
	\end{equation*}
	which tends to zero by the boundedness of the surface areas \ref{item:l2BoundOnArea} and the well-separatedness condition \ref{item:wellSeperated}. This finishes the proof.
\end{proof}

\subsection{Macroscopic Integral Operator and Strong Convergence}

Using the asymptotic representation of the total interaction energy derived in \Cref{prop:energy_asymptotics}, we make the following observation. The total interaction energy is similar to the energy (6.1) in \cite{james_mueller_internal_variables_and_fine_scale_oscillations_in_micromagnetics} which describes the total energy of dipoles on a lattice. One difference from the dislocation setting is that the oriented surface area function $ \of_{ \rho } $ is $ \R^{ 3 \times 3 } $-valued in contrast to the $ \R^{ 3 } $-valued dipole function. The kernel in the dipole setting is $ -3 $-homogeneous (to be precise, it is the second derivative of the fundamental solution to the Laplace equation in 3 dimensions) and $ \R^{ 3 \times 3 } $-valued. Since our kernel $ M $ is also $ - 3 $-homogeneous, but $ \R^{ (3\times 3) \times ( 3 \times 3 ) } $-valued, it should be possible to analyze the asymptotic behaviour of the total interaction energy by using the tools from \cite{james_mueller_internal_variables_and_fine_scale_oscillations_in_micromagnetics}.

In \cite[Thm.~6.2]{james_mueller_internal_variables_and_fine_scale_oscillations_in_micromagnetics}, the authors show that there exists some $ C \in \R^{ 3\times 3} $ such that if the array of dipoles $ f_\rho $ converges strongly in $ \lp^{ 2 } ( \R^{ 3 } ; \R^{ 3 } ) $ to some $ f \in \lp^{ 2} ( \R^{ 3 } ; \R^{ 3 } ) $, then the total interaction energy converges to
\begin{equation}
	\label{eq:energy_convergence_strong_case}
	\lim_{ \rho \to 0 } 
	\interactionEnergy_\rho
	=
	\frac{ 1 }{ 2 }
	\inner*{f}{M \ast f + C f }_{ \lp^{ 2 } }.
\end{equation}
Note that because $ M $ is not locally integrable, the convolution $ M \ast f $ is a priori not well-defined. The term is understood in a distributional sense by using that $ M $ can be written as the derivative of a $-2 $-homogeneous function, which is locally integrable. Since we are going to adapt a Fourier picture of the setting, we are not going to delve into the details, for which we refer to (\cite[Sct.~2 and Rmk.~6.3]{james_mueller_internal_variables_and_fine_scale_oscillations_in_micromagnetics}).

\subsection{Weak-Long Oscillations}
\label{sct:weak_long}

Because of the quadratic nature of the limit in equation (\ref{eq:energy_convergence_strong_case}), we do not expect that without strong convergence of the function $ \widetilde{ \of_{ \rho } } $, this result still holds. Since we exclude escape to infinity by confining the loops to a bounded domain, we can only get weak convergence without strong convergence if oscillations occur, like in equation (\ref{eq:example_weak_long}). If the oscillations of $ \widetilde{\of_{ \rho }} $ take place on a scale much larger than the scale of the lattice, we might expect that the macroscopic operator $ (M + C \delta_0 ) \ast $ in equation (\ref{eq:energy_convergence_strong_case}) still plays a role for the limit of the total interaction energy.

In the case of oscillations on the lattice however, we do not expect that the macroscopic operator can be used to give us information about the limiting behaviour of the total interaction energy. We will instead rescale the lattice and use the Fourier series of $ \of_{ \rho } $ to analyze the limiting behaviour, see \Cref{sct:weak_short}.

The following Lemma is an equivalent of \cite[Prop.~7.1]{james_mueller_internal_variables_and_fine_scale_oscillations_in_micromagnetics}. We will use slightly different methods for the proof by adapting a Fourier picture.

\begin{lemma}
	\label{lem:weak_long_energy_rewritten}
	Let $ \disdens_{ \rho } =\left( b_{ \rho } ( x ) \otimes \tau \hm^{ 1 } \llcorner_{ \gamma_{ \rho } ( x ) } \right)_{ x \in \lattice_{ \rho } } $ be a sequence of arrays of dislocation loops which satisfies assumptions \ref{item:l2BoundOnArea},  \ref{item:wellSeperated} and only has long-range oscillations (\ref{eq:weak_long}). Then
	\begin{equation}
		\lim_{ \rho \to 0 }
		\interactionEnergy(\disdens_{ \rho })
		=
		\lim_{ \rho \to 0 }
		\frac{ 1 }{ 2 } 
		\inner*{ \fourier \widetilde{ \of_{ \rho } } }{\Psi \fourier \widetilde{ \of_{ \rho } } }_{ \lp^{2 } ( \R^{ 3 } ) } 
		.
	\end{equation}
	Here $ \Psi \colon \R^{ 3 } \setminus \{ 0 \}  \to \R^{ (3 \times 3 ) \times ( 3 \times 3 )} $  is a smooth, $ 0 $-homogeneous function defined in (\ref{eq:Psi_def}) which only depends on the lattice $ \lattice_{ 1 } $ and the elastic tensor $ \eltensor $.
\end{lemma}
\begin{proof}
	We first rewrite the asymptotic energy representation in equation (\ref{eq:interaction_energy_two_loops_approx}) through a double integral. To this end we discretize the kernel $ M $ on the lattice $ \lattice_{ \rho } $ and introduce a suitable cutoff since we only sum over non-equal lattice points. 
	In fact let $ \eta \in \ccinf \left( \R^{ 3 } ;[0, 1] \right) $ be radially symmetric with $ \eta = 1  $ in $ B_{ \omega/2 } ( 0 ) $ and $ \eta = 0 $ outside of $ B_{ \omega } ( 0 ) $, where $ \omega > 0 $ is chosen sufficiently small such that $ \lattice_{ 1 } \setminus \{ 0 \} \subseteq \R^{ 3 } \setminus B_{ \omega } ( 0 ) $. Then we define
	\begin{align*}
		\tilde{ M }^{ ( \rho ) }
		( x , y )
		& \coloneqq 
		\sum_{ \substack{
				v, w \in \lattice_{\rho }\\ v \neq w } }
		\chi_{ v + \rho U } ( x )
		M ( v - w )
		\chi_{ w + \rho U } ( y )
		\shortintertext{and the regularized kernel}
		M^{ ( \rho ) }
		( z )
		& \coloneqq 
		( 1 - \eta ( z/\rho ) ) M ( z ).
	\end{align*}
	We first note that $ \widetilde{ \of_{ \rho } } $ is $ \lp^{ 2 } $-bounded by inequality (\ref{eq:l2_bound_weaker}) and the assumption \ref{item:l2BoundOnArea}. Using the above definitions, we can rewrite the approximate total interaction energy as
	\begin{align}
		\notag
		\frac{ 1 }{ 2 }
		\sum_{ \substack{x, y \in \lattice_{ \rho } \\ x \neq y } }
		\of_{ \rho } ( x ) \colon M ( x - y ) \of_{ \rho } ( y )
		& =
		\frac{ 1 }{ 2 }
		\inner*{ \widetilde{ \of_{ \rho } } }{ \tilde{ M }^{ (\rho) } \ast \widetilde{ \of_{ \rho } } }_{ \lp^{ 2 } } 
		\\
		\label{eq:split_of_interaction_energy}
		&=
		\frac{ 1 }{ 2 }
		\inner*{ \widetilde{ \of_{\rho} } }{ M^{ (\rho) } \ast \widetilde{ \of_{ \rho } } }_{ \lp^{ 2 } }
		+
		\frac{ 1 }{ 2 }
		\inner*{ \widetilde{ \of_{ \rho } } }{ \left(\tilde{ M }^{ (\rho ) } - M^{ (\rho) } \right)  \ast \widetilde{ \of_{ \rho } } }_{ \lp^{ 2 } }.
	\end{align}
	Let us consider the first summand in the term (\ref{eq:split_of_interaction_energy}).
	We apply the Plancherel identity to obtain
	\begin{equation*}
		\frac{ 1 }{ 2 }
		\inner*{ \widetilde{ \of_{\rho} } }{ M^{ (\rho) } \ast \widetilde{ \of_{ \rho } } }_{ \lp^{ 2 } }
		=
		\frac{ 1}{ 2 }
		\inner*{\fourier \widetilde{\of_{ \rho } } }{ \fourier M^{ (\rho ) } \fourier \widetilde{\of_{ \rho } } }_{ \lp^{ 2 } } 
		=
		\frac{ 1 }{ 2 }
		\inner*{\fourier \widetilde{\of_{ \rho } } }{ \fourier \left(  (1- \eta\left( \cdot /\rho \right) ) M  \right) \fourier \widetilde{\of_{ \rho } } }_{ \lp^{ 2 } }.
	\end{equation*}
	We claim that, in the sense of tempered distributions (and 
	pointwise), we have
	\begin{equation}
		\label{eq:distribution_convergence_cutoff}
		\fourier \left( ( 1- \eta ( \cdot /\rho )  ) M \right)
		\to 
		\fourier M - \fint_{ \Sph^{ 2 } }\fourier M ( \xi ) \dd{ \hm^{ 2 } ( \xi ) }.
	\end{equation}
	Note that $ \fourier M $ is uniquely associated to a smooth, $0$-homogeneous function on $ \R^{ 3 } \setminus \{ 0 \} $. In order to show the claim, we compute
	\begin{align*}
		\fourier \left(
		\varphi \left(\frac{ \cdot }{\delta } \right)M
		\right) ( \xi )
		& =
		\delta^{ n } 
		\fourier \varphi  (\delta \cdot )
		\ast
		\fourier M (\xi )
		\\
		& =
		\int_{ \mathbb{R}^{ n } }
		\fourier \varphi ( y )
		\fourier M \left(\frac{ \delta \xi - y }{ \abs{ \delta \xi - y } }\right)
		\dd{ y }
		\\
		& \to
		\int_{ \mathbb{R}^{ n } }
		\fourier \varphi ( y )
		\fourier M ( - y )
		\dd{ y }.
	\end{align*} 
	The last equality follows from the $0$-homogeneity of $ \fourier M $ and the convergence follows from the dominated convergence theorem. Since $ \varphi $ is radially symmetric, so is its Fourier transform. Additionally using the $ 0$-homogeneity of $ \fourier M $ again, we get by the coarea formula that
	\begin{align*}
		\int_{ \mathbb{R}^{ n } }
		\fourier \varphi ( y ) 
		\fourier M ( - y )
		\dd{ y }
		& =
		\int_{ 0 }^{ \infty }
		\fourier \varphi ( r )
		r^{ n - 1 }
		\int_{ \mathbb{ S }^{ n - 1 } }
		\fourier M (-yr)
		\dd{ \hm^{ n - 1 } }
		\dd{ r }
		\\
		& =
		\int_{ \mathbb{ S }^{ n - 1 } }
		\fourier M ( y )
		\dd{ \hm^{ n - 1 } ( y ) }
		\int_{ 0 }^{ \infty }
		\fourier \varphi ( r )
		r^{ n - 1 }
		\dd{ r }
		\\
		& =
		\int_{ \mathbb{ S }^{ n - 1 } }
		\fourier M ( y )
		\dd{ \hm^{ n - 1 } ( y ) }
		\frac{ 1}{ \abs{ \mathbb{ S }^{ n - 1 } } }
		\varphi ( 0 ).
	\end{align*}
	Here the last equality follows from the Fourier inversion formula.
	We thus conclude that 
	\begin{equation*}
		\lim_{ \delta \to 0 }
		\fourier \left(
		\varphi \left(\frac{ \cdot }{\delta } \right)M
		\right) ( \xi )
		=
		\fint_{ \mathbb{ S }^{ n - 1 } }
		\fourier M ( y )
		\dd{ \hm^{ n - 1 } },
	\end{equation*}
	which proves the claim (\ref{eq:distribution_convergence_cutoff}).
	
	Let $ 0 < \varepsilon < \rho $. Then we have
	\begin{align}
		\label{eq:maxwell_self_energy-convergence}
		& \inner*{\fourier \widetilde{\of_{ \rho } }}{ \fourier M^{ (\rho) }  \fourier \widetilde{\of_{ \rho } } }
		-
		\inner*{\fourier \widetilde{\of_{ \rho } }}{\left( \fourier M - (\fourier M)_{ \Sph^{ 2 }} \right) \fourier \widetilde{\of_{ \rho } } }
		\\
		\notag
		={} &
		\inner*{\fourier \widetilde{\of_{ \rho } }}{ \fourier M^{ (\varepsilon ) } \fourier \widetilde{\of_{ \rho } } }
		-
		\inner*{\fourier \widetilde{\of_{ \rho } }}{\left( \fourier M - (\fourier M)_{ \Sph^{ 2 }} \right) \fourier \widetilde{\of_{ \rho } } }
		+
		\inner*{\fourier \widetilde{\of_{ \rho } }}{\left(\fourier M^{ (\rho ) } - \fourier M^{ (\varepsilon ) }  \right) \fourier \widetilde{\of_{ \rho } } }
	\end{align}
	The difference of the first summands converges to zero if $ \rho>0 $ is fixed and $ \varepsilon \to 0 $ since $ \fourier M^{ \varepsilon }  $ is uniformly bounded in $ \lp^{ \infty } $, see also the proof of \Cref{prop:Feps_behaviour}. 
	To see that the last term vanishes, we use that $ \widetilde{ \of_{ \rho} } $ has only long-range oscillations. In fact, we argue as in \cite[(7.13)]{james_mueller_internal_variables_and_fine_scale_oscillations_in_micromagnetics}, for which we first need to show that the kernel $ M $ has the cancellation property
	\begin{equation*}
		\int_{ \Sph^{ 2 } }
		M ( x ) 
		\dd{ \hm^{ 2 } ( x ) }
		= 0,
	\end{equation*}
	which we moved to the appendix (see \Cref{lem:cancellation_of_M}). Applying Plancherel then yields that
	\begin{equation}
		\label{eq:third_summand_after_plancherel}
		\inner*{\fourier \widetilde{\of_{ \rho } }}{\left(\fourier M^{ (\rho ) } - \fourier M^{ (\varepsilon ) }  \right) \fourier \widetilde{\of_{ \rho } } }
		=
		\inner*{ \widetilde{ \of_{ \rho} }}{(M^{ ( \rho ) } -M^{ ( \eps )} ) \ast \widetilde{ \of_{ \rho} }}
	\end{equation}
	We then note that the convolution operator $ M^{ (\rho )} \ast $ is 
	uniformly bounded as $ \rho \to 0 $ as a map from $ \lp^{ 2 } $ to 
	$ \lp^2 $ by 
	\cite[Ch.~II,~Sct.~3,~Thm.~2]{stein_singular_integrals_and_differentiability_properties_of_functions}.
	
	Since $ \of_{ \rho } $ only has long-range oscillations, it follows in particular that for a radial symmetric standard mollifier $ \varphi \in \ccinf ( B_\omega ( 0 ) ) $, we have that
	\begin{equation}
		\label{eq:g_eps_replaced_with_convolution}
		\lim_{ \rho \to 0 }
		\lim_{ \eps \to 0 }
		\norm{(M^{ ( \rho ) } -M^{ ( \eps )} ) \ast (\widetilde{ \of_{ 
					\rho} } - \varphi_{ \rho } \ast \widetilde{ \of_{ \rho} } ) }_{ 
			\lp^{ 2 } ( \R^{ 3 } ) } 
		=
		0.
	\end{equation}
	Moreover we note that by Minkowski's integral inequality and because $ M^{ ( \rho ) } - M^{ ( \eps ) } $ is supported in $ B_{ \omega \rho } ( 0 ) $, we have that
	\begin{align}
		\notag
		&\int_{ \R^{ 3 } }
		\abs{
			\varphi_\rho \ast ( M^{ (\rho )} - M^{ ( \eps ) } ) \ast \widetilde{ \of_{ \rho} }
		}^{ 2 }
		\dd{ x }
		\\
		\notag
		= {} &
		\int_{ \R^{ 3 } }
		\abs{ \int_{ B_{ 2\omega \rho } ( 0 ) }  \varphi_{ \rho } \ast ( M^{ ( \rho ) } - M^{ ( \eps ) } )( y ) \widetilde{ \of_{ \rho} } ( x - y )
			\dd{ y } }^2
		\dd{ x }
		\\
		\notag
		={} &
		\int_{ \R^{ 3 } }
		\abs{ \int_{ B_{ 2 \omega \rho } ( 0 ) }  \varphi_{ \rho } \ast ( M^{ ( \rho ) } - M^{ ( \eps ) } )( y ) ( \widetilde{ \of_{ \rho} } ( x - y ) - \widetilde{ \of_{ \rho} } ( x ) ) \dd{ y }
		}^2
		\dd{ x }
		\\
		\label{eq:kernel_computation}
		\leq {} &
		\left( \int_{ B_{ 2 \omega \rho } ( 0 ) }
		\abs{ \varphi_{ \rho } \ast ( M^{ ( \rho ) } - M^{ ( \eps ) } ) ( y ) } 
		\left(
		\int_{ \R^{ 3 } } 
		\abs{ \widetilde{ \of_{ \rho} } ( x - y ) - \widetilde{ \of_{ \rho} } ( x ) }^{ 2 }
		\dd{ x }
		\right)^{1/2}
		\dd{ y } 
		\right)^2.
	\end{align} 
	Since $ \of_{ \rho} $ has only long-range oscillations, it thus suffices to show that $ \varphi_{ \rho } \ast ( M^{ ( \rho ) } - M^{ ( \eps ) } ) $ is uniformly bounded in $ \lp^{ 1 } ( \R^{ 3 } ) $.
	In fact, using again that $ \int_{ \Sph^{ 2 } } M = 0 $ together with the radial symmetry of $ \varphi $, we estimate
	\begin{align*}
		&
		\int_{ B_{ 2\omega \rho } ( 0 ) }
		\abs{ \varphi_{ \rho } \ast ( M^{ (\rho ) } - M^{ ( \eps )})}
		\dd{ x }
		\\
		={} &
		\int_{ B_{ 2 \omega \rho } ( 0 ) }
		\abs{ 
			\int_{ \R^{ 3 } } ( M^{ ( \rho ) } ( x - y ) - M^{ ( \eps  
				) }  ( x - y  ) ) ( \varphi_{ \rho } ( y) - \varphi_{ \rho 
			} ( x ) )
			\dd{ y }  
		}
		\dd{ x }
		\\
		={} &
		\int_{ B_{ 2 \omega \rho } ( 0 ) }
		\abs{
			\int_{ 0 }^1 
			\chi_{
				B_{ 2 \omega  \rho ( 0 ) } } ( x - y)
			( M^{ ( \rho ) } ( x - y ) - M^{ ( \eps ) } ( x - y ) )
			\nabla ( \varphi_{ \rho } ) ( x + t ( y - x ) ) \cdot ( y - x  )
			\dd{ y }
			\dd{ t } }
		\dd{ x }
		\\
		\lesssim 
		{} &
		\frac{1}{\rho }
		\int_{ B_{ 2 \omega \rho } ( 0 ) }
		\frac{1}{\abs{z}^{ 2 } }
		\dd{ z },
	\end{align*}
	which stays uniformly bounded as $ \rho \to 0 $ independent of $ \eps > 0 $. Combining this with (\ref{eq:g_eps_replaced_with_convolution}) and (\ref{eq:kernel_computation}) concludes that (\ref{eq:third_summand_after_plancherel}) vanishes in the limit when we first send $ \eps $ to zero and then $ \rho $ to zero.
	
	We have thus proven that
	\begin{equation}
		\label{eq:convergence_maxwell_self_energy}
		\lim_{ \rho \to 0 }
		\abs{\inner*{ \widetilde{ \of_{ \rho} } }{ M^{ (\rho) } \ast \widetilde{ \of_{ \rho} } }_{ \lp^{ 2 } }
			-
			\inner*{\fourier \widetilde{ \of_{ \rho} } }{\left( \fourier M  - (\fourier M )_{ \Sph^{ 2 }} \right) \fourier \widetilde{ \of_{ \rho} } }
		}
		=0.
	\end{equation}
	
	For the second summand in the term (\ref{eq:split_of_interaction_energy}), we first note that 
	\begin{equation*}
		(\tilde{ M }^{ (\rho) } - M^{ ( \rho ) })  ( x , y ) 
		= 
		\frac{1}{\rho^3}(\tilde{ M }^{ ( 1 ) } - M^{ ( 1 ) } ) \left( \frac{x}{\rho},\frac{y}{\rho} \right).
	\end{equation*} 
	Moreover we note due to the $ - 3 $ homogeneity of $ M $ that there exists some integrable majorant $ h \in \lp^{ 1 } ( \R^{ 3 } ; [0, \infty) ) $ such that $ \abs{ \tilde{ M }^{ ( 1 ) } - M^{ (1)} } ( x, y ) \leq h ( x - y ) $.
	We thus write
	\begin{align*}
		& \int_{ \R^{ 3 } } ( \tilde{ M }^{ ( \rho ) } - M^{ (\rho ) } ) ( x, y ) \widetilde{ \of_{ \rho} } ( y ) 
		\dd{ y }
		\\
		={} &
		\int_{ \R^{ 3 } }
		( \tilde{ M }^{ ( \rho ) } - M^{ (\rho ) } ) ( x, y )
		\dd{ y }
		\widetilde{ \of_{ \rho} } ( x )
		+
		\int_{ \R^{ 3 } }
		( \tilde{ M }^{ ( \rho ) } - M^{ (\rho ) } ) ( x, y )
		( \widetilde{ \of_{ \rho} } ( y ) - \widetilde{ \of_{ \rho} } ( x ) )
		\dd{ y }.
	\end{align*}
	For the first summand we note that due to the integrable majorant and the cancellation property of $ M^{ ( \rho ) } $, we have
	\begin{align*}
		\int_{ \R^{ 3} } 
		( \tilde{ M }^{ ( \rho ) } - M^{ (\rho ) } ) ( x, y )
		\dd{ y }
		& =
		\lim_{ n \to \infty }
		\int_{ B_{n } ( x ) }
		( \tilde{ M }^{ ( \rho ) } - M^{ (\rho ) } ) ( x, y )
		\dd{ y }
		\\
		\\
		& =
		\lim_{ n \to \infty }
		\int_{ B_{n } ( x ) }
		\tilde{ M }^{ ( \rho ) } ( x, y )
		\dd{ y }
		\\
		& =
		\lim_{ R \to \infty }
		\sum_{ x \in \lattice_{ 1 } \setminus \{0\} \cap B_{ R } ( 0 )}
		M ( x )
		\eqqcolon S.
	\end{align*}
	For the second summand, we use the majorant $ h $, the transformation formula and Minkowski's integral inequality to estimate
	\begin{align*}
		&
		\left(
		\int_{ \R^{ 3 } }
		\abs{
			\int_{ \R^{ 3 } } 
			( \tilde{ M }^{ ( \rho ) } - M^{ (\rho ) } ) ( x, y ) ( \widetilde{ \of_{ \rho} } ( x ) - \widetilde{ \of_{ \rho} } ( y ) ) 
			\dd{ y }
		}^{ 2 }
		\dd{ x }
		\right)^{1/2}
		\\
		\leq{}
		&
		\left(
		\int_{ \R^{ 3 } }
		\left(
		\int_{ \R^{ 3 } } 
		\frac{1}{\rho^3} h\left(\frac{ x - y }{\rho }\right) 
		\abs{ \widetilde{ \of_{ \rho} } ( x ) - \widetilde{ \of_{ \rho} } ( y ) } 
		\dd{ y }
		\right)^{ 2 }
		\dd{ x }
		\right)^{1/2}
		\\
		={}&
		\left(
		\int_{ \R^{ 3 } }
		\abs{
			\int_{ \R^{ 3 } } 
			h(y) 
			( \widetilde{ \of_{ \rho} } ( x ) - \widetilde{ \of_{ \rho} } ( x - \rho y  ) ) 
			\dd{ y }
		}^{ 2 }
		\dd{ x }
		\right)^{1/2}
		\\
		\leq{} &
		\int_{ \R^{ 3 } }
		h ( y ) 
		\left(
		\int_{ \R^{ 3 } }
		\abs{ \widetilde{ \of_{ \rho} } ( x ) - \widetilde{ \of_{ \rho} } ( x - \rho y  ) }^2 
		\dd{ x }
		\right)^{1/2}
		\dd{ y }.
	\end{align*}
	Using the integrability of $ h $, splitting the integral into $ B_N ( 0 ) $ and its complement, and then sending $ N \to \infty $ yields that this term vanishes due to $ \of_{ \rho } $ only having long-range oscillations. We have thus shown that
	\begin{equation}
		\label{eq:convergence_local_lorentz_energy}
		\lim_{ \rho \to 0 }
		\abs{ \inner*{ \widetilde{ \of_{ \rho} } }{ \left(\tilde{ M }^{ (\rho ) } - M^{ (\rho) } \right)  \ast \widetilde{ \of_{ \rho} } }_{ \lp^{ 2 } }
			-
			\inner*{ \widetilde{ \of_{ \rho} } }{ S \widetilde{ \of_{ \rho} } }_{ \lp^{ 2 } } }
		=
		0,
	\end{equation}
	where $ S \in \R^{ (3 \times 3 )\times ( 3 \times 3 ) } $ is given by
	\begin{equation}
		\label{eq:def_lattice_sum}
		S
		\coloneqq
		\lim_{ R \to \infty }
		\sum_{ x \in \lattice_{ 1 } \setminus \{ 0 \} \cap B_{ R } ( 0 ) }
		M ( x ).
	\end{equation}
	By equations (\ref{eq:convergence_maxwell_self_energy}) and (\ref{eq:convergence_local_lorentz_energy}), we thus conclude that
	\begin{equation}
		\label{eq:weak_long_asymptotic_convergence_of_energy}
		\lim_{ \rho \to 0 }
		\abs{ 
			\frac{ 1 }{ 2 }
			\sum_{ \substack{x, y \in \lattice_{ \rho } \\ x \neq y } }
			\of_{ \rho } ( x ) \colon M ( x - y ) \of_{ \rho } ( y )
			-
			\frac{ 1 }{ 2 }
			\inner*{\fourier \widetilde{ \of_{ \rho} } }{ \left( \hat{M} - (\hat{M})_{ \Sph^{ 2 } } + S \right) \fourier \widetilde{ \of_{ \rho} } }_{ \lp^{ 2 } }
		}
		=
		0,
	\end{equation}
	and thus by defining 
	\begin{equation}
		\label{eq:Psi_def} 
		\Psi \coloneqq \hat{M} - ( \hat{M})_{ \Sph^{ 2 } } + S,
	\end{equation}
	the proof is complete.
\end{proof}

\subsection{\texorpdfstring{$ \Hm $}{H}-measures}
\label{sct:H_measures}

In this section, we will further investigate the right hand side of equation (\ref{eq:main_thm}) by introducing $ \Hm $-measures and computing some examples. We close the Section by proving \Cref{prop:weak_long_energy_convergence}.
Let $ N \coloneqq 3 \times 3 = 9 $ for notational simplicity.
By \Cref{lem:weak_long_energy_rewritten}, we have reduced the weak-long case to the study of
\begin{equation}
	\label{eq:l2_lim_we_study}
	\lim_{ \rho \to 0 }
	\frac{ 1 }{ 2 }
	\inner*{\fourier \widetilde{ \of_{ \rho} } }{ \Psi \fourier \widetilde{ \of_{ \rho} } }_{ \lp^{ 2 } }
\end{equation}
under the assumption that $ \widetilde{\of_{ \rho }} $ is weakly converging in $ \lp^{ 2 }  \left( \R^{ 3 } ; \R^{ N }\right)$ and $ \Psi \colon \R^{ 3 }\setminus\{ 0 \} \to \R^{ N \times N } $ is some smooth, $ 0 $-homogeneous function. 

One of the tools to study limits of oscillating sequences are Young measures introduced in \cite{young_42_generalized_surfaces_in_calcvar}, see also the textbook \cite{rindler_18_calcvar} for a more modern treatment. 
However, we can quickly see that these are not suitable for the study 
of limits such as (\ref{eq:l2_lim_we_study}) since they can only detect 
the values between which a sequence oscillates, but not in which 
direction. This is illustrated in the following example.
\begin{example} 
	\label{ex:h_measure_young_measure}
	Consider the sequence $ f_{ n } ( x ) = f ( x ) \exp ( 2 \pi i n \inner*{ \nu }{ x } ) $ for some fixed $ f \in \lp^{ 2 } ( \R^{ 3 } ) $. Then by Weyl's equidistribution Theorem, $ f_n $ generates the Young measure 
	\begin{equation*}
		\nu_x = \frac{1}{ 2 \pi \abs{ f ( x ) } }
		\hm^{ 1 } \llcorner_{ \abs{ f ( x ) } \Sph^{ 1 } },
	\end{equation*}
	with $ \nu_x = \delta_0 $ if $ f ( x ) = 0 $. Since its Fourier transform is given by $ \fourier f_n ( \xi ) = \fourier f ( \xi - n \nu ) $, $ \fourier f_n $ generates the Homogeneous Young measure $ \delta_0 $. 
	However, in the context of the limit (\ref{eq:l2_lim_we_study}), we get by the $ 0 $-homogeneity of $ \Psi $ that
	\begin{align*}
		\int_{ \R^3 } \hat{ f_{ n } } ( \xi ) \overline{\Psi ( \xi ) \hat{ f_{ n } } ( \xi ) } \dd{ \xi }
		& =
		\int_{ \R^{ 3 } }
		\abs{ \hat{ f } ( \xi - n \nu ) }^{2 }
		\Psi ( \xi ) 
		\dd{ \xi }
		\\
		& =
		\int_{ \R^3}
		\abs{ \hat{ f } ( \xi ) }^{ 2 }
		\Psi ( \xi + n \nu )
		\dd{ \xi  }
		\\
		& \to 
		\int_{ \R^3}
		\abs{ \hat{ f } ( \xi ) }^{ 2 }
		\dd{ \xi }
		\Psi ( \nu ),
	\end{align*}
	where the last convergence follows from the dominated convergence theorem.
\end{example}
Thus the direction of oscillation plays a role for the limit of the $ \lp^{ 2 } $-inner product. L. Tartar proved the following (see \cite[Thm.~1.1]{tartar_H_measures_a_new_approach}).
\begin{theorem}
	\label{thm:H_measures}
	Let $ (f_{ n } )_{ n \in \N } \subseteq \lp^{ 2 } \left( \R^{ d } ; \R^{ N } \right) $ be a sequence converging weakly in $ \lp^{ 2 } $ to zero. Then there exists some non-relabelled subsequence and a family $ ( \mu_{ i j } )_{ i, j \in \{ 1, \ldots, N \} } $ of complex-valued Radon measures on $ \R^{ d } \times \Sph^{ d - 1 } $ such that for every $ \phi_{ 1 }, \phi_{ 2 } \in \cont_{ 0 } \left( \R^{ d } \right) $, every $ 0 $-homogeneous $ \psi \in \cont \left( \R^{ d } \setminus \{ 0 \}\right)  $ and all $ i, j \in \{ 1 , \ldots , N \} $, we have
	\begin{equation*}
		\lim_{ n \to \infty }
		\inner*{  \fourier \left( \phi_{ 1 } f_{ n }^{ i } \right) }{ \psi \fourier \left( \phi_{ 2 } f_{n }^{ j } \right) }_{ \lp^{ 2 } }
		= 
		\int_{ \R^{ d } \times \Sph^{ d - 1 } }
		\phi_{ 1 } ( x ) \overline{ \psi ( \xi ) \phi_{ 2 } ( x ) }
		\dd{ \mu_{ i j } ( x , \xi ) }.
	\end{equation*}
	Furthermore the measure $ \mu $  is hermitian non-negative in the sense that 
	\begin{enumerate}
		\item $ \mu_{i j } = \overline{ \mu_{j i } } $ for all $ i, j \in \{ 1 , \ldots, N \} $ and
		\item For all $\phi_1,\ldots,\phi_N \in
		\cont_0(\R^d;\C)$ and all
		$\psi\in\cont(\Sph^{d-1};\R)$ with $\psi\geq0$, we have
		\begin{equation*}
			\sum_{i,j=1}^N
			\int_{\R^d\times\Sph^{d-1}}
			\phi_i(x)\overline{\phi_j(x)}\,\psi(\xi)
			\dd{\mu_{ij}(x,\xi)}
			\geq 0.
		\end{equation*}
	\end{enumerate}
\end{theorem} 
Building on \Cref{ex:h_measure_young_measure}, we want to examine more closely what possible $ \Hm $-measures look like.
\begin{example}
	\label{ex:h_measures_of_simple_form}
	Let the dimension $ d $ and co-dimension $ N $ be given as in 
	\Cref{thm:H_measures}. Furthermore let $ f \in \lp^{ 2 } ( \R^{ d } 
	) $, $ \nu \in \Sph^{ d - 1 } $ and $ A \in \C^{ N \times N } $ be 
	a hermitian non-negative matrix. We want to show that there exists 
	a sequence $ f_{ n } \in \lp^{ 2 } ( \R^{ d } ; \C^{ N } ) $ 
	converging weakly to zero in $ \lp^{ 2 } ( \R^{ d } ; \C^{ N } ) $ 
	which generates the $ \Hm$-measure $ \mu = \abs{ f }^{ 2 } A \dd{ 
		\lm^{ d } } \otimes \delta_{ \nu } $, which acts on $ \varphi \in 
	\cont ( \R^{ d } \times \Sph^{ d - 1 } ; \C^{ N \times N } ) $ via
	\begin{equation*}
		\inner*{ \mu }{ \varphi }
		=
		\int_{ \R^{ d } }
		\sum_{ i, j = 1 }^{ N }
		A_{ i j } \varphi_{ i j }(x , \nu )
		\abs{ f(x )  }^{ 2 }
		\dd{ x }.
	\end{equation*}
	To this end, we first assume that $ A $ is a diagonal matrix. Then all eigenvalues $ \lambda_{ 1 }, \ldots, \lambda_{ N } $ are non-negative. Take pairwise different positive numbers $ \theta_{ 1 }, \ldots , \theta_{ N } $  and define
	\begin{equation*}
		f_{ n } ( x )  
		\coloneqq 
		f ( x ) \sum_{ j=1 }^{ N } 
		\sqrt{ \lambda_{ j } } 
		\exp \left( 2 \pi i \theta_{ j } n \inner*{ x }{ \nu } \right)
		e_{ j }.
	\end{equation*}
	Then $ f_{ n } $ is uniformly bounded in $ \lp^{ 2 } $ and we compute for every $ \phi \in \ccinf ( \R^{ d } ) $ by using the convolution theorem for the Fourier transform
	\begin{equation*}
		\fourier ( \phi f_{ n } ) ( \xi )
		= 
		\sum_{ j = 1 }^{ N }
		\sqrt{ \lambda_{ j } } \hat{ \phi } \ast \hat{ f } ( \xi - \theta_{ j } n \nu ) e_{ j }. 
	\end{equation*}
	If $ \phi_{ 1 }, \phi_{ 2 } \in \ccinf ( \R^{ d } ) $ and $ \psi $ is continuous and $ 0 $-homogeneous, then we have for all $ i, j  \in \{ 1, \ldots , N \} $ that
	\begin{equation*}
		\int_{ \R^{ d } }
		\fourier \left( \phi_{ 1 } f_{ n }^{ i } \right) ( \xi ) 
		\overline{ \fourier \left( \phi_{ 2 } f_{ n }^{ j } \right) ( 
			\xi ) 
			\psi ( \xi )}
		\dd{ \xi }
		=
		\sqrt{ \lambda_{ i } } \sqrt{ \lambda_{ j } } 
		\int_{ \R^{ d } }
		\hat{\phi_{ 1 } } \ast \hat{ f } ( \xi )
		\overline{ \hat{ \phi_{ 2 } } \ast \hat{ f } ( \xi - ( \theta_{ 
				j } - \theta_{ i } ) n \nu ) 
			\psi ( \xi + \theta_{ i } n \nu )} 
		\dd{ \xi }.
	\end{equation*}
	By the dominated convergence theorem and the $ 0 $-homogeneity of $ \psi $, we have that $ \hat{ \phi } \ast \hat{f} \psi ( \cdot + \theta_{ i } n \nu ) $ converges to $ \hat{ \phi } \ast \hat{f} \psi ( \nu ) $ in $ \lp^{ 2 } $. Moreover we have the weak convergence
	\begin{equation*}
		\hat{ \phi_{ 2 } } \ast \hat{ f } ( \cdot - ( \theta_{ j } - \theta_{ i } ) n \nu ) \rightharpoonup 
		\begin{cases*}
			\hat{ \phi_{ 2 } } \ast \hat{ f }, & \text{ if } i = j 
			\\
			0, &\text{ else}. 
		\end{cases*}
	\end{equation*}
	Combining weak and strong convergence thus yields by applying the Fourier convolution theorem again that
	\begin{align*}
		\lim_{ n \to \infty }
		\int_{ \R^{ d } }
		\fourier \left( \phi_{ 1 } f_{ n }^{ i } \right) ( \xi ) 
		\overline{ \fourier \left( \phi_{ 2 } f_{ n }^{ j } \right) ( 
			\xi ) 
			\psi ( \xi )}
		\dd{ \xi }
		& =
		A_{ i j }
		\int_{ \R^{ d } }
		\hat{ \phi_{ 1 } } \ast \hat{ f }
		\overline{ \hat{ \phi_{ 2 } } \ast \hat{ f } }
		\dd{ \xi }
		\overline{
			\psi ( \nu )}
		\\
		& = 
		A_{ i j }
		\int_{ \R^{ d } }
		\phi_{ 1 } ( x ) \overline{ \phi_{ 2 } ( x ) } 
		\abs{ f ( x ) }^{ 2 }
		\dd{ x }
		\overline{\psi( \nu )} .
	\end{align*}
	This finishes the construction in the case where $ A $ is diagonal. 
	
	If $ A $ is any hermitian non-negative matrix, then we find a unitary matrix $ U $ and a diagonal matrix $ D $ with ordered non-negative eigenvalues $ \lambda_{ 1 } , \ldots , \lambda_{ N } $ such that $ A = U D \overline{ U }^{ \top } $. 
	In this case defining 
	\begin{equation*}
		f_{ n } ( x ) 
		= 
		f( x )
		\sum_{ j = 1 }^{ N }
		\sum_{ l = 1 }^{ N }
		U_{ j l } \sqrt{ \lambda_{ l } }
		\exp( 2 \pi  i \theta_{ l } n \inner*{x }{ \nu } )
		e_j
	\end{equation*}
	yields by a similar computation that $ f_{ n } $ generates the $ \Hm $-measure $ \mu $.
\end{example}

\begin{example}
	\label{ex:h_measure_generated_by_sin_cos}
	For applications to arrays of dislocation loops, we are interested in $ \Hm $-measures generated by real-valued sequences. We first note that if $ g $ is a real-valued $ \lp^{ 2 } $-function, we have the identity $ \fourier ( g ) ( \xi ) = \overline{ \fourier ( g ) ( - \xi ) } $. 
	Using this and partitioning the test functions $ \phi_{ 1 } $ and $ 
	\phi_{ 2 } $ into their real and imaginary part yields that if a 
	real-valued sequence $ f_{ n } $ generates an $ \Hm $-measure $ \mu 
	$, then for all $ \phi_{1 }, \phi_{ 2 } \in \cont_{ 0 } ( \R^{ d } 
	) $ and $ \psi \in \cont ( \R^{ d  } \setminus \{ 0 \} ) $, which 
	is 
	$ 0 $-homogeneous, we have
	\begin{equation*}
		\int_{ \R^{ d } \times \Sph^{ d - 1 } }
		\phi_{ 1 } ( x ) \overline{ \phi_{  2} ( x ) \psi ( \xi ) } 
		\dd{ \mu ( x , \xi ) }
		=
		\int_{ \R^{ d } \times \Sph^{ d - 1 } }
		\phi_{ 1 } ( x ) \overline{ \phi_{  2} ( x )  \frac{\psi ( \xi 
				)+ \psi ( - \xi ) }{ 2  } }
		\dd{ \mu ( x , \xi ) }.
	\end{equation*}
	We recall the trigonometric identities
	\[
	\sin(x)=\frac{\exp(ix)-\exp(-ix)}{2i},
	\qquad
	\cos(x)=\frac{\exp(ix)+\exp(-ix)}{2}.
	\]
	Then, for $f\in\lp^2(\R^d)$ and $\nu\in\Sph^{d-1}$,
	the sequence
	\[
	f_n(x)\coloneqq f(x)\sin(n\inner*{x}{\nu})
	\]
	generates the $\Hm$-measure
	\[
	\mu=\abs{ f}^2\dd{\lm^d}
	\otimes\frac{\delta_\nu+\delta_{-\nu}}{4}.
	\]
\end{example}

\begin{remark}
	One can in fact show that every hermitian non-negative measure $ \mu $ in the sense specified in \Cref{thm:H_measures} is an $ \Hm $-measure, meaning that there exists some sequence converging weakly to zero in $ \lp^{ 2 } $ which generates $ \mu $. The first step to proving this is that the space of $ \Hm $-measures is convex (with a similar argument as in \Cref{ex:h_measures_of_simple_form}) and closed under weak convergence (by a diagonal sequence argument). Furthermore it holds that linear combinations of Diracs are dense in the space of Radon measures on $ \Sph^{ d - 1 } $. Moreover finite Radon measures on $ \R^{ d } $ can be approximated by linear combinations of the Lebesgue measure restricted to dyadic cubes. Finally we can use the Disintegration of Radon measures \cite[Thm.~2.28]{ambrosio_fusco_pallara_functions_of_bv_and_free_discontinuity_problems} to conclude the claim. Since we do not make use of this statement, the full detailed proof is beyond the scope of this paper.
\end{remark}

Using \Cref{thm:H_measures}, we are able to conclude the proof of \Cref{prop:weak_long_energy_convergence}.
\begin{proof}[Proof of \Cref{prop:weak_long_energy_convergence}]
	Using \Cref{lem:weak_long_energy_rewritten}, we find that for a non-relabelled subsequence, we have
	\begin{equation}
		\label{eq:macroscopic_equation}
		\lim_{ \rho \to 0 } \interactionEnergy ( \disdens_{ \rho } )
		=
		\lim_{ \rho \to 0 } 
		\frac{ 1 }{ 2 } 
		\inner*{ \fourier \widetilde{ \of_{ \rho} } }{ \Psi \fourier \widetilde{ \of_{ \rho} } }_{ \lp^{ 2 } }.
	\end{equation}
	We first reduce the right-hand side to the case that $ \widetilde{ \of_{ \rho} } $ converges weakly to zero. In fact, by recalling the uniform $ \lp^{ 2 } $-bound (\ref{eq:l2_bound_weaker}) and passing to a non-relabelled subsequence , let $ \of $ be the weak $ \lp^{ 2}$-limit of $ \widetilde{ \of_{ \rho} } $ and rewrite by the symmetry of $ \Psi $ the right-hand side of equation (\ref{eq:macroscopic_equation}) as
	\begin{equation*}
		\inner*{ \fourier \widetilde{ \of_{ \rho} } }{ \Psi \fourier \widetilde{ \of_{ \rho} } }_{ \lp^{ 2 }}
		=
		\inner*{ \fourier( \widetilde{ \of_{ \rho} }-  \of  )}{ \Psi \fourier( \widetilde{ \of_{ \rho} }-  \of ) }_{ \lp^{ 2 }} 
		-
		\inner*{\fourier \of  }{ \Psi \fourier  \of }_{ \lp^{ 2 }}
		+
		2 \inner*{ \fourier \widetilde{ \of_{ \rho} }}{ \Psi \fourier  \of }_{ \lp^{ 2 } }.
	\end{equation*}
	Notice that by weak convergence, the sum of the last two summands converges to $ \inner*{\fourier \of }{ \Psi \fourier  \of }_{ \lp^{ 2 }} $, so we are only left with taking care of the first summand, which concerns a sequence converging weakly to zero. By recalling that $ \widetilde{ \of_{ \rho} } $ has compact support since we only placed dislocation loops inside the bounded set $ \Omega $ and applying \Cref{thm:H_measures}, we thus obtain for a non-relabelled subsequence that
	\begin{equation*}
		\lim_{ \rho \to 0 } 
		\frac{ 1 }{ 2 }
		\inner*{ \fourier \widetilde{ \of_{ \rho} } }{ \Psi \fourier \widetilde{ \of_{ \rho} } }_{ \lp^{ 2 } }
		=
		\frac{ 1 }{ 2 }
		\inner*{ \fourier \of }{ \Psi \fourier \of }_{ \lp^{ 2 } }
		+
		\frac{ 1 }{ 2 }
		\sum_{ i, j = 1 }^{ N }
		\int_{ \R^{ 3 } \times \Sph^{ 2 } }
		\Psi_{ i j } ( \xi )
		\dd{ \mu_{ i j } ( x, \xi ) },
	\end{equation*}
	where $ (\mu_{ i j })_{i, j = 1 , \ldots, N  } $ is the $ \Hm $-measure generated by $ \widetilde{\of_{ \rho }} - \of $. 
\end{proof}

\subsection{Weak-Short Oscillations}
\label{sct:weak_short}

To analyze weak-short oscillations, we take another approach inspired by techniques from \cite{firoozye_93_homogenization_on_lattices}.
First we treat the case in which there are only weak-short oscillations in the sense of \Cref{def:weak_short}, and we later decompose the sequence into a part which has only weak-long oscillations, and one which has only weak-short oscillations.

To get started, we recall the definition of the \emph{reciprocal lattice}. Given the Bravais lattice $ \lattice_{ 1 } = \{ \sum_{ i = 1 }^{3 } \nu^{ i } v_i \colon \nu^{ i } \in \Z \} $, we can always choose unique $ w_{ i } \in \R^{ 3 } $, $ i \in \{1,2,3\} $, such that $ \inner*{v_i}{w_j} = \delta_{ i j } $ for all $ i, j \in \{1,2,3\} $. We denote by  
\begin{align*}
	\lattice_{ 1 }^{ \ast } 
	&\coloneqq 
	\left\{ \sum_{ i = 1 }^{ 3 } \mu^{ i } w_{ i } \colon \mu_{ i } \in \Z \text{ for all } i \in \{1,2,3 \}\right\}
	\shortintertext{the reciprocal lattice with unit cell}
	U^{ \ast } 
	&\coloneqq 
	\left\{ \sum_{ i = 1 }^{ 3 } \lambda^{ i } w_{ i } \colon \lambda^{i } \in [0,1) \text{ for all }i \in \{1,2,3\} \right\}.
\end{align*}
Given a summable function on the lattice $ f \in \lps^{ 1 } ( \lattice_{ 1 } ; \C) $, we define the \emph{inverse discrete Fourier series of $ f $} as
\begin{align}
	\notag
	\idfs{f} \colon U^{ \ast } & \to \C
	\\
	\label{eq:inverse_discrete_fourier_series}
	\xi & \mapsto \sum_{ m \in \lattice_{ 1 } } 
	f ( m ) \exp( -2 \pi i \xi \cdot m ).
\end{align}
The name is motivated by the fact that due to the Fourier inversion 
formula, $ \idfs{f} $ is the Fourier series of the inverse Fourier 
transform of $ f $.
We note that $ \idfs{f} $ is periodic on $ U^{ \ast } $, and when it is 
convenient for us, we will interpret $ \idfs{f} $ as a function on $ 
\R^{ 3 } $.

By definition of the reciprocal unit cell, $ \idfs{f} $ is periodic on $ U^{ \ast } $, and satisfies for $ f, g \in \lps^{1 } ( \lattice_{ 1 }) \cap \lps^{ 2 } ( \lattice_{ 1 }) $ the Plancherel identity
\begin{equation}
	\label{eq:plancherel_idfs}
	\inner*{f}{g}_{ \lps^{ 2 } ( \lattice_{ 1 } ) }
	=
	\inner*{\idfs{f}}{\idfs{g}}_{ \lp^{ 2 } ( U^{ \ast})}.
\end{equation}
By density of $ \lps^{ 1 }(\lattice_{ 1 }) \cap \lps^{ 2 } ( \lattice_{ 1 }) $ in $ \lps^{ 2 } (\lattice_{ 1 }) $, we extend the definition of the inverse discrete Fourier series to $ \lps^{ 2 } ( \lattice_{ 1 } ) $ and preserve the above Plancherel formula.

In order to handle weak-short oscillations, we will use the following rescaling $ \rof_{ \rho } \colon \lattice_{ 1 } \to \R^{ 3 \times 3 } $ of $ \of_\rho $ (see (\ref{eq:def_or_area_function})) defined as
\begin{equation}
	\label{eq:rof_definition}
	\rof_\rho ( x ) \coloneqq \rho^{-3/2} \of_{ \rho } ( \rho x ).
\end{equation} 
The resulting function $ \rof_{ \rho } $ is then an element of $ \lps^{ 2 } ( \lattice_1 ) $ since
\begin{equation}
	\label{eq:l2_bound_h_rho}
	\sum_{ y \in \lattice_{ 1 } } 
	\abs{ \rof_{ \rho} ( y ) }^{ 2 }
	=
	\sum_{ x \in \lattice_{ \rho } }
	\frac{1}{\rho^{ 3 } } \abs{ \of_{ \rho } ( x ) }^{ 2 }
	\leq
	\sum_{ x \in \lattice_{ \rho } }
	\frac{1}{\rho^{ 3 } }
	\left( \int_{ S_{ \rho } (  x ) } \abs{ b ( z ) } \dd{ \hm^{ 2 } ( z ) } 
	\right)^2,
\end{equation}
which is bounded by assumption \ref{item:l2BoundOnArea}.
Moreover, we introduce a cutoff which also provides exponential decay at infinity.
Let $ \psi \in \cont^{\infty} ( \R^{ 3} ; [0,1]) $ be radially symmetric such that $ \psi (x  ) = 0 $  for $ \abs{x } \leq \omega/2 $ and $ \psi ( x ) = 1 $ for $ \abs{ x } \geq \omega $, where $ \omega > 0 $ is chosen such that $ \lattice_{ 1 }\setminus\{0\} \subseteq \R^{ 3 } \setminus B_{ \omega } ( 0 )$. We then set
\begin{equation*}
	M_\eps ( x ) \coloneqq \psi ( x ) \exp ( - \eps \abs{ x } ) M ( x ),
\end{equation*}
which is a Schwartz function on $ \R^{ 3 } $ mapping into $ \R^{ 3 \times 3 \times 3 \times 3 } $.
Sometimes we also interpret $ M_\eps $  as a function on $ \lattice_{ 1 } $.
We then obtain the following.
\begin{lemma}
	\label{lem:firoozye_rewriting}
	We have
	\begin{equation*}
		\frac{1}{2}
		\sum_{ x \neq y \in \lattice_{ \rho } }
		\of_\rho ( x ) \colon M ( x - y ) \of_\rho ( y )
		=
		\lim_{ \eps \to 0 }
		\frac{1}{2}
		\int_{ U^{ \ast }}
		\idfs{M_\eps} ( \xi ) \cdot 
		\left( 
		\idfs{\rof_{\rho}} ( \xi ) \otimes  \overline{ \idfs{\rof_{\rho}} ( \xi )}
		\right)
		\dd{ \xi }.
	\end{equation*}
\end{lemma}
\begin{proof}
	First we note that both sides are well-defined since $ \of_{ \rho } 
	$ and $ \rof_{ \rho } $ have compact support.
	By using the $ - 3 $-homogeneity of $ M $ and the definition of $ \rof_{ \rho} $, we have
	\begin{align}
		\notag
		\frac{1}{2}
		\sum_{ \substack{x, y \in \lattice_{ \rho } \\ x \neq y } }
		\of_\rho ( x ) \colon M ( x - y ) \of_\rho ( y )
		& =
		\frac{1}{2}
		\sum_{ \substack{x, y \in \lattice_{ 1 } \\ x \neq y } }
		\rho^{-3/2} \of_\rho ( \rho x ) \colon M (x-y) \rho^{-3/2} \of_\rho (\rho y )
		\\
		\label{eq:firoozye_sum}
		& =
		\frac{1}{2}
		\sum_{ \substack{x, y \in \lattice_{ 1 } \\ x \neq y } }
		\rof_\rho ( x ) \colon M ( x - y ) \rof_\rho ( y ),
	\end{align}
	We note that due to the compact support of $ \rof_{ \rho } $ and 
	the pointwise convergence of $ M_\eps $ to $ M $, we have
	\begin{equation*}
		\frac{1}{2}
		\sum_{ \substack{x, y \in \lattice_{ \rho }\\x \neq y } }
		\of_{ \rho } (x)\colon M ( x - y ) \of_{ \rho } (y)
		=
		\lim_{ \eps \to 0 }
		\frac{1}{2}
		\sum_{ x, y \in \lattice_{ 1 } }
		\rof_\rho ( x ) \colon M_\eps ( x - y ) \rof_\rho ( y ).
	\end{equation*}
	Applying the Plancherel identity (\ref{eq:plancherel_idfs}) yields 
	that the term on the right hand side is equal to
	\begin{equation}
		\label{eq:sum_plancharel_applied}
		\lim_{ \eps \to 0 }
		\frac{1}{2}
		\int_{U^{\ast}}\idfs{\rof_{\rho}} (\xi ) \colon \overline{\idfs{M_\eps\ast\rof_{ \rho}} ( \xi )} 
		\dd{ \xi }
		=
		\lim_{ \eps \to 0 }
		\frac{1}{2}
		\int_{ U^{ \ast }}
		\idfs{\rof_{\rho}} ( \xi ) \colon \overline{\idfs{M_\eps} ( \xi ) \idfs{\rof_{\rho}} ( \xi )}
		\dd{ \xi }
	\end{equation}
	which is what we wanted to show since $ \overline{\idfs{M_\eps} } = \idfs{M_\eps} $ due to the symmetry of $ M $.
\end{proof}
We now have to show that as $ \eps \to 0 $, $ \idfs{ M_{ \eps } } $ indeed has a limit. These limits have been investigated in \cite{wainger_65_special_trigonometric_series_in_kd}, and we add the observation of uniform boundedness in $ \lp^{ \infty } $. 
\begin{proposition}
	\label{prop:Feps_behaviour}
	Let $n \geq 3 $ and $ K \in C^{ \infty } ( \R^{ n }\setminus \{0\} ) $ be a $ -n $-homogeneous smooth function such that
	\begin{equation*}
		\int_{ \Sph^{ n - 1 } } K ( x ) \dd{ \hm^{ n - 1 } }
		=
		0.
	\end{equation*}
	Let the cutoff $ \psi $ be as above and define
	\begin{equation*}
		F_\eps ( \xi )
		\coloneqq 
		\sum_{ x \in \lattice_{ 1 } }
		\psi ( x ) \exp ( - \eps \abs{ x}) K ( x ) \exp( - 2 \pi i x \cdot \xi ) .
	\end{equation*}
	Then the following hold.
	\begin{enumerate}[label=(\roman*)]
		\item \label{item:Feps_ptw_convergence}
		For all $ \xi \notin \lattice_{ 1 }^{\ast} $, $ \lim_{ \eps \to 0} F_\eps ( \xi ) \eqqcolon F ( \xi ) $ exists and $ F \in C^{ \infty } ( \R^{ n } \setminus \lattice_{ 1 }^{\ast}) $ is periodic with respect to the reciprocal lattice $ \lattice_{ 1 }^{ \ast } $.
		\item \label{item:Feps_linfty_bound}
		$ F_\eps $ is dominated by an $ \lp^{ \infty } $-function on $ U^{*} $ as $ \eps \to 0 $.
	\end{enumerate}
\end{proposition}
\begin{proof}
	By the exponential decay of $ \exp ( - \eps \abs{ x } ) $, we can apply the Poisson summation formula (see \cite[Cor.~7.2.6]{stein_weiss_introduction_to_fourier_analysis_on_euclidean_spaces}) to obtain that
	\begin{equation}
		\label{eq:Feps_poisson_applied}
		F_\eps ( \xi )
		=
		\sum_{ \zeta \in \lattice_{ 1 }^{ \ast } }
		\fourier ( \psi ( x ) \exp ( - \eps \abs{x} )  K ( x ) ) ( 
		\zeta + \xi  ).
	\end{equation}
	As in \cite[Thm.~2]{wainger_65_special_trigonometric_series_in_kd}, we decompose $ K $ into its spherical harmonics via
	\begin{equation}
		\label{eq:K_disintegrated_into_spherics}
		K ( x )
		=
		\sum_{ l = 0}^{ \infty }
		\sum_{ k = 1 }^{ b_l }
		a_{ k }^{(l)} 
		\abs{x}^{-n}
		Y_{ k }^{ ( l ) } ( x/\abs{x} ),
	\end{equation}
	where $b_l $ is the dimension of the spherical harmonics of degree $ l $, $ a_k^{ (l ) } \in \C $, and $ Y_k^{ (l ) } $ is a spherical harmonic of degree $ l $, which form an orthonormal basis of $ \lp^{ 2 } ( \Sph^{ n - 1 } ; \hm^{ n - 1 } ) $. 
	
	We first note that due to $ K $ having zero mean, we must have $ 
	a_1^{ ( 0 ) } = 0 $ since by 
	\cite[Cor.~4.2.4]{stein_weiss_introduction_to_fourier_analysis_on_euclidean_spaces},
	we have $ \int_{ \Sph^{ n-1 } } Y_{ k }^{(l) } \dd{ \hm^{ n -1  } 
	} = 0 $ for all $ l \geq 1 $ and $ 1 \leq k \leq b_l $. Moreover 
	the coefficients $ a_k^{ (l ) } $ decay fast by 
	\Cref{lem:spherical_harmonics}.
	This decay combined with equations (\ref{eq:Feps_poisson_applied}) and (\ref{eq:K_disintegrated_into_spherics}) yields that
	\begin{equation*}
		F_\eps ( \xi ) 
		=
		\sum_{ \zeta \in \lattice_{ 1 }^{\ast} }
		\sum_{ l, k } a_{ k }^{ (l ) } \fourier \left( \psi( x ) \exp ( 
		- \eps \abs{ x } ) \abs{x}^{-n} Y_{ k }^{ (l ) } ( x/\abs{x}) 
		\right) ( \zeta + \xi ).
	\end{equation*}
	
	Now we prove \ref{item:Feps_linfty_bound}.
	Since $ \psi $ is radially symmetric, we have by \cite[Thm.~2.6.1]{bochner_55_harmonic_analysis_and_probability} that
	\begin{align*}
		& \abs{\fourier ( \psi ( x ) \exp ( - \eps \abs{x} ) 
			\abs{x}^{-n} Y_{k}^{ (l)} ( x / \abs{x}) ) ( \xi )  }
		\\
		\leq{} &
		C \abs{\xi}^{(2-n)/2}
		\abs{ Y_{k}^{ ( l ) } \left(\frac{\xi}{\abs{ \xi } } \right) }
		\int_{ 0 }^{ \infty }
		\psi ( t )  \exp ( - \eps t ) t^{-n} t^{n/2}
		J_{l + (n-2)/2} ( 2 \pi \abs{\xi} t )
		\dd{ t }, 
	\end{align*}
	where $ J_\alpha $ denotes the Bessel function with parameter $ \alpha \in \R $. This already yields a uniform $ \lp^{ \infty } $-bound close to zero: We know that as $ h \to 0 $, $ J_\alpha ( h ) \lesssim h^{ \alpha } $, and as $ h \to \infty $, $ J_\alpha ( h ) \lesssim h^{ -1/2 } $ (see \cite[(9.1.7),(9.1.61)]{abramowitz_stegun_64_handbook_of_math_functions}). 
	Assuming for ease of notation that $ \psi = 0 $ on $(0,1) $, we get that as $ r \to 0 $,
	\begin{align*}
		& r^{(2-n)/2}\int_{ 0 }^{ \infty } \psi ( t ) \exp ( - \eps t ) t^{ -n/2 }
		J_{ l + (n-2)/2} ( 2 \pi r t) 
		\dd{ t }
		\\
		\lesssim{}&
		r^{(2-n)/2}
		\int_{ 1 }^{ r^{-1} } 
		t^{-n/2} (rt)^{l+ (n-2)/2}
		\dd{ t }
		+
		r^{(2-n)/2}
		\int_{ r^{-1}}^{ \infty }
		t^{-n/2} (rt)^{-1/2}
		\dd{ t }
		\\
		={} &
		r^{l} \int_{ 1 }^{r^{-1}} t^{l-1} \dd{ t }
		+
		r^{(1-n)/2}\int_{r^{-1}}^{ \infty } t^{-(n+1)/2}
		\dd{ t } \lesssim 1.
	\end{align*}
	Moreover we can apply \cite[Thm.~1]{wainger_65_special_trigonometric_series_in_kd} to also deduce a suitable decay at infinity, namely that for every integer $ m \in \N $,
	\begin{equation*}
		\fourier ( \psi(x) \exp ( - \eps \abs{x}) \abs{x}^{-n} 
		Y_{k}^{(l)} ( x/\abs{x} ) ) ( \xi ) \in 
		\mathcal{O}(\abs{\xi}^{-m} )
		\text{ uniformly in }\eps > 0 \text{ as }\abs{\xi } \to \infty.
	\end{equation*}
	Combining the decay at infinity with the boundedness at zero, the decay of the coefficients $a_k^{(l)} $ and the dimensional bound $ d_l \lesssim l^{n-2} $ (see \Cref{lem:spherical_harmonics}), we get that $ F_\eps $ is uniformly bounded in $ \lp^{ \infty } ( U^{ \ast } ) $ as $ \eps \to 0 $.
	
	The claim \ref{item:Feps_ptw_convergence} is now an immediate 
	consequence of equations (\ref{eq:Feps_poisson_applied}), 
	(\ref{eq:K_disintegrated_into_spherics}) and the decay of the 
	coefficients $a_k^{(l)} $ given by  \Cref{lem:spherical_harmonics} 
	and 
	\cite[Thm.~1,i),ii)]{wainger_65_special_trigonometric_series_in_kd}.
\end{proof}

Applying \Cref{lem:firoozye_rewriting},  \Cref{prop:Feps_behaviour} and, denoting for $ \xi \in U^{ \ast } \setminus \{ 0 \} $,
\begin{equation*}
	\idfs{M} ( \xi ) 
	\coloneqq 
	\lim_{ \eps \to 0 }
	\idfs{M_\eps} ( \xi ),
\end{equation*}
we obtain that
\begin{equation}
	\label{eq:int_energy_via_Wigner}
	\frac{1}{2 } \sum_{ \substack{x, y \in \lattice_{ \rho } \\ x \neq y } }
	\of_{ \rho } ( x ) \colon M ( x - y ) \of_{ \rho } ( y )
	=
	\frac{1}{2}
	\int_{ U^{ \ast } }
	\idfs{M} ( \xi ) \cdot 
	\left( 
	\idfs{\rof_{\rho}} ( \xi ) \otimes \overline{ \idfs{\rof_{\rho}} ( \xi ) } 
	\right)
	\dd{ \xi }
	\eqqcolon 
	\frac{1}{2}
	\int_{ U^{ \ast } }
	\idfs{M} ( \xi )
	\cdot
	\dd{ \mu_{ \rho } ( \xi ) }.
\end{equation}
Here the \emph{discrete Wigner measure} is defined as 
\begin{equation}
	\label{eq:wigner_measure}
	\mu_{ \rho } 
	\coloneqq 
	\idfs{\rof_{ \rho } } ( \xi ) 
	\otimes 
	\overline{
		\idfs{\rof_{ \rho } } ( \xi )
	}
	\dd{ \xi }
\end{equation}
Due to $ \idfs{M} $ potentially not being continuous in $ 0 $, passing to the limit $ \rho \to 0 $ in the above equation is not immediate after having identified a weak limit $ \mu_\rho \rightharpoonup \mu $. In fact we can show that in the case of weak-long oscillations, the Wigner measures $ \mu_\rho $		
concentrate on the set $\{0\} $ as $ \rho \to 0 $. This is the content of the following Lemma, see also \cite[Lem.~3]{firoozye_93_homogenization_on_lattices}.
\begin{lemma}
	Assume that the sequence $ \of_{ \rho } $ has only 
	long-range oscillations in the sense of \Cref{def:weak_long}. If 
	the associated Wigner measures $ \mu_{ \rho } $ converge weakly to 
	some periodic Radon measure $ \mu $ on $ U^{ \ast} $, then 
	\begin{equation*}
		\mu ( U^{ \ast } \setminus \{ 0 \} )
		=
		0.
	\end{equation*}
\end{lemma}

\begin{proof}
	Let $ \varphi \in \ell^{1} ( \lattice_{ 1 }; [0,\infty) ) $ have 
	finite support with $ \sum_{ x \in \lattice_{ 1 } } \varphi ( x ) = 
	1 $. As in (\ref{eq:def_extension_function}), we denote the 
	piecewise constant extension of a function $ \rho^{-3 } g $ on the 
	lattice $ \lattice_{ \rho } $ to $ \R^{ 3 } $ by $ \tilde{g} $. 
	Moreover we define $ \varphi^{ (\rho ) } ( x ) \coloneqq \varphi ( 
	x/\rho ) $ as a function on $ \lattice_{ \rho } $. We first note 
	that
	\begin{align*}
		\varphi \ast \rof_{ \rho } ( x ) 
		=
		\rho^{-3/2} \sum_{ y \in \lattice_{ \rho } }
		\of_{ \rho } ( \rho x - y ) \varphi ( y/\rho )
		=
		\rho^{-3/2} \of_{ \rho } \ast \varphi^{(\rho )} ( \rho x ). 
	\end{align*}
	Using the Plancherel identity, we thus compute that since $ \tr  
	\mu_{\rho }  = \abs{ \idfs{\rof_{ \rho }} ( \xi )}^{ 2 } \dd{ \xi 
	}$, we 
	have
	\begin{align*}
		\int_{ U^{ \ast } }
		\abs{1 - \idfs{\varphi} ( \xi ) }^{2} \dd{ \tr  \mu_{ \rho } ( \xi ) }
		&=
		\norm{ \varphi \ast \rof_{\rho} - \rof_{\rho} }_{ \ell^{2 } ( \lattice_{ 1 } ) }^{2}
		\\
		& =
		\rho^{-3 } \sum_{ x \in \lattice_{ \rho } }
		\abs{ \of_{ \rho } \ast \varphi^{ ( \rho ) } ( x ) - \of_{ \rho } ( x ) }^{ 2 }
		\\
		& =
		\norm{ \widetilde{ \of_{ \rho } \ast \varphi^{ ( \rho ) } } - 
			\widetilde{ \of_{ \rho } } }_{ \lp^{ 2 } ( \R^{ 3 } ) }^2.
	\end{align*}
	The last expression vanishes in the limit $ \rho \to 0 $
	since $ \varphi^{ (\rho ) } $ has finite support and $ \of_{ \rho } $ 
	has only weak-long oscillations.
	By choosing $ \varphi ( x ) = \delta_{ v_{ j } } ( x ) $, where we recall from (\ref{eq:unit_cell}) that $ ( v_{ j })_{ j \in \{1,2,3\}} $ yield a basis for the lattice $ \lattice_{ 1 } $, we note that $ \idfs{ \delta_{ v_{ j } } } ( \xi ) = \exp( - 2 \pi i \xi \cdot v_{ j }) $, and thus
	\begin{equation*}
		\int_{ U^{ \ast } } 
		\abs{ 1 - \exp ( - 2 \pi i \xi \cdot v_{ j } ) }^2 \dd{ \tr \mu 
			( \xi ) }
		= 0
		\quad
		\text{ for all } j \in \{1,2,3\}.
	\end{equation*}
	Since $ \tr \mu  \geq 0 $, this already yields that $ \tr \mu ( U^{ \ast } \setminus \{ 0 \}) = 0 $. Since $ \mu $ is hermitian positive semidefinite as a weak limit of hermitian positive semidefinite measures, this already yields that $ \mu ( U^{\ast} \setminus \{0\}) = 0 $, finishing the proof.
\end{proof}
We thus see that using Wigner measures is not useful for handling long-range oscillations. However, recalling the definition of weak-short oscillations in \Cref{def:weak_short}, we will see that the Wigner measures of those sequences who only exhibit short-range oscillations do not have concentration in zero.
This will be the content of \Cref{lem:weak_short_implies_no_zero}. 
Therefore the limiting energies of such sequences can be fully understood by investigating
\begin{equation*}
	\lim_{ \rho \to 0 }
	\int_{ U^{ \ast } }
	\idfs{M} ( \xi ) \cdot \dd{ \mu_{ \rho } ( \xi ) }.
\end{equation*}

Weak-short convergence can be roughly understood as $ \of_\rho $ only having oscillations on the scale of lattice, and additionally a local average of zero. We first show how this property relates to the rescaled function $ \rof_{ \rho} $ on $ \lattice_{ 1 } $.
To this end, we introduce the additional notation
\begin{equation*}
	\phi_{ \eps } ( \xi )
	\coloneqq 
	\sum_{ \zeta \in \lattice_{ 1 }^{ \ast } }
	\phi \left(\frac{ \xi + \zeta }{\eps }
	\right).
\end{equation*}
Recall from the conventions before \Cref{def:weak_short} the definition 
of 
$ \phi  $, the definition of $ \Phi \coloneqq \fourier^{-1} \phi  $ 
and $ \Phi_\eps 
( x ) = \eps^{ - 3 } \Phi ( x/\eps ) $.
Using the Poisson summation formula, 
\cite[Cor.~7.2.6]{stein_weiss_introduction_to_fourier_analysis_on_euclidean_spaces},
and the action of the Fourier transform on dilations, we note that for 
$ \lambda > 0 $, we have
\begin{equation}
	\label{eq:phi_idfs}
	\idfs{ (\Phi_{\lambda^{-1} } ) } ( \xi )
	=
	\sum_{ x \in \lattice_{ 1 } }
	\lambda^{ 3 } \Phi ( \lambda x ) \exp ( - 2 \pi i x \cdot \xi )
	=
	\sum_{ \zeta \in \lattice_{ 1 }^{ \ast  }}
	\phi \left(\frac{\zeta + \xi }{ \lambda }\right)
	=
	\phi_{ \lambda } ( \xi ).
\end{equation}
The following result is similar to \cite[Rmk.~2]{firoozye_93_homogenization_on_lattices}.
\begin{lemma}
	\label{lem:weak_short_equivalence}
	Assume that $ (\rof_{ \rho})_{ \rho>0 } $ is bounded in $ \lps^{ 2 
	} ( \lattice_{ 1 } ) $. Then the sequence $ \of_{ \rho } $ 
	converges weak-short if and only if
	\begin{equation*}
		\lim_{ \eps \to 0 }
		\lim_{ \rho \to 0 }
		\sum_{ x \in \lattice_{ 1 } }
		\abs{ \Phi_{1/\eps } \ast \rof_{ \rho} ( x ) }^{ 2 }
		=0.
	\end{equation*}
\end{lemma}
\begin{proof}
	Setting $ S_\lambda (f) ( x ) \coloneqq \lambda^{3/2} f ( \lambda x ) $, we note that 
	\begin{equation*}
		\fourier ( S_{ \lambda } (f) ) ( \xi ) = S_{ 1/\lambda } ( \fourier ( f ) ) ( \xi )
	\end{equation*}
	and 
	\begin{equation*}
		(f \ast g) ( x/\lambda )
		=
		S_{ 1/\lambda } ( f ) \ast S_{ 1/\lambda } ( g ) ( x ).
	\end{equation*}
	Thus we have by Plancherel that
	\begin{align}
		\notag
		\int_{ \R^{ 3 } } 
		\abs{ \Phi_{ \rho/\eps } \ast \widetilde{ \of_{ \rho } }}^{ 2 }
		\dd{ x }
		& =
		\rho^{ 3 }
		\int_{ \R^{ 3 } }
		\abs{ S_{ \rho } ( \Phi_{ \rho/ \eps } ) \ast S_{ \rho } ( \widetilde{ \of_{ \rho } })}^{2}
		\dd{ x }
		\\
		\notag
		& =
		\int_{ \R^{ 3 } }
		\abs{ \Phi_{1/\eps } \ast S_{ \rho} ( \widetilde{ \of_{ \rho } }) }^{ 2 } \dd{ x }
		\\
		\label{eq:weak_short_equivalence_1}
		& =
		\int_{ \R^{ 3 } }
		\phi \left(\frac{\xi}{\eps}\right)^{ 2 } 
		\abs{ S_{ 1/\rho } ( \fourier \widetilde{ \of_{ \rho } } )  ( \xi ) }^{ 2 }
		\dd{ \xi }.
	\end{align}
	In order to connect this to $ \rof_{ \rho} $, we note that
	\begin{align*}
		S_{1/\rho } (\fourier \widetilde{ \of_{ \rho } }) ( \xi )
		& =
		\rho^{3/2}
		\int_{ \R^{ n } }
		\widetilde{ \of_{ \rho } } ( \rho y ) \exp ( - 2 \pi i y \cdot \xi ) \dd{ y }
		\\
		& =
		\sum_{ x \in \lattice_{ 1 } }
		\rof_{ \rho } ( x ) \exp ( - 2 \pi i x \cdot \xi )
		\int_{ U } \exp ( - 2 \pi i y \cdot \xi )
		\dd{ y }
		\\
		& =
		\idfs{ \rof_{ \rho } } ( \xi ) \fourier (\chi_{ U } )( \xi ).
	\end{align*}
	Using that $ \lm^{ 3 } ( U ) = 1 $, we note that $ \abs{ \fourier ( 
		\chi_{ U } ) } \leq 1 $.
	Moreover we recall  $ \phi_\eps = \idfs{ ( \Phi_{ \eps^{-1} } )  } 
	$ from equality (\ref{eq:phi_idfs}).
	By exploiting the Plancherel identity and the fact that $ \phi $ is 
	supported in the unit ball, we can estimate 
	(\ref{eq:weak_short_equivalence_1}) by
	\begin{align}
		\notag
		\int_{ \R^{ 3 } }
		\abs{ \Phi_{ \rho/\eps} \ast \widetilde{ \of_{ \rho } }}^{ 2 }
		\dd{ x }
		&=
		\int_{ \R^{ 3 } }
		\phi \left( \frac{ \xi }{ \eps }\right)^{ 2 }
		\abs{ \idfs{ \rof_{ \rho } } ( \xi )}^{ 2 }
		\abs{ \fourier ( \chi_{ U } ) ( \xi ) }^{ 2 }
		\dd{ \xi }
		\\
		\label{eq:weak_short_number_1}
		& \leq
		\int_{ \R^{ 3 } }
		\phi\left(\frac{\xi}{\eps}\right)^{2}
		\abs{ \idfs{\rof_{ \rho } } ( \xi )}^{ 2 }
		\dd{ \xi }
		\\
		& =
		\int_{ U^{ \ast }}
		\phi_{ \eps } ( \xi )^{ 2 } 
		\abs{ \idfs{\rof_{ \rho } } ( \xi )}^{ 2 }
		\dd{ \xi }
		\\
		\notag
		& =
		\int_{ U^{\ast} }
		\abs{ \idfs{ \Phi_{1/\eps } } ( \xi ) }^{ 2 }
		\abs{ \idfs{ \rof_{ \rho } } ( \xi ) }^{ 2 }
		\dd{ \xi }
		\\
		\notag
		& = 
		\sum_{ x \in \lattice_{ 1 } }
		\abs{ \Phi_{1/\eps} \ast \rof_{ \rho } }^{ 2 }
	\end{align}
	The reverse inequality of (\ref{eq:weak_short_number_1}) follows up to a factor $ 2 $ for $ \eps > 0 $ small. This is due to the identity $ \fourier ( \chi_{ U } ) ( 0 ) = 1 $, the continuity of $ \fourier ( \chi_{ U } ) $ at $ 0 $, and because $ \phi $ has support in $ B_{ 1 } ( 0 ) $.
\end{proof}
The next observation is that if $ \of_{ \rho } $ converges weak-short, then the associated Wigner measures $ \mu_{ \rho } $ do not have any concentration at $ 0 $, see also \cite[Lem.~4]{firoozye_93_homogenization_on_lattices}.
\begin{lemma}
	\label{lem:weak_short_implies_no_zero}
	Assume that $ \of_{ \rho } $ converges weak-short and that the associated Wigner measures $ \mu_{ \rho } $ converge weakly to some $ \mu $. Then $ \abs{\mu} \perp \delta_{0 } $.
\end{lemma}
\begin{proof}
	We want to show that $ \abs{\mu }( \{ 0 \}) = 0 $, for which it suffices to show that
	\begin{equation*}
		\lim_{ \eps \to 0 }
		\int_{ U^{ \ast } }
		\phi_{ \eps } ( \xi ) \dd{ \abs{\mu} ( \xi ) }
		=
		0.
	\end{equation*}
	To this end, we note that by the Plancherel identity, it holds
	\begin{align*}
		\int_{ U^{ \ast } }
		\phi_{ \eps } ( \xi )^{2} \dd{ \tr \mu( \xi ) }
		&
		=
		\lim_{ \rho \to 0 }
		\int_{ U^{ \ast } }
		\phi_{ \eps } ( \xi )^2
		\abs{ \idfs{\rof_{ \rho} }( \xi)}^{ 2 }
		\dd{ \xi }
		\\
		& =
		\lim_{ \rho \to 0 }
		\int_{ U^{ \ast } }
		\abs{ \idfs{ (\Phi_{1/\eps } \ast \rof_{ \rho } ) } ( \xi )}^{ 2 }
		\dd{ \xi }
		\\
		& = 
		\lim_{ \rho \to 0 }
		\sum_{ x \in \lattice_{ 1 } }
		\abs{ \Phi_{ 1/\eps } \ast \rof_{ \rho } }^{ 2 }.  
	\end{align*}
	Applying \Cref{lem:weak_short_equivalence}, we thus deduce that $ 
	\tr ( \mu ) ( \{0 \}) = 0 $. Combining this with the fact that $ 
	\mu $ is hermitian and positive semidefinite, we get the claim.
\end{proof}

We conclude this section with the proof of \Cref{prop:weak_short_energy_convergence}

\begin{proof}[Proof of \Cref{prop:weak_short_energy_convergence}]
	By inequality (\ref{eq:l2_bound_h_rho}), we have that $ \rof_{ \rho 
	} $ is bounded in $ \lps^{ 2 } ( \lattice_{ 1 }) $, and thus by the 
	Plancherel identity, the associated Wigner measures $ \mu_{ \rho } 
	$ (see (\ref{eq:wigner_measure})) have uniformly bounded mass. 
	Therefore by passing to a non-relabelled sequence, we may assume that there exists a Radon measure $ \mu_{ \mathrm{W}} $ on $ U^{ \ast } $ such that $ \mu_{ \rho } \rightharpoonup \mu_{ \mathrm{W}} $. By \Cref{lem:weak_short_implies_no_zero}, $ \mu_{ \mathrm{W}} $ has no concentration at zero, and since $ \idfs{M} $ is continuous everywhere except at 0 by \Cref{prop:Feps_behaviour}, we deduce that
	\begin{equation}
		\label{eq:limit_weak_short_conclusion}
		\lim_{ \rho \to 0 }
		\frac{1}{2}
		\int_{ U^{ \ast } } 
		\idfs{M} ( \xi ) \cdot \dd{ \mu_{ \rho } ( \xi ) }
		=
		\frac{1}{2}
		\int_{ U^{ \ast } }
		\idfs{ M } ( \xi ) \cdot \dd{ \mu_{ \mathrm{W}} ( \xi ) }.
	\end{equation}
	Since the limit of the total interaction energy can be rewritten as the left hand side of equation (\ref{eq:limit_weak_short_conclusion}) by \Cref{prop:energy_asymptotics} and \Cref{lem:firoozye_rewriting}, this concludes the proof.
\end{proof}

\subsection{Scale Separation}
\label{sct:scale_separation}

Since we have shown that oscillations on a scale comparable to $ \rho $ 
and oscillations on larger scales require different methods, the next 
step is to divide the sequence $ \disdens_{ \rho } $ (see 
(\ref{eq:disdens_def})), resp. $ \of_{ \rho} $ (see 
(\ref{eq:def_or_area_function})), into its long- and short-range 
oscillatory part. 
With the abbreviation $ e_0 \coloneqq \chi_{\{0\}} \colon \lattice_{ 
	\rho }\to \R $ and recalling $ \Phi_\eps = \eps^{-3} \Phi ( \cdot / 
\eps ) $, we write
\begin{align*}
	\disdens_{ \rho } ( x )
	& =
	\disdens_{ \rho }^{ \mathrm{L}} ( x ) 
	+
	\disdens_{ \rho }^{ \mathrm{S} } ( x ) 
	\\
	& \coloneqq
	\sum_{ y \in \lattice_{ \rho }}
	\rho^{ 3 } \Phi_{ \rho / \delta ( \rho ) } ( x - y )
	(T_{ x- y } )_{\#} \disdens_{ \rho } ( y )  
	+
	\left(\disdens_{ \rho } ( x ) - 
	\sum_{ y \in \lattice_{ \rho } }
	\rho^{ 3 } \Phi_{ \rho / \delta ( \rho ) } ( x - y ) 
	(T_{ x- y } )_{\#} \disdens_{ \rho } ( y )   \right).
\end{align*} 
This corresponds to decomposing the associated oriented surface area 
functions via
\begin{align}
	\label{eq:decomposition_LS}
	\of_{ \rho } ( x )
	&=
	\of_{ \rho }^{\mathrm{ L }} ( x ) + \of_{\rho}^{ \mathrm{ S } } ( x ) 
	\\
	\notag
	&\eqqcolon
	\rho^{3} \Phi_{ \rho/\delta} \ast \of_{ \rho } ( x ) 
	+
	(e_{ 0 } - \rho^{ 3 } \Phi_{ \rho /\delta} ) \ast \of_{ \rho} ( x )
	\\
	\notag
	&=
	\sum_{ y \in \lattice_{ \rho } }
	\rho^{ 3 } \Phi_{ \rho/\delta} ( x- y ) \of_{ \rho } ( y ) 
	+
	\left(\of_{ \rho } ( x ) - 
	\sum_{ y \in \lattice_{ \rho } }
	\rho^{ 3 } \Phi_{ \rho/\delta} ( x- y ) \of_{ \rho } ( y )
	\right).
\end{align}
The parameter $ \delta > 0 $ depends on $ \rho $ and will be chosen such that $ \rho/\delta \to 0 $, but $ \delta \to 0 $. 
The exact choice of $ \delta $ will depend, however, on the sequence $ \of_{ \rho } $ and can not, at least with the methods of our proof, be chosen independently. The factor $ \rho^{ 3 } $ compensates for the discrepancy of summation over $ \lattice_{ \rho } $ and integration over $ \R^{ 3 } $.
We first show that the sequences $ \of_{ \rho }^{\mathrm{L}} $ and $ 
\of_{ \rho }^{ \mathrm{ S } } $ converge weak-long and weak-short, 
respectively.
\begin{lemma}
	\label{lem:decomposition}
	Given a sequence $ (\of_{ \rho })_{ \rho > 0  } $ satisfying the $ 
	\ell^{2} $-boundedness assumption \ref{item:l2BoundOnArea}, there 
	exists a non-relabelled subsequence of $ \rho \to 0 $ and $ \rho 
	\mapsto \delta ( \rho ) $ with $ \delta ( \rho ) \gg \rho $ 
	such that the 
	decomposition (\ref{eq:decomposition_LS}) satisfies
	\begin{enumerate}[label=(\roman*)]
		\item \label{item:weak_long_claim}
		$ \of_{ \rho }^{ \mathrm{ L } } $ converges weak-long in the sense of (\ref{eq:weak_long}),
		\item \label{item:weak_short_claim}
		$ \of_{ \rho }^{ \mathrm{ S } } $ converges weak-short in the sense of (\ref{eq:weak_short_def}) and
		\item \label{item:no_interaction} $ \of_{ \rho }^{\mathrm{ L }} 
		$ and $ \of_{ \rho }^{ \mathrm{ S } } $ have vanishing total 
		interaction energy in the sense that
		\begin{equation*}
			\lim_{ \rho \to 0 }
			\sum_{ \substack{x, y \in \lattice_{ \rho } \\ x \neq y }}
			\of_{ \rho }^{ \mathrm{ S } } ( x )
			\colon 
			M ( x - y )
			\of_{\rho }^{ \mathrm{ L } } ( y ) 
			=
			0.
		\end{equation*}
	\end{enumerate}
\end{lemma} 
\begin{proof}
	Since $ \abs{ \mu_{ \rho } } $ is a bounded sequence of measures on $ U^{ \ast } $, we find a non-relabelled subsequence and Radon measures $ \omega, \mu $ such that $ \abs{ \mu_{ \rho } } \rightharpoonup \omega $ and $ \mu_{ \rho } \rightharpoonup \mu $.
	
	We start by showing that there exists  $\rho \mapsto \delta ( \rho 
	) $ with $ \delta ( \rho ) \to 0 $ as $ \rho \to 0 $ such that
	\begin{align}
		\label{eq:delta_rho_choice_one}
		\lim_{ \rho \to 0 }
		\int_{ U^{ \ast } }
		\phi_{ \delta ( \rho ) } \left( \xi  \right)
		\dd{ \mu_{ \rho } ( \xi ) }
		& = 
		\mu ( \{ 0 \}),
		\\
		\label{eq:delta_rho_choice_squared}
		\lim_{ \rho \to 0 }
		\int_{ U^{ \ast } }
		\phi_{ \delta ( \rho ) } \left( \xi  \right)^2
		\dd{ \mu_{ \rho } ( \xi ) }
		& = 
		\mu ( \{ 0 \}),
		\\
		\label{eq:delta_rho_mass_one}
		\lim_{\rho \to 0 }
		\int_{ U^{ \ast } }
		\phi_{ \delta ( \rho ) } \left( \xi  \right)
		\dd{ \abs{ \mu_{ \rho }  } ( \xi ) }
		& =
		\omega ( \{ 0 \}),
		\\
		\label{eq:delta_rho_mass_squared}
		\lim_{\rho \to 0 }
		\int_{ U^{ \ast } }
		\phi_{ \delta ( \rho ) } \left( \xi  \right)^2
		\dd{ \abs{ \mu_{ \rho } }( \xi ) }
		& =
		\omega ( \{ 0 \}) \quad \text{and}
		\\
		\label{eq:delta_rho_decay}
		\lim_{ \rho \to 0 }
		\frac{ \rho }{ \delta(\rho) }
		& = 0.
	\end{align} 
	In fact, we notice that due to the monotonicity of $ \eps \mapsto 
	\phi ( \xi / \eps ) $, the weak-$\ast $ convergence of $ \mu_{ \rho 
	} $ to $ \mu $ and the fact that $ \phi ( 0 ) = 1 $, we have
	\begin{equation*}
		\mu(\{ 0\})
		=
		\lim_{ \eps \to 0 }
		\int_{ U^{ \ast } }
		\phi \left(\frac{\xi }{ \eps  } \right)
		\dd{ \mu ( \xi ) }
		=
		\lim_{ \eps \to 0 }
		\lim_{ \rho \to 0 }
		\int_{ U^{ \ast } }
		\phi \left(\frac{\xi }{ \eps }\right)
		\dd{ \mu_{ \rho } ( \xi ) }.
	\end{equation*}
	The same computation holds with a square of $ \phi $, and also with $ \abs{ \mu_{ \rho } } $ and $ \omega $ instead of $ \mu_{ \rho } $ and $ \mu $.
	Thus by choosing a suitable diagonal sequence, we find $ \rho 
	\mapsto \delta ( \rho ) $ with $ \rho \ll \delta ( \rho ) $ and 
	such that equations 
	(\ref{eq:delta_rho_choice_one})-(\ref{eq:delta_rho_mass_squared}) 
	are satisfied. 
	
	We note that due to Young's inequality, the resulting long- and 
	short-range parts still satisfy the $ \lp^{ 2 } $-bound 
	(\ref{eq:l2_area}). Moreover since $ \Phi $ does not have compact 
	support, the convolutions might not be compactly support anymore, 
	which has so far been an assumption in our results. 
	However the tail is small since $ \Phi $ is a Schwartz 
	function, and thus can be trivially estimated.
	
	Having chosen $ \delta ( \rho ) $ and splitting via equation 
	(\ref{eq:decomposition_LS}), we now prove 
	\ref{item:weak_long_claim}. 
	In order to prove that $ \of_{ \rho }^{\mathrm{ L }} $ only has 
	long-range oscillations, it suffices to consider $ \xi \in 
	\lattice_{ \rho } $ with $ \abs{ \xi } \lesssim \rho $ in 
	\Cref{def:weak_long} since $ \widetilde{ \of_{ \rho} } $ is
	piecewise constant. 
	We have that
	\begin{align}
		\notag
		\norm{ \widetilde{ \of_{ \rho }^{ \mathrm{ L } } } ( \cdot + \xi ) - 
			\widetilde{ \of_{ \rho }^{ \mathrm{ L } } } }_{ \lp^{ 2 } ( \R^{ 3 } ) }^{2}
		&
		=
		\sum_{ x \in \lattice_{ \rho } }
		\frac{1}{\rho^{ 3 }}
		\abs{  \of_{ \rho }^{ \mathrm{ L } }( x + \xi ) -  \of_{ \rho }^{ \mathrm{ L } }  ( x ) }^2
		\\
		\notag
		& =
		\rho^{3}
		\norm{ \left( \Phi_{ \rho/\delta} ( \xi + \cdot ) - \Phi_{ \rho/\delta } \right) \ast \of_{ \rho } }_{ \ell^{ 2 } ( \lattice_{ \rho } ) }^{ 2 }
		\\
		\label{eq:weak_long_confirm_1}
		& \leq  
		\rho^{ 6 } \norm{ \Phi_{ \rho/\delta} ( \xi + \cdot ) - \Phi_{ \rho/\delta } }_{ \ell^{1}}^{ 2 }
		\rho^{ - 3 } \norm{ \of_{ \rho } }_{ \ell^{ 2 } ( \lattice_{ \rho } ) }^2.
	\end{align}
	By the boundedness assumption, the factor $ \rho^{ - 3 } \norm{ 
		\of_{ \rho } }_{ \ell^{ 2 } ( \lattice_{ \rho } ) }^2 $ stays 
	uniformly bounded. For the other factor, we use the fundamental 
	theorem of calculus and $ \abs{ \xi } \lesssim \rho $ to deduce 
	\begin{align*}
		\rho^{ 6 }\norm{ \Phi_{ \rho/\delta } ( \xi + \cdot ) - \Phi_{ \rho/\delta } }_{ \ell^{ 1 } ( \lattice_{ \rho } ) }^{2 }
		& =
		\delta^{ 6 }
		\left( \sum_{ x \in \lattice_{ 1 } }
		\abs{ \Phi ( \delta x + \frac{ \delta}{\rho } \xi ) - \Phi ( \delta x ) }
		\right)^{ 2 }
		\\
		& \lesssim 
		\delta^{ 6 }
		\left(
		\sum_{ x \in \lattice_{ 1 } } 
		\abs{ 
			\int_0^1 \nabla \Phi \left( \delta x + t \frac{\delta}{\rho } \xi  \right) \cdot \frac{ \delta}{\rho } \xi   \dd{ t } } 
		\right)^{ 2 }
		\\
		& \lesssim 
		\delta^{2}
		\left( \int_{ \R^{ 3 } } \abs{ \nabla \Phi ( x ) } \dd{ x } \right)^{ 2 },
	\end{align*}
	which vanishes as $ \delta \to 0 $.
	
	We now show the claim \ref{item:weak_short_claim}. Denote by $ 
	\rof_{ \rho }^{ \mathrm{ S } } ( x ) \coloneqq \rho^{ -3/2 } \of_{ 
		\rho}^{ \mathrm{ S } } ( \rho x ) $ the rescaling of $ \of_{ \rho 
	}^{ 
		\mathrm{ S } } $ as in (\ref{eq:rof_definition}). 
	It follows that $ \idfs{ \rof_\rho^{ \mathrm{ S } } } ( \xi ) = ( 1- \phi_{ \delta } ( \xi ) ) \idfs{ \rof_{ \rho } } ( \xi ) $.
	By applying the Plancherel identity (\ref{eq:plancherel_idfs}) and 
	(\ref{eq:phi_idfs}), we obtain
	\begin{align}
		\notag
		\lim_{ \eps \to 0 }
		\lim_{ \rho \to 0 }
		\sum_{ x \in \lattice_{ 1 } }
		\abs{ \Phi_{ 1/\eps } \ast \rof_{ \rho }^{ \mathrm{ S } } }^{ 2 }
		& =
		\lim_{ \eps \to 0 }
		\lim_{ \rho \to 0 }
		\int_{ U^{ \ast } }
		\phi_{ \eps } ( \xi )^{ 2 } 
		\abs{ \idfs{\rof_{ \rho }^{ \mathrm{ S } }} ( \xi )}^{ 2 } \dd{ \xi }
		\\
		\notag
		& =
		\lim_{ \eps \to 0 }
		\lim_{ \rho \to 0 }
		\int_{ U^{ \ast } }
		\phi_{ \eps } ( \xi )^{ 2 }
		\abs{ \idfs{ \rof_{ \rho}} ( \xi ) }^{ 2 }
		( 1 - \phi_{ \delta ( \rho ) } )^{ 2 }
		\dd{ \xi }
		\\
		\label{eq:three_sumands}
		& =
		\lim_{ \eps \to 0 }
		\lim_{ \rho \to 0 }
		\int_{ U^{ \ast } }
		\phi_{ \eps } ( \xi )^{ 2 }
		( 1  - 2 \phi_{ \delta ( \rho ) } ( \xi ) + \phi_{ \delta ( \rho ) } ( \xi )^{ 2 } )
		\dd{ \tr \mu_{ \rho } ( \xi ) }.
	\end{align}
	For the first summand, we note that by weak-$\ast$ convergence of $ \mu_{ \rho } $ to $ \mu $, we have
	\begin{equation*}
		\lim_{ \eps \to 0 }
		\lim_{ \rho \to 0 }
		\int_{ U^{ \ast } }
		\phi_{ \eps } ( \xi )^{ 2 }
		\dd{ \tr \mu_{ \rho } ( \xi ) }
		=
		\lim_{ \eps \to 0 }
		\int_{ U^{ \ast } }
		\phi_{ \eps } ( \xi )^{ 2 }
		\dd{  \tr \mu ( \xi ) }
		=
		\tr \mu ( \{ 0 \} ).
	\end{equation*}
	For the second summand, we note that by equation (\ref{eq:delta_rho_choice_one}), we have
	\begin{align*}
		\lim_{ \eps \to 0 }
		\lim_{ \rho \to 0 }
		\abs{ \tr \mu ( \{ 0  \} )
			-
			\int_{ U^{ \ast }}
			\phi_{ \eps } ( \xi )^{ 2 }
			\phi_{ \delta ( \rho ) } ( \xi )
			\dd{ \tr \mu_{ \rho } ( \xi ) }
		}
		& =
		\lim_{ \eps \to 0 }
		\lim_{ \rho \to 0 }
		\abs{ \int_{ U^{ \ast } }
			\phi_{ \delta ( \rho ) } ( \xi )
			-
			\phi_{ \eps } ( \xi )^{ 2 }
			\phi_{ \delta ( \rho ) } ( \xi )
			\dd{ \tr \mu_{ \rho } ( \xi ) }
		}
		\\
		& \leq
		\limsup_{ \eps \to 0 }
		\limsup_{ \rho \to 0 }
		\sup_{ B_{ \delta ( \rho ) } ( 0 ) }
		\abs{ 1 - \phi_{ \eps }^{2} }
		\tr \mu_{ \rho } ( U^{ \ast } )
		=0.
	\end{align*}
	The same computation but using 
	(\ref{eq:delta_rho_choice_squared}) instead of 
	(\ref{eq:delta_rho_choice_one}) shows that the third term converges to $ \mu ( \{ 0 \}) $. We thus conclude that the term (\ref{eq:three_sumands}) is equal to zero, which proves by \Cref{lem:weak_short_equivalence} that $ \of_{ \rho }^{ \mathrm{ S } } $ converges weak-short.
	
	The last part of the proof concerns the claim 
	\ref{item:no_interaction}. 
	We rescale to the unit lattice $ 
	\lattice_{ 1 } $ as in (\ref{eq:firoozye_sum}) and use the 
	Plancherel identity as for the proof of 
	\ref{item:weak_short_claim}. 
	By using that $ ( 1 - \phi_\delta ) \phi_\delta \geq 0 $, we obtain 
	that
	\begin{align*}
		\abs{
			\sum_{ \substack{x, y \in \lattice_{ \rho } \\ x \neq y }}
			\of_{ \rho }^{ \mathrm{ S } } ( x ) \colon M ( x - y ) \of_{ \rho }^{ \mathrm{ L } } ( y )
		}
		& =
		\abs{
			\int_{ U^{ \ast } }
			(1- \phi_{ \delta ( \rho ) } ( \xi ) )
			\phi_{ \delta ( \rho ) } ( \xi )
			\idfs{M} ( \xi ) \colon \dd{ \mu_{ \rho } ( \xi ) }
		}
		\\
		& \leq
		\norm{ \idfs{ M } }_{ \lp^{ \infty } ( U^{ \ast } ) }
		\int_{ U^{ \ast } }
		\phi_{ \delta( \rho )} ( \xi )
		-
		\phi_{ \delta ( \rho ) } ( \xi )^{ 2 }
		\dd{ \abs{ \mu_{ \rho } } ( \xi ) }
		\\
		& =
		\norm{ \idfs{ M } }_{ \lp^{ \infty } ( U^{ \ast } ) }
		\left(
		\int_{ U^{ \ast } }
		\phi_{ \delta ( \rho ) } ( \xi )
		\dd{ \abs{ \mu_{ \rho } } ( \xi ) }
		-
		\int_{ U^{ \ast } }
		\phi_{ \delta( \rho ) } ( \xi )^{2}
		\dd{ \abs{ \mu_{ \rho } } ( \xi ) }
		\right).
	\end{align*}
	This term converges to zero as $ \rho \to 0 $ by equations (\ref{eq:delta_rho_mass_one}) and (\ref{eq:delta_rho_mass_squared}), which finishes the proof.
\end{proof}

We conclude this section with the proof of the main result \Cref{thm:main_theorem}.

\begin{proof}[Proof of \Cref{thm:main_theorem}]
	We choose the decomposition according to \Cref{lem:decomposition}, which preserves the $ \lp^{ 2 } $-boundedness. 
	In view of \Cref{lem:decomposition} \ref{item:no_interaction} and 
	\Cref{prop:energy_asymptotics}, we get that
	\begin{equation*}
		\lim_{ \rho \to 0 }
		\interactionEnergy ( \disdens_{ \rho })
		=
		\lim_{ \rho \to 0 }
		\frac{1}{2}
		\sum_{ \substack{x, y \in \lattice_{ \rho } \\ x \neq y }}
		\of_{ \rho }^{ \mathrm{ L } } ( x ) \colon M ( x- y ) \of_{ \rho }^{ \mathrm{ L } } ( y )
		+
		\lim_{ \rho \to 0 }
		\frac{1}{2}
		\sum_{ \substack{x, y \in \lattice_{ \rho } \\ x \neq y }}
		\of_{ \rho }^{ \mathrm{ S } } ( x ) \colon M ( x- y ) \of_{ \rho }^{ \mathrm{ S } } ( y )
	\end{equation*}
	We then apply \Cref{prop:weak_long_energy_convergence} and \Cref{prop:weak_short_energy_convergence} to the weak-long and weak-short part respectively. By choosing a suitable subsequence, we obtain the desired result. 
\end{proof}
We want to finish by presenting an example on how to compute the splitting and the corresponding $ \Hm $-measure and Wigner measure.
\begin{example}
	Let $ \gamma \coloneqq \Sph^1 \times \{ 0 \} \subseteq \R^{ 3 } $ 
	be a flat loop in $ \R^{ 3 } $, and choose a fixed tangent 
	orientation $ \tau $. For parameters $ \eps > 0 $, the atomic 
	lattice spacing, $ r > 0 $, the radius of the dislocation loops, 
	and 
	$ \rho > 0 $, the spacing of the dislocation loops, consider the 
	following sequence 
	of arrays of dislocation loops:
	\begin{equation*}
		\disdens_{ \rho }
		\coloneqq 
		\frac{1}{ \pi }
		\sum_{ x \in \rho \Z^{ 3 } \cap Q }
		\left( \sin ( x_1 \rho^{ -1/2 } ) + (-1)^{\rho^{-1 } x_1 } \right)
		\eps^{ - 1 } e_{ 1 }
		\otimes 
		\tau \hm^{ 1 } \llcorner_{ x + r \gamma },
	\end{equation*}
	where $ Q $ is the unit cube.
	Moreover we choose $ \eps \ll r \ll \rho $ such that $ \eps^{ - 1 } 
	r^2 = \rho^{ 3 } $. With this choice of parameters, we note that $ 
	\disdens_{ \rho } $ is $ \lp^{ 2 } $-bounded in the sense of 
	\ref{item:l2BoundOnArea} since
	\begin{equation*}
		\sum_{ x \in \rho \Z^{ 3 } \cap Q }
		\frac{1}{\rho^3 } (\eps^{-1} e_1 r^{ 2 } )^2
		=
		\rho^{ -6 } ( \eps^{ -1 } r^{ 2 } )^2,
	\end{equation*}
	which stays uniformly bounded as $ \rho \to 0 $. Moreover we compute that
	\begin{align*}
		\of_{ \rho } ( x ) 
		&= 
		\chi_{ Q } ( x ) ( \sin ( x_1 \rho^{ -1/2 } ) + (-1)^{ 
			\rho^{-1} x_1 } ) \eps^{ -1 } r^{ 2 } e_1 \otimes e_3
		\\
		&=
		\chi_{ Q } ( x ) \left( \sin ( x_1 \rho^{ -1/2 } ) + (-1)^{ \rho^{-1} x_1 } \right) \rho^{3 } e_1 \otimes e_3.
	\end{align*} 
	This sequence splits into its weak-long and weak-short oscillatory part via
	\begin{align*}
		\disdens_{ \rho }^{ \mathrm{ L } } 
		& =
		\frac{1}{\pi}
		\sum_{ x \in \rho \Z^{ 3 } \cap Q }
		\sin ( x_1 \rho^{ -1/2 } ) 
		\eps^{ - 1 } e_{ 1 }
		\otimes 
		\tau \hm^{ 1 } \llcorner_{ x + r \gamma },
		\\
		\disdens_{ \rho }^{ \mathrm{ S } } 
		& =
		\frac{1}{\pi}
		\sum_{ x \in \rho \Z^{ 3 } \cap Q }
		(-1)^{ \rho^{ -1 } x_1 }
		\eps^{ - 1 } e_{ 1 }
		\otimes 
		\tau \hm^{ 1 } \llcorner_{ x + r \gamma }.
	\end{align*}
	That this is consistent with the definition of the splitting in 
	\Cref{thm:main_theorem} can be seen by showing that for  $ \delta ( 
	\rho ) \coloneqq \rho^{1/4} $ , we have that
	\begin{equation*}
		\sum_{ y \in \rho \Z^{ 3 } } \rho^{ 3 } \Phi_{ \rho / \delta } ( x -y ) (-1)^{ \rho^{ - 1 } y_1 } \eps^{ -1 } r^2 e_1 \otimes e_3 
	\end{equation*} 
	goes to zero in $ \lp^{ 2 } ( \R^{ 3 } ) $.
	
	For the long-range oscillatory part, we note that by writing $ \sin 
	( x ) = (\exp ( i x ) - \exp ( - i x ))/(2i) $, the sequence $ 
	\widetilde{ 
		\of_{ \rho} } $ generates, as in 
	\Cref{ex:h_measure_generated_by_sin_cos}, the $ \Hm $-measure
	\begin{equation*}
		\mu_{ \Hm } 
		=
		\lm^{ 3 } \llcorner_{ Q } \otimes \frac{\delta_{ e_1 } + 
			\delta_{ - e_1 } }{ 4 } (e_{ 1 } \otimes e_{ 3 }) \otimes (e_1 
		\otimes e_3 ).
	\end{equation*}
	For the short-range oscillatory part, we first note that
	\begin{equation*}
		\rof_{ \rho } ( x ) =\rho^{ -3/2 } \of_{ \rho } ( \rho x )
		=
		\rho^{ 3/2 }
		\chi_{ \rho^{ -1 } Q } ( x )
		( - 1 )^{ x_1 } e_{ 1 } \otimes e_{ 3 }.
	\end{equation*}
	By using that $ (-1)^{ x_1} = \exp ( 2 \pi i x_1/2) $ for $ x \in 
	\Z^{ 3 } $, we have that the inverse discrete Fourier series of the 
	sequence $ \rof_{ \rho} $ is given by
	\begin{align*}
		\idfs{ \rof_{ \rho} } ( \xi )
		& =
		\sum_{ x \in \Z^3} \rof_{ \rho } ( x ) \exp ( - 2 \pi i x \cdot \xi )
		\\
		& =
		\sum_{ x \in \Z^{ 3 } \cap \rho^{ -1 } Q }
		\rho^{ 3/2 } 
		\exp \left( - 2 \pi i x \cdot \left( \xi - \frac{1}{2} e_1 \right) \right) e_1 \otimes e_3.
	\end{align*}
	Thus the corresponding Wigner measure is given by
	\begin{align*}
		\idfs{ \rof_{ \rho } } ( \xi ) \otimes \overline{ \idfs{ \rof_{ 
					\rho } } ( \xi ) } \dd{ \xi }
		=
		\rho^{ 3 } \abs{ \sum_{ x \in \Z^{ 3 } \cap \rho^{ -1 } Q } 
			\exp \left( - 2 \pi i x \cdot \left( \xi - \frac{1}{2} e_1 
			\right) \right) }^{ 2 } ( e_1 \otimes e_3 ) \otimes ( e_1 
		\otimes e_3 ) \dd{ \xi }.
	\end{align*}
	Using Riemann sums, this measure has the same limit as
	\begin{align*}
		& \rho^{ 3 } \abs{ \int_{ \rho^{ - 1 } Q } \exp \left( - 2 \pi 
			i x \cdot \left( \xi - \frac{1}{2} e_1 \right)\right) \dd{ x } 
		}^{2 }  
		( e_1 \otimes e_3 ) \otimes ( e_1 \otimes e_3 )
		\dd{ \xi }
		\\
		={} & 
		\rho^{ - 3 } \abs{ \fourier \chi_{ Q } \left( \frac{1}{\rho } \left( \xi - \frac{1}{2} e_1 \right)\right)}^{ 2 }
		( e_1 \otimes e_3 ) \otimes ( e_1 \otimes e_3 ) \dd{ \xi }.
	\end{align*}
	This sequence of measures converges to $ \delta_{ \frac{1}{2} e_1 } 
	( e_1 \otimes e_3 ) \otimes ( e_1 \otimes e_3 )$, since for any 
	periodic test function $ \varphi \in C ( Q ) $, we have, due to the 
	square-integrability of $ \fourier \chi_{ Q } $, that
	\begin{equation*}
		\int_{ Q } \varphi ( \xi ) \rho^{ - 3 } \abs{ \fourier \chi_{ Q } \left( \frac{1}{ \rho } \left( \xi - \frac{1}{2} e_ 1 \right) \right) }^2 \dd{ \xi }
		\to 
		\varphi \left( \frac{1}{2} e_1 \right) \int_{ \R^{ 3 } } \abs{ \fourier \chi_{ Q } }^{ 2 } \dd { \xi }
		= 
		\varphi \left( \frac{1}{2} e_1 \right).
	\end{equation*}
	By combining all the above observations and using \Cref{thm:main_theorem}, we thus conclude that
	\begin{equation*}
		\lim_{ \rho \to 0 }
		\interactionEnergy ( \disdens_{ \rho } )
		=
		\frac{1}{2}
		\left(\frac{1}{4} ( \Psi_{1313} ( e_{ 1 } ) + \Psi_{1313} ( -e_{ 1 } ) ) +  \idfs{M}_{1313} \left( \frac{1}{2} e_1 \right)\right).
	\end{equation*}
\end{example}

\section{Relaxation}
\label{sct:relaxation}

We want to close the paper by proving \Cref{cor:negative_interaction}. We start by finding out more about the structure of the $ 0 $-homogeneous function $ \Psi $ appearing in the main result \Cref{thm:main_theorem}.

\subsection{Symmetries}

To start off, we show that using the symmetries of the cubic lattice, we can deduce that the lattice sum (\ref{eq:def_lattice_sum}) is already determined up to a scalar if the lattice has cubic symmetry.
\begin{lemma}
	Assume that the lattice $ \lattice_{ 1 } $ has cubic symmetry in the sense of equation (\ref{eq:cubic_symmetry}) and that $ \eltensor $ is an isotropic elasticity tensor in the sense of equation (\ref{eq:def_isotropic_elasticity_tensor}). Then there exists $ \lambda \in \R $ such that
	\begin{equation}
		\label{eq:structure_of_lattice_sum}
		S =
		\lambda 	
		\begin{pmatrix}
			\begin{pmatrix}
				1 & 0 & 0 \\
				0 & -1/2 & 0\\
				0 & 0 & -1/2 
			\end{pmatrix}
			& 	
			\begin{pmatrix}
				0 &  -1/2& 0\\
				-1/2 & 0 & 0\\
				0 & 0 & 0
			\end{pmatrix}
			&
			\begin{pmatrix}
				0 &  0& -1/2\\
				0 & 0 & 0\\
				-1/2 & 0 & 0
			\end{pmatrix}
			\\
			\begin{pmatrix}
				0 &  -1/2& 0\\
				-1/2 & 0 & 0\\
				0 & 0 & 0
			\end{pmatrix}
			&
			\begin{pmatrix}
				-1/2 & 0 & 0 \\
				0 & 1 & 0\\
				0 & 0 & -1/2 
			\end{pmatrix}
			&
			\begin{pmatrix}
				0 & 0 & 0 \\
				0 & 0 & -1/2\\
				0 & -1/2 & 0
			\end{pmatrix}
			\\
			\begin{pmatrix}
				0 &  0& -1/2\\
				0 & 0 & 0\\
				-1/2 & 0 & 0
			\end{pmatrix}
			&
			\begin{pmatrix}
				0 & 0 & 0 \\
				0 & 0 & -1/2\\
				0 & -1/2 & 0
			\end{pmatrix}
			&
			\begin{pmatrix}
				-1/2 & 0 & 0 \\
				0 & -1/2 & 0\\
				0 & 0 & 1 
			\end{pmatrix}
		\end{pmatrix}
	\end{equation}
\end{lemma}

\begin{proof}
	We first note by the definition of $ S $ in equation (\ref{eq:def_lattice_sum}) and the definition of $ M $ in equation (\ref{eq:formula_M}), we have
	\begin{equation}
		\label{eq:S_relation_to_T}
		S_{ i j i' j'} = 
		\sum_{ k l k' l' }
		\eltensor_{ i j k l }
		\eltensor_{ i' j' k' l' }
		\left(
		\lim_{ n \to \infty }
		\sum_{ x \in \lattice_{ 1 } \cap B_{ n } \setminus \{ 0 \} }
		\partial_{ l l' } K_{ k k'} ( x ) 
		\right).
	\end{equation}
	Let $ f \colon \R^{ 3 } \to \R $ be a twice continuously 
	differentiable function and $ R \in \mathrm{O} ( 3 ) $ . Then
	\begin{equation*}
		\diff^{ 2 } f ( R y )
		=
		R ( \diff^{ 2 } ( f \circ R ) ( y ) ) R^{ \top }.
	\end{equation*}
	We recall that the Fourier transformation of $ K $ is given by 
	(\ref{eq:fourier_K_formula}), and thus satisfies for all rotations 
	$ R $ that $ \fourier K ( R \xi ) = R \fourier K ( \xi ) R^{ \top } 
	$. In particular $ K $ also satisfies this identity since $ R $ is 
	orthogonal.
	Thus the action of $ \diff^{ 2 } K ( R y ) $ on matrices A is described by
	\begin{equation}
		\label{eq:rotational_symmetry_K}
		\diff^{ 2 } K ( R y ) ( A )
		=
		R \diff^{ 2 } K ( y ) ( R^{ \top } A R ) R^{ \top }.
	\end{equation}
	In order to study the lattice sum (\ref{eq:S_relation_to_T}) we first consider
	\begin{equation}
		\label{eq:def_of_T}
		T \coloneqq \lim_{ n \to \infty }
		\sum_{ x \in \lattice_{ 1 } \cap B_{ n } \setminus \{ 0 \} }
		\diff^{ 2 } K ( x ).
	\end{equation}
	Taking any element $ R $ of the point group $ P_{ 1 } $ (see \Cref{def:point_group_and_cubic_symmetry}) reveals that for any matrix $ A $, we have
	\begin{align}
		\notag
		T (A)&=
		\lim_{ n \to \infty }
		\sum_{ x \in R( \lattice_{ 1 } \cap B_{n } \setminus \{ 0 \} ) }
		\diff^{ 2 } K ( x ) ( A )
		\\
		\notag
		& =
		\lim_{ n \to \infty }
		\sum_{ x \in  \lattice_{ 1 } \cap B_{n } \setminus \{ 0 \} }
		\diff^{ 2 }K ( Ry ) ( A )
		\\
		\label{eq:id_for_lattice_sum}
		& =
		R T ( R^{ \top } A R ) R^{ \top }. 
	\end{align} 
	We want to use this equality to show that $ T $ must already be of 
	the form (\ref{eq:structure_of_lattice_sum}). 
	We order the indices by $ (\diff^{ 2 } K)_{ i j k l } \coloneqq 
	\partial_{ i k } K_{ jl } $.
	Throughout this paragraph, let $ i, j, k , l \in \{1,2,3\} $.
	First we insert a reflection $ R_{ i } \coloneqq - e_{ i } \otimes 
	e_{ i } + \sum_{ j \neq i } e_{ j } \otimes e_{ j } $  and a matrix 
	$ A = e_{ k } \otimes e_{ l } $ into the identity 
	(\ref{eq:id_for_lattice_sum}). We note that
	\begin{equation*}
		R_{ i }^{ \top } e_{ k } \otimes e_{ l } R_{ i }
		=
		\sigma ( i k l ) e_{ k } \otimes e_{ l },
	\end{equation*}
	where $ \sigma (i k l ) $ is equal to $ 1 $ if either $ k,l= i $ or $ k,l\neq i $, and $ -1 $ else. We thus obtain that 
	\begin{equation}
		\label{eq:id_for_T_thorugh_reflection}
		T (e_{ k } \otimes e_{ l } )
		= 
		\sigma(i k l ) R_{ i } T(e_{ k } \otimes e_{ l } ) R_{ i }^{ \top}
		\text{ for all } i, k, l.
	\end{equation}
	Note that the linear map 
	\begin{equation*}
		B \mapsto R_{ i } B R_{ i }^{ \top }
	\end{equation*}
	multiplies every entry of $ B $ with a minus one which has \emph{exactly one} index equal to $ i $. Thus the identity (\ref{eq:id_for_T_thorugh_reflection}) can be written
	\begin{equation*}
		T_{ p q k l } = \sigma( i k l ) \sigma( i p q ) T_{ p q k l },
	\end{equation*}
	from which we may deduce that $ T_{ p q k l }$ is $ 0 $ if one 
	of the following two cases holds for some $ i \in \{1,2,3\} $:
	\begin{enumerate}
		\item\label{item:case_1} $ \sigma(ikl)=1$, $ \sigma (ipq)=-1$ : 
		[$ k=l $ or $ k, l \neq i $ ] and [ $ p\neq q $ and ($ p=i $ or 
		$ q= i $)].
		\item\label{item:case_2}
		$ \sigma ( i k l ) = -1 $, $ \sigma ( i p q ) = 1 $:
		[$ k \neq l $ and ($ k= i $ or $ l = i $)] and  [$ p=q $ or $ 
		p, q \neq i $].
	\end{enumerate}
	By a careful case distinction, it follows that
	\begin{equation}
		\label{eq:T_zero_off_diagonal}
		T_{ p q k l } = 0 \text{ if } (p\neq q \text{ or } k \neq l) \text{ and } \{ p,q \} \neq \{ k , l \}.
	\end{equation}
	From equation (\ref{eq:T_zero_off_diagonal}) we deduce that the only non-zero entries of $ T $  have to be of the form
	\begin{equation*}
		T_{iiii}, T_{ i i j j }, T_{ i j i j }, T_{ j i i j }.
	\end{equation*}

	Next up we want to consider elements $ R_{ \sigma } $ of the point group corresponding to permutations $ \sigma \in S_{ 3 } $ via $ R_{ \sigma } ( e_{ i } ) = e_{ \sigma ( i ) } $ for all $ i \in \{1,2,3\} $. 
	We first note that
	\begin{align*}
		R_{ \sigma } e_{ k } \otimes e_{ l } R_{ \sigma }^{ \top }
		& =
		(R_{ \sigma } e_{ k }) \otimes (R_{ \sigma } e_{ l } )
		=
		e_{ \sigma( k ) }\otimes e_{ \sigma( l ) }
		\shortintertext{and thus we have for every $ B = (b_{ i j })_{ i j } \in \R^{ 3 \times 3 } $ that}
		(R_{ \sigma } B R_{ \sigma }^{ \top })_{ i j } &=
		\sum_{ k l } b_{ k l } (R_{ \sigma } e_{ k } \otimes e_{ l } 
		R_{ \sigma }^{ \top } )_{ i j }
		\\
		& = \sum_{ k l } b_{ k l } (e_{ \sigma(k) } \otimes e_{ \sigma ( l ) } )_{ i j }
		\\
		& =
		\sum_{ k l } b_{ k l } \delta_{ \sigma( k ) i } \delta_{ \sigma( l ) j }
		\\
		& = B_{ \sigma^{-1}(i) \sigma^{-1}(j) }.
	\end{align*}
	From identity (\ref{eq:id_for_lattice_sum}) we can therefore deduce 
	by using $ R_\sigma^{\top } = R_{ \sigma^{-1}} $ that for all $ 
	\sigma \in S_{ 3 } $ and all $ i, j, k, l \in \{1,2,3\} $, we have
	\begin{align*}
		T_{ i j k l }
		&=
		T_{ \sigma^{ - 1} ( i ) \sigma^{ - 1 } ( j ) \sigma^{ - 1 } ( k ) \sigma^{ - 1 } ( l ) }
		\shortintertext{or equivalently}
		T_{ i j k l }
		&=
		T_{ \sigma ( i ) \sigma ( j ) \sigma ( k ) \sigma ( l ) } 
		\text{ for all } \sigma \in S_{3} \text{ and for all } i, j , k 
		,l \in \{1,2,3 \}.
	\end{align*}
	By combining this with equation (\ref{eq:T_zero_off_diagonal}) and 
	using $ K_{ i j } = K_{ j i } $, we can thus deduce that there must 
	exist $ \lambda, \mu , \nu \in \R $ such that for all $ i\neq j $, 
	we have
	\begin{align*}
		T_{ i i i i} &= \lambda
		\\
		T_{ i i j j } & = \nu 
		\\
		T_{ i j i j } = T_{ j i i j } &= \mu.
	\end{align*}
	We recall by equations (\ref{eq:greens_function_as_inverse}) and (\ref{eq:A_inverse_isotropic}) that in the isotropic case, $ K $ is given  by
	\begin{equation}
		\label{eq:k_in_isotropic_case}
		K( x) =  \fourier^{ - 1 } \left( \frac{ 1 }{ 2 \pi^2 \abs{ \xi 
			}^{ 2 } } \left( \mathrm{Id} - c \frac{ \xi }{ \abs{ \xi }} 
		\otimes \frac{\xi }{\abs{ \xi } } \right)\right) ( x )
	\end{equation}
	for constants $ c, C \in \R $.
	Thus we have that
	\begin{align*}
		6 \mu  + 3 \lambda 
		& =
		\sum_{i, j = 1  }^{ 3 } T_{ i j i j }
		\\
		& = 
		C
		\lim_{ n \to \infty }
		\sum_{ x \in \lattice_{ 1 } \cap B_{n } \setminus \{ 0 \} }
		\sum_{ i ,j = 1 }^{ 3 }
		\partial_{ i i } \fourier^{ - 1 } \left(
		\frac{ 1 }{ 2 \pi^2\abs{ \xi }^{ 2 } } \left( 1 - c \frac{ 
			\xi_{ j }^{ 2 } }{ \abs{ \xi }^{ 2 } } \right)\right) ( x )
		\\
		& = -C
		\lim_{ n \to \infty }
		\sum_{ x \in \lattice_{ 1 } \cap B_{n } \setminus \{ 0 \} }
		\fourier^{ -1 }
		\left(
		\sum_{ i ,j=1 }^3
		\frac{2\xi_{ i }^{ 2}}{ \abs{ \xi }^{ 2 } }
		\left(
		1 -c \frac{ \xi_{ j }^{ 2 } }{ \abs{ \xi }^{ 2 } }
		\right)
		\right) ( x )
		\\
		& = -C 
		\lim_{ n \to \infty }
		\sum_{ x \in \lattice_{ 1 } \cap B_{n } \setminus \{ 0 \} }
		\fourier^{ -1 }
		\left(
		2 (3-c)
		\right)
		( x )
		\\
		& = 0
	\end{align*}
	since the inverse Fourier transform of a constant function is the 
	Dirac measure at zero. We thus obtain the relation
	\begin{equation}
		\label{eq:relation_mu_lambda}
		\mu = - \frac{ \lambda }{ 2 }.
	\end{equation}
	Using a similar argument, we can also relate $ \nu $ in the same manner by computing
	\begin{align*}
		6 \nu 
		&=
		-C 
		\lim_{ n \to \infty }
		\sum_{ x \in \lattice_{ 1 } \cap B_{ n } \setminus \{ 0 \} }
		\sum_{ i \neq j }\fourier^{ - 1 } \left(
		\frac{ 2 \xi_{ i } \xi_{ j }}{ \abs{ \xi }^{ 2 } } \left( 0 - c \frac{ \xi_{ i }\xi_{ j } }{ \abs{ \xi }^{ 2 } } \right)\right) ( x )
		\\
		&=
		-C 
		\lim_{n\to \infty }
		\sum_{ x \in \lattice_{ 1 } \cap B_{ n } \setminus \{ 0 \} }
		\sum_{ i \neq j }
		\fourier^{ - 1 } \left(
		\frac{ 2 \xi_{ i }^{ 2 } }{ \abs{ \xi }^{ 2 } } \left( 1 - c \frac{ \xi_{ j }^{ 2 } }{ \abs{ \xi }^{ 2 } } \right) \right) ( x )
		-
		\fourier^{ - 1 } \left(
		\frac{ 2 \xi_{ i }^{ 2 } }{ \abs{ \xi }^{ 2 } } \right) ( x )
		\\
		& =
		-C 
		\lim_{ n \to \infty }
		\sum_{ x \in \lattice_{ 1 } \cap B_{ n } \setminus \{ 0 \} }
		\left( \sum_{ i \neq j }
		\fourier^{ - 1 } \left(
		\frac{ 2 \xi_{ i }^{ 2 } }{ \abs{ \xi }^{ 2 } } \left( 1 - c \frac{ \xi_{ j }^{ 2 } }{ \abs{ \xi }^{ 2 } } \right) \right) ( x )
		\right)
		-
		2 \fourier^{ - 1 } ( 2 ) ( x ) 
		\\
		&=6 \mu + 0,
	\end{align*}
	from which we deduce together with equation (\ref{eq:relation_mu_lambda}) that
	\begin{equation*}
		\nu = \mu = - \frac{ \lambda}{ 2 }.
	\end{equation*}
	This concludes the proof of the claim that $ T $ has the desired structure (\ref{eq:structure_of_lattice_sum}). 
	To conclude that $ S $ satisfies equation (\ref{eq:structure_of_lattice_sum}), we note that 
	\begin{enumerate}
		\item \label{item:symmetry_T} $ T_{ i j k l } = T_{ k l i j } = T_{  l k i j } $ for all $ i,j,k,l\in \{1,2,3\} $,
		\item \label{item:zero_on_skeq}$ T F = 0 $ for every $ F \in \R^{ 3 \times 3 }_{ \mathrm{skew} } $, 
		\item $ T ( \mathrm{Id} ) = 0 $ and
		\item \label{item:trace_zero} $ \trace T B = 0 $ for every $ B \in \R^{ 3 \times 3 } $.
	\end{enumerate}
	Using the relation of $ S $ and $ T $ (\ref{eq:S_relation_to_T}), 
	we get
	\begin{equation*}
		S = L T L^{ \top }.
	\end{equation*}
	To conclude that $ S = T $, we compute using items \ref{item:symmetry_T}-\ref{item:trace_zero} that for any given matrix $ A \in \R^{ 3 \times 3 } $, we have
	\begin{align*}
		L T L^{ \top } ( A ) & =
		L T L ( A )
		\\
		& = \frac{ T ( L A ) + ( T ( L A ))^{ \top } }{ 2 }
		+ 
		a \trace ( T ( LA ) ) \mathrm{Id} 
		\\
		& = 
		T ( L A )
		\\
		& = 
		T \left( \frac{ A + A^{ \top } }{ 2 } + a \trace ( A ) \mathrm{Id} \right)
		\\
		& = 
		T \left( A \right),
	\end{align*}
	which proves $ T = S $, finishing the proof.
\end{proof}

\subsection{Indefiniteness}

By using the structure (\ref{eq:structure_of_lattice_sum}) of the lattice sum, we can now show that $ \Psi $ is indefinite on the space of $ \Hm $-measures.
\begin{lemma}
	\label{lem:indefiniteness}
	Suppose that $ \lattice_{ 1 } $ has cubic symmetry in the sense of equation (\ref{eq:cubic_symmetry}) and that $ \eltensor $ is an isotropic elasticity tensor in the sense of (\ref{eq:def_isotropic_elasticity_tensor}). Let the lattice sum $ S $ and kernel $ M $ be defined by equation (\ref{eq:def_lattice_sum}) respectively equation (\ref{eq:formula_M}). Then $ (\hat{ M } ( \xi ) - (\hat{ M } )_{ \Sph^{ 2 } } + S ) $  is indefinite on the set of $ \Hm $-measures in the sense that we find $ \Hm $-measures $ \mu_{ 1} $ and $ \mu_{ 2 } $ such that 
	\begin{align}
		\label{eq:negative_h_measure}
		\int_{ \R^{ 3 } \times \Sph^{ 2 } }
		(\hat{ M } ( \xi ) - (\hat{ M } )_{ \Sph^{ 2 } } + S )
		\cdot
		\dd{ \mu_{ 1 } ( x, \xi ) } &< 0 
		\shortintertext{and}
		\label{eq:positive_h_measure}
		\int_{ \R^{ 3 } \times \Sph^{ 2 } }
		(\hat{ M } ( \xi ) - (\hat{ M } )_{ \Sph^{ 2 } } + S )
		\cdot
		\dd{ \mu_{ 2 } ( x, \xi ) } &> 0.
	\end{align}
	Moreover $ \mu_{ 1 } $ and $ \mu_{ 2 } $ can be chosen such that they are generated by sequences of real-valued $ \lp^{ 2 } $-functions.
\end{lemma}
\begin{proof}
	We first note that $ S $ is either indefinite or 0 in the following sense.
	Consider indices $ i, j \in \{1, 2, 3 \} $  and let $ A = (e_{ i } \otimes e_{ j }) \otimes ( e_{ i } \otimes e_{ j } ) $, which is symmetric and positive-semidefinite. Moreover if $ i \neq j $, we have
	\begin{equation*}
		\inner*{ S }{ A } = - \frac{ \lambda }{ 2 }.
	\end{equation*}
	Similarly if $ i = j $, we see that
	\begin{equation*}
		\inner*{ S }{ A } = \lambda.
	\end{equation*}
	Next we note that for any $ i, j \in \{1,2,3\} $, the function $ \hat{ M }_{ i j i j } $ is not constant as can be seen by considering equation (\ref{eq:fourier_transform_of_M}) and the explicit computation of the Green's function in the isotropic case, see equation (\ref{eq:k_in_isotropic_case}). 
	
	By combining these two facts, we already get the claim: Since $ S $ is indefinite or zero, we find $ A_{ 1 }, A_{ 2 } $ symmetric and positive-semidefinite of the form $ (e_{ i } \otimes e_{ j } ) \otimes ( e_{ i } \otimes e_{j } ) $ such that 
	\begin{equation} 
		\label{eq:choice_of_Ai}
		\inner*{ S }{ A_{ 1 } } \leq 0 
		\quad \text{and} \quad \inner*{ S }{ A_{ 2 } } \geq 0 .
	\end{equation} 
	Consider the formula (\ref{eq:fourier_transform_of_M}) for $ \hat{ 
		M } $. We note that $ \hat{ M }_{i j i j } $ is not constant for 
	any choice of indices $ i , j \in \{1,2,3 \} $ due to 
	\Cref{lem:M_not_constant}.
	Thus we find $ \xi_{ 1 }, \xi_{ 2 } \in \Sph^{ 2 } $ such that
	\begin{equation}
		\label{eq:choice_of_xi}
		\inner*{ \hat{ M } ( \xi_{ 1 } ) - (\hat{ M } )_{ \Sph^{ 2 } } }{ A_{ 1 } } < 0
		\quad \text{and} \quad
		\inner*{ \hat{ M } ( \xi_{ 2 } ) - (\hat{ M } )_{ \Sph^{ 2 } } }{ A_{ 2 } } > 0.		
	\end{equation}
	Take any $ f \in \ccinf ( \R^{ 3 } ) $ which is not zero. By \Cref{ex:h_measure_generated_by_sin_cos}, the measures 
	\begin{equation*} 
		\mu_{ 1  } \coloneqq  \abs{ f }^{ 2 } A_{ 1 } \dd{ \lm^{ d } } \otimes \frac{ \delta_{ \xi_{ 1 } } + \delta_{ - \xi_{ 1 } } }{ 4 }
		\quad \text{and} \quad
		\mu_{ 2  } \coloneqq  \abs{ f }^{ 2 } A_{ 2 } \dd{ \lm^{ d } } \otimes \frac{ \delta_{ \xi_{ 2 } } + \delta_{ - \xi_{ 2 } } }{ 4 }
	\end{equation*} 
	are $ \Hm $-measures generated by sequences of real-valued $ \lp^{ 2 } $-functions. 
	By \Cref{lem:symmetry_properties}, the kernel $ \hat{ M } $ is even.
	Combining this with the previous inequalities (\ref{eq:choice_of_Ai}) and (\ref{eq:choice_of_xi}), it follows that
	\begin{align*}
		& \int_{ \R^{ 3 } \times \Sph^{ 2 } }
		( \hat{ M } ( \xi ) - ( \hat{ M } )_{ \Sph^{ 2 } } + S )
		\cdot 
		\dd{ \mu_{ 1 } ( x, \xi ) }
		\\
		={} & 
		\frac{1}{2}
		\norm{ f }_{ \lp^{ 2 } ( \R^{ 3 } ) }^{ 2 }
		\inner*{\frac{ \hat{ M } ( \xi_{ 1 } ) + \hat{ M } ( - \xi_{ 1 } ) }{ 2 } - ( \hat{ M } )_{ \Sph^{ 2 } }  + S }{ A_{ 1 } }
		\\
		= {} &
		\frac{1}{2}
		\norm{ f }_{\lp^{ 2 } ( \R^{ 3 } )}^{ 2 }
		\left(
		\inner*{ \hat{ M } ( \xi_{ 1 } )  - ( \hat{ M }  )_{ \Sph^{ 2 } } }{A_{ 1 } } + \inner*{ S}{ A_{ 1 } } 
		\right) 
		< 0, 
	\end{align*}
	and the same holds true for $ \mu_{ 2 } $ with the opposite sign. Thus the claimed inequalities (\ref{eq:negative_h_measure}) and (\ref{eq:positive_h_measure}) follow.
\end{proof}

By constructing the appropriate sequence of arrays of dislocations using \Cref{lem:indefiniteness}, we are able to finish the proof of \Cref{cor:negative_interaction}.
\begin{proof}[Proof of \Cref{cor:negative_interaction}]
	By \Cref{lem:indefiniteness} and its proof, we find indices $ i, j \in \{ 1, 2, 3 \} $, some $ \nu \in \Sph^{ 2 } $ and some real-valued $ f \in \ccinf ( \R^{ 3 } ) $ such that the sequence (for some given $ \alpha \in (0,1 ) $)
	\begin{equation*}
		f_{ \rho } ( x ) 
		\coloneqq
		\ef ( x )
		\sin \left( 2 \pi  \rho^{ - \alpha } \inner*{ x }{ \nu } 
		\right) e_{ i } \otimes e_{ j }
	\end{equation*}
	generates an $ \Hm $-measure $ \mu $ which satisfies
	\begin{equation}
		\label{eq:choice_of_mu}
		\int_{ \R^{ 3 } \times \Sph^{ 2 } }
		( \hat{ M } ( \xi ) - ( \hat{ M } )_{ \Sph^{ 2 } } + S )
		\cdot
		\dd{ \mu ( x, \xi ) }
		< 
		0.
	\end{equation}
	In order to generate the same $ \Hm $-measure by a sequence of arrays of dislocation loops, we consider the piecewise constant function
	\begin{equation*}
		\widetilde{ \of_{ \rho } } ( y )
		\coloneqq
		\sum_{ x \in \lattice_{ \rho } }
		\chi_{ x + \rho U } ( y )
		f_{ \rho } ( x ).
	\end{equation*}
	Since $ \alpha \in (0, 1 ) $, the scale of oscillation of $ f_{ \rho } $ is larger than the lattice size $ \rho $, from which we can deduce that
	\begin{equation*}
		\lim_{ \rho \to 0 }
		\norm{ \widetilde{\of_{ \rho } }- f_{ \rho } }_{ \lp^{ 2 } ( \R^{ 3 } ) }
		= 
		0. 
	\end{equation*}
	Thus $\widetilde{ \of_{ \rho} } $ also generates the $ \Hm 
	$-measure $ \mu $.
	
	Let $ r_{ \rho } $ be a sequence such that $ r_{ \rho }/\rho \to 0 $ as $ \rho \to 0 $. At each $ x \in \lattice_{ \rho } $ place an oriented circle $ \gamma_{ \rho } ( x ) $ with diameter $ r_{ \rho } $ such that the oriented disc $ S_{ \rho } ( x ) $ which satisfies $  \partial S_{ \rho } ( x ) = \gamma_{ \rho } ( x ) $ has the normal $ e_{ j } $, where $ j $ is chosen as above. Furthermore we define the Burgers vector
	\begin{equation}
		\label{eq:burgers_vector_choice}
		b_{ \rho } ( x )
		\coloneqq
		\frac{ \rho^{ 3 } }{  \hm^{ 2 } ( S_{ \rho } ( x ) ) }
		f ( x ) \sin( 2 \pi \rho^{ - \alpha } \inner*{ x }{ \nu } )
		e_{ i }
	\end{equation}
	and let $ \disdens_{ \rho } \coloneqq \left( b_{ \rho } ( x ) 
	\otimes \tau \hm^{ 1 } \llcorner_{ \gamma_{ \rho } ( x ) } 
	\right)_{ x \in \lattice_{ \rho } } $. The corresponding oriented 
	surface area function (see equation 
	(\ref{eq:def_or_area_function})) is then given by
	\begin{equation*}
		\of_{ \rho } ( x )
		=
		\rho^{ 3 }
		f ( x ) \sin ( 2 \pi \rho^{ - \alpha } \inner*{ x }{ \nu } ) 
		e_{ i } \otimes e_{ j },
	\end{equation*}
	and thus its piecewise constant extension defined by equation 
	(\ref{eq:def_extension_function}) coincides with $ \widetilde{ 
		\of_{ \rho} } $. 
	Using \Cref{prop:weak_long_energy_convergence}, we thus obtain that
	\begin{equation*}
		\lim_{ \rho \to 0 }
		\interactionEnergy ( \disdens_{ \rho } )
		=
		\frac{1}{2}
		\int_{ \R^{ 3 } \times \Sph^{ 2 } }
		( \hat{ M } ( \xi ) - ( \hat{ M } )_{\Sph^{ 2 } } + S )
		\cdot
		\dd{ \mu ( x , \xi ) },
	\end{equation*} 
	as long as we can show that $ \disdens_{ \rho } $ satisfies the assumptions \ref{item:l2BoundOnArea} and \ref{item:wellSeperated} and only has long-range oscillations. If they hold, then the proof is complete by inequality (\ref{eq:choice_of_mu}). Indeed for the desired $ \lp^{ 2 } $-bound \ref{item:l2BoundOnArea}, we compute
	that 
	\begin{equation*}
		\sum_{ x \in \lattice_{ \rho } }
		\frac{ 1 }{ \rho^{ 3 } }
		\abs{ \int_{ S_{\rho } ( x ) } b ( x ) \dd{ \hm^{ 2 } } }^{ 2 }
		\leq
		\sum_{ x \in \lattice_{ \rho } }
		\rho^{ 3 } \abs{ f ( x ) }^{ 2 }
	\end{equation*}
	which due to $ f \in \ccinf ( \R^{ 3 } ) $ stays uniformly bounded. The well-separatedness condition \ref{item:wellSeperated} follows by our choice of $ r_{ \rho } $ and the loops $ \gamma_{ \rho } $. Finally no short-scale oscillations occur by our choice of $ \alpha \in (0,1 ) $.
\end{proof}

We want to conclude this Section by showing how our result points towards the formation of microstructures due to the possibly negative interaction energy.
\begin{example}
	\label{ex:rotations_energy}
	Assume that we are still in the setting of isotropic elasticity 
	(\ref{eq:def_isotropic_elasticity_tensor}).
	Imagine a setting where dislocation loops are placed on a lattice as described in our model. We assume that for some reason, the self-energy and location of the dislocations are fixed. Our aim is to show that it is then favourable for the configuration to produce small-scale oscillations. The special case we consider is that all dislocation loops face the same direction and have the same Burgers vector. We show that, without changing the location and the self-energy of the loops, a configuration with less total interaction energy exists by introducing rotations of the loops.
	
	We assume that $ \gamma \subseteq \R^{ 3 } $ is an oriented dislocation loop with tangent vector $ \tau $ such that
	for a corresponding slip surface with $ \partial S = \gamma $, we have $ \int_{ S } n \dd{ \hm^{ 2} } \eqqcolon a $. 
	Let $ b \in \R^{ 3 } \setminus \{ 0 \} $ be a Burgers vector, which 
	has been rescaled so that its length stays equal to $ 1 $. Moreover 
	let $ 0 < r = r ( \rho ) \ll \rho  $. Since the dislocation loops 
	should have to be volume-preserving, it follows that $ 
	\inner*{b}{a} = 0 $. 
	The configuration we want to consider is given by
	\begin{equation}
		\label{eq:constant_config}
		\disdens_{ \rho } 
		\coloneqq 
		\frac{\rho^{ 3 }}{ r^{ 2 } }
		\sum_{ x \in \lattice_{ \rho } \cap Q }
		b \otimes \tau \hm^{ 1 } \llcorner_{ r \gamma + x  },
	\end{equation}
	where $ Q \subseteq \R^{ 3 } $ is the unit cube.
	We want to find  oscillating rotations $ R_{ \theta } \in \sporth ( 
	3 ) $ such that the array of dislocation loops
	\begin{equation}
		\label{eq:E_rho_def}
		\mathcal{E}_{ \rho }
		\coloneqq
		\frac{\rho^{ 3 } }{ r^{ 2 } }
		\sum_{ x \in \lattice_{ \rho } \cap Q}
		(R_\theta b) \otimes (R_\theta \tau) \hm^{ 1 } \llcorner_{ r 
			R_{ \theta }( 
			\gamma ) + x }
	\end{equation}
	has less total interaction energy in the limit $ \rho \to 0 $.
	Since the only modification is the same rotation of the dislocation 
	line and the Burgers vector, the self-energy of $ \mathcal{E}_{ 
		\rho } $ is the same as the self-energy of $ \disdens_\rho $, and 
	we 
	do not change the position of the loops.
	
	As observed above, the Burgers vector $ b $ and the oriented 
	surface area $ a $ are orthogonal, thus we find some $ R_{ - 1 } 
	\in 
	\sporth( 3 ) $ such that $ R_{-1} ( b ) = - b $ and $ R_{-1} ( a ) 
	= a $. 
	Define $ R_{ 1 } \coloneqq \Id $. Let $ \nu \in \Sph^{ 2 } $ 
	be a unit vector. Define $ \sigma \colon \R \to 
	\{ - 1 , 1\} $ as the sign-function and denote
	\begin{equation*}
		\theta_{ \rho } ( x ) \coloneqq 
		\sigma ( \sin ( 2 \pi \rho^{ -1/2} \inner*{x}{\nu} ) ).
	\end{equation*}
	We then define as in (\ref{eq:E_rho_def}) the sequence of arrays of 
	dislocation loops $ \mathcal{E}_{ \rho } $ by 
	setting $ R_{ \theta } = R_{ \theta_{ \rho } ( x ) } $.
	
	By inserting the definition of the oriented surface area function 
	(\ref{eq:def_or_area_function}), we note that
	\begin{equation*}
		\of_{ \rho} ( x ) 
		=
		\rho^{ 3 }
		(R_{ \theta } b ) \otimes ( R_{ \theta } a )
		\chi_{ Q } ( x )
		=
		\rho^{ 3 }
		\theta_{ \rho } ( x ) (b \otimes a) \chi_{ Q } ( x ) 
	\end{equation*}
	To compute the limit of the total interaction energy of $ 
	\mathcal{E}_{ \rho } $, we first note that $ \theta_{ \rho } \chi_{ 
		Q } $ converges to zero weakly in $ \lp^{ 2 } ( \R^{ 3 } ) $ and 
	only has oscillations of weak-long type.
	Thus, by \Cref{prop:weak_long_energy_convergence}, we need to 
	compute the $ \Hm $-measure the sequence $ \of_{ \rho} $ generates. 
	In fact, since $ \theta_{ \rho } $ oscillates between $ 1 $ and $ - 
	1 $, and the direction of oscillations is $ \nu $, we should have 
	that
	\begin{equation}
		\label{eq:Hm_claim}
		\mu_{ \Hm }
		=
		\lm^{ 3 } \llcorner_{ Q } \otimes \frac{1}{2} ( \delta_{ \nu } 
		+ \delta_{ -\nu } )
		(b \otimes a ) \otimes ( b \otimes a ).
	\end{equation}
	We are now going to show this claim. First define the function $ f 
	\colon \R \to \R $ via $ f ( t ) \coloneqq \sigma ( \sin ( 2 \pi t 
	) ) $. Since $ f $ is locally integrable and periodic, we can write 
	it as its Fourier series
	\begin{equation*}
		f ( t ) = \sum_{ n \in \Z }
		f_{ n } \exp (  2 \pi i n t ),
	\end{equation*}
	where the Fourier coefficients $ f_n \in \C $ are given by
	\begin{equation*}
		f_n \coloneqq \int_0^1 f ( t ) \exp ( - 2 \pi i n t ) \dd{ t }.
	\end{equation*}
	We consider the definition of the $ \Hm $ measure given by 
	\Cref{thm:H_measures}. To this end, take test functions $ \phi_{ 1 
	}, \phi_{2} \in \cont_{\mathrm{c}} ( \R^{3 } ) $ and a $ 0 
	$-homogeneous function $ \psi \in \cont ( \R^{ 3 } \setminus \{ 0 
	\} ) $. 
	Since the 
	Fourier coefficients $ f_{ n } $ are square-summable, we have by 
	the dominated convergence theorem that
	\begin{align*}
		&
		\int_{ \R^{ 3 } }
		\fourier ( \phi_1 \theta_{ \rho }\chi_{ Q }  ) 
		\overline{ \fourier \left( \phi_2 \theta_{ \rho } \chi_{ Q } 
			\right) \psi  
		}
		\dd{ \xi }
		\\
		={}&
		\sum_{ n, m \in \Z }
		f_{ n } \overline{f_{ m } }
		\int_{ \R^{ 3 } }
		\fourier ( \phi_1 \chi_Q ) \ast \fourier (\exp ( - 2 \pi i n  
		\rho^{-1/2 } \inner*{\cdot}{ \nu } ) )
		\overline{ \fourier ( \phi_2 \chi_Q ) \ast \fourier  (\exp ( - 
			2 \pi i m \rho^{-1/2 } \inner*{\cdot }{ \nu } ) \psi }
		\dd{ \xi }.  
	\end{align*}
	By using that $ \fourier ( \exp ( - 2 \pi i n \rho^{-1/2} 
	\inner*{\cdot}{\nu} ) ) 
	= \delta_{ -n \rho^{-1/2}\nu } $, this term can be rewritten as
	\begin{equation}
		\label{eq:sum_over_nm}
		\sum_{ n, m \in \Z }
		f_n \overline{f_m}
		\int_{ \R^{ 3 } } 
		\fourier ( \phi_1 \chi_{ Q} ) ( \xi + n \rho^{-1/2} \nu )
		\overline{\fourier ( \phi_2 \chi_{ Q } ) ( \xi + m \rho^{-1/2} 
			\nu ) \psi }\dd{ \xi }.
	\end{equation}
	As in \Cref{ex:h_measures_of_simple_form}, the off-diagonal terms 
	in (\ref{eq:sum_over_nm}) cancel out in the limit $ \rho \to 0 $, 
	and for $ n \in \Z $, the diagonal terms converge to 
	\begin{equation*}
		\abs{f_n}^2
		\int_{ \R^{ 3 } }
		\fourier ( \phi_1 \chi_{ Q } ) \overline{ \fourier ( \phi_2 
			\chi_Q ) } \dd{ \xi }
		\overline{\psi ( \sigma ( n ) \nu ) }.
	\end{equation*}
	Note that due to $ f $ having mean zero, $ f_0 = 0 $, and since $ f 
	$ is real-valued, its Fourier coefficients satisfy $ f_n = 
	\overline{f_{-n}} $.
	Applying the Plancherel identity and the dominated convergence 
	theorem thus yields that 
	\begin{equation}
		\label{eq:limit_to_compute_muh}
		\lim_{ \rho \to 0 }
		\int_{ \R^{ 3 } }
		\fourier ( \phi_1 \theta_{ \rho }\chi_{ Q }  ) 
		\overline{ \fourier \left( \phi_2 \theta_{ \rho } \chi_{ Q } 
			\right) \psi  
		}
		\dd{ \xi }
		=
		\sum_{ n \in \N_{\geq 1 } }
		\abs{f_n}^{ 2 }
		\int_{ Q } \phi_1 \overline{ \phi_2 } \dd{ x }
		( \overline{\psi ( \nu ) + \psi ( - \nu ) } ).
	\end{equation}
	Again by the Plancherel identity and the above identities for $ f_n 
	$, we have that 
	\begin{equation*}
		\sum_{ n \in \N_{ \geq 1 } }
		\abs{f_n}^{ 2 }
		=
		\frac{1}{2}
		\int_0^{ 1 }
		\abs{ f ( t ) }^2
		\dd{ t }
		=\frac{1}{2}.
	\end{equation*}
	By combining this with equation (\ref{eq:limit_to_compute_muh}), we 
	conclude that $ \theta_\rho \chi_Q $ generates the $ \Hm $-measure
	\begin{equation*}
		\lm^{ 3 } \llcorner_{ Q } \otimes \frac{1}{2 } \left( 
		\delta_\nu + \delta_{ - \nu }  \right).
	\end{equation*}
	This proves the claim (\ref{eq:Hm_claim}).
	
	By now applying \Cref{prop:weak_long_energy_convergence}, we obtain 
	that
	\begin{equation}
		\label{eq:Erho_limit}
		\lim_{ \rho \to 0 }
		\interactionEnergy ( \mathcal{E}_{ \rho } )
		=
		\frac{1}{2}
		\int_{ \R^{ 3 } \times \Sph^{ 2 } }
		\Psi( \xi ) \cdot \dd{ \mu_{ \Hm } ( x, \xi ) }
		=
		\frac{1}{2}
		\int_{ \R^{ 3 } } \chi_Q \dd{ x }
		\frac{1}{2}
		\inner*{ (\Psi ( \nu ) + \Psi ( - \nu ) ) (b \otimes a) 
		}{b\otimes a}.
	\end{equation}
	We recall that we assume $ \eltensor $ is isotropic.
	From \Cref{lem:symmetry_properties} and the definition of $ \Psi $ 
	via (\ref{eq:Psi_def}), we can deduce that $ \Psi ( \nu ) = \Psi ( 
	- \nu ) $. Additionally applying the Plancherel identity to the 
	right hand side of (\ref{eq:Erho_limit}) thus yields that
	\begin{equation}
		\label{eq:comparison_E}
		\lim_{ \rho \to 0 }
		\interactionEnergy ( \mathcal{E}_{ \rho } )
		=
		\frac{1}{2}
		\int_{ \R^{ 3 } }
		\abs{ \fourier  \chi_{ Q } ( \xi ) }^{ 2 }
		\inner*{ \Psi ( \nu ) ( b \otimes a ) }{ b \otimes a }
		\dd{ \xi }
	\end{equation}
	On the other hand, the constant array of dislocation loops $ 
	\disdens_{ \rho } $ defined in (\ref{eq:constant_config}) has no 
	oscillations, and its associated oriented surface area function 
	weakly converges to $ \chi_{ Q } (b \otimes a ) 
	$.
	By again applying 
	\Cref{prop:weak_long_energy_convergence}, we obtain
	\begin{equation}
		\label{eq:comparison_D}
		\lim_{ \rho \to 0 }
		\interactionEnergy ( \disdens_\rho )
		=
		\frac{1}{2}
		\int_{ \R^{ 3 } }
		\abs{ \fourier \chi_{ Q } ( \xi ) }^{ 2 }
		\inner*{ \Psi ( \xi ) ( b \otimes a)  }{b \otimes a }
		\dd{ \xi }.
	\end{equation}
	By comparing the right hand sides of equations 
	(\ref{eq:comparison_E}) and (\ref{eq:comparison_D}), we see that 
	since $ \nu $ can be freely chosen, the rotated sequence $ 
	\mathcal{E}_{ \rho } $ can be constructed to have less total 
	(interaction) energy in the limit than 
	$ \disdens_\rho $. In fact, it suffices to show that the function
	\begin{equation*}
		\nu \mapsto \inner*{ \Psi ( \nu ) ( b\otimes a ) }{b \otimes a 
		}
	\end{equation*}
	is not constant, which is the content of \Cref{lem:M_not_constant} 
	since $ \Psi $ is the sum of $\hat{ M } $ and a constant term. This 
	finishes 
	our example.

\end{example}
\appendix

\section*{Acknowledgements}
\phantomsection
\addcontentsline{toc}{section}{Acknowledgements}

The author is grateful for the guidance provided by his advisors Prof. 
Sergio Conti and Prof. Stefan Müller. He also thanks Dr. Camillo Tissot 
for the discussions about $ \Hm $-measures, and Alexander West for his 
insights regarding the oriented surface area.

The author is grateful for the funding by the Deutsche 
Forschungsgemeinschaft (DFG, 
German Research Foundation) under Germany’s Excellence Strategy – 
EXC-2047/1 – 390685813, the DFG project 
211504053 - SFB 1060 and 539309657 - SFB 1720.

\section{Appendix}
\label{sct:appendix}
In the following Lemma, we show the existence of a strain field given a dislocation density. This has already been shown in \cite[Thm.~4.1]{conti_garroni_ortiz_the_line_tension_approximation_as_the_diluate_limit_of_linear_elastic_dislocations}, and we present a condensed version here.
\begin{lemma}
	\label{lemma:existence_of_strain_field}
	Given $ \rho \in \ccinf \left( \R^{ 3 }; \mathbb{R}^{ 3 \times 3 } \right) $ with $ \divg \rho = 0 $ and an elastic tensor $ \eltensor $, the equation
	\begin{align}
		\label{eq:curl_equation}
		\cur F & = \rho \\
		\label{eq:divergence_free}
		\divg (\eltensor F ) &= 0 
	\end{align}
	has a unique solution in $ \lp^{ 2 } \left( \R^{ 3 } ; \R^{ 3 \times 3 } \right) $.
\end{lemma}
\begin{proof}
	Applying the Fourier transform to equations (\ref{eq:curl_equation}) and (\ref{eq:divergence_free}) yields that for almost every $ \xi \in \R^{ 3 } $, we must have
	\begin{align}
		\label{eq:curl_fourier_form}
		2 \pi i \xi \times \fourier F ( \xi ) & = \fourier \rho ( \xi )
		\\
		\label{eq:div_fourier_form}
		2 \pi i \left( \eltensor \fourier F ( \xi ) \right) \xi & = 0.
	\end{align}
	Let $ \xi \neq 0 $.
	Consider the first equation (\ref{eq:curl_fourier_form}). Since $ \rho $ is divergence-free, we deduce that $ \fourier \rho ( \xi ) \xi = 0 $, or in other words, $ \fourier \rho ( \xi ) $ is row-wise orthogonal to $ \xi $. Thus we can solve equation (\ref{eq:curl_fourier_form}) with a function of the form 
	\begin{equation}
		\label{eq:decomp_fourier_F}
		\fourier F ( \xi ) 
		= 
		\Phi ( \xi ) + b ( \xi ) \otimes \xi ,
	\end{equation}
	where $ \Phi ( \xi ) $ is an almost everywhere uniquely determined function such that $ \Phi^{ i } ( \xi ) \in \xi^{ \perp } $ row-wise almost everywhere, and $ b \colon \R^{ 3 } \to \C^{ 3 } $ is some map which represents our degrees of freedom. Furthermore we must have that
	\begin{equation}
		\label{eq:growth_of_Phi}
		\abs{\Phi } \lesssim \frac{ \abs{ \fourier \rho ( \xi ) } }{ \abs{ \xi } }.
	\end{equation}
	Since $ \fourier \rho $ is a Schwartz function, this implies that $ \Phi \in \lp^{ 2 } \left( \R^{ 3 } ; \R^{ 3 \times 3 } \right) $. 
	
	To determine $ b $ we consider the second equation (\ref{eq:div_fourier_form}), which by the decomposition (\ref{eq:decomp_fourier_F}) can be written as
	\begin{equation}
		\label{eq:div_fourier_form_rewritten}
		( \eltensor b( \xi ) \otimes \xi ) \xi = 
		- ( \eltensor \Phi ( \xi ) ) \xi.
	\end{equation} 
	To solve this equation, consider the linear map $ A \colon \R^{ 3 } \to \R^{ 3 } $ defined by
	\begin{equation*}
		A ( b ) \coloneqq 
		( \eltensor b \otimes \xi ) \xi. 
	\end{equation*}
	Then $ b $ satisfies equation (\ref{eq:div_fourier_form_rewritten}) if and only if
	$ A( \xi ) b ( \xi ) = - (\eltensor\Phi ( \xi ) ) \xi $. But $ A $ is actually positive definite since
	\begin{equation*}
		\inner*{ A( \xi ) b }{ b }
		= 
		\eltensor b \otimes \xi \colon b \otimes \xi 
		\gtrsim
		\abs{ b }^{ 2 } \abs{ \xi }^{ 2 }.
	\end{equation*}
	Thus we can always solve equation (\ref{eq:div_fourier_form_rewritten}) and we moreover get that $ b $ is measurable with growth bound
	\begin{equation*}
		\abs{ b ( \xi ) } 
		\lesssim 
		\frac{ \abs{ \Phi ( \xi ) } \abs{ \xi } }{ \abs{ \xi }^{ 2 } }
		\lesssim 
		\frac{ \abs{ \fourier \rho  ( \xi ) } }{ \abs{ \xi }^{ 2 } }.
	\end{equation*}
	Looking back at equation (\ref{eq:decomp_fourier_F}), we see that $ \fourier F \in \lp^{2 } \left( \R^{ 3 } ; \R^{ 3 \times 3 } \right) $. Moreover the previous computation justifies that $ F $ is a distributional solution of equations (\ref{eq:curl_equation}) and (\ref{eq:divergence_free}) and shows that if an $ \lp^{ 2 } $ solution exists, it must be unique.
\end{proof}

\begin{lemma}
	\label{lem:cancellation_of_M}
	Let the kernel $ M $ be defined by equation (\ref{eq:formula_M}). Then we have
	\begin{equation*}
		\int_{ \Sph^{ 2 } }
		M ( z ) 
		\dd{ \hm^{ 2 } ( z ) }
		=
		0.
	\end{equation*}
\end{lemma}
\begin{proof}
	By formula (\ref{eq:formula_M}), it suffices to show the claim if $ M = \partial_l  f  $ for some $ -2 $-homogeneous function $ f \in C^{ \infty } ( \R^{ n } \setminus \{0 \} ) $. Applying the divergence theorem yields
	\begin{align*}
		\int_{ \Sph^{ 2 } }
		\partial_{ l } f 
		\dd{ \hm^{ 2 } }
		& =
		\lim_{ \eps \to 0 }
		\fint_{ B_{ 1 } \setminus B_{ 1 - \eps } }
		\partial_{ l } f 
		\dd{ x }
		\\
		& =
		\lim_{ \eps \to 0 }
		\frac{1}{ \lm^{ 3} ( B_{ 1 } \setminus B_{ 1 - \eps } ) }
		\left(\int_{ \partial B_{  1} } f x_{ l } \dd{ \hm^{ 2 } }
		-
		\int_{ \partial B_{ 1- \eps } } f \frac{x_l}{\abs{ x } } \dd{ \hm^{ 2 } ( x ) } 
		\right)
		\\
		& =
		\lim_{ \eps \to 0 }
		\frac{1}{ \lm^{ 3} ( B_{ 1 } \setminus B_{ 1 - \eps } ) }
		\int_{ \partial B_{  1} } f x_{ l } 
		-
		(1- \eps)^{ 2 }f ( (1- \eps ) x ) x_l \dd{ \hm^{ 2 } }
		\\
		& =
		0
	\end{align*}
	due to the $ - 2 $-homogeneity of $ f $, which finishes the proof.
\end{proof}

\begin{lemma}
	\label{lem:spherical_harmonics}
	Let $ K \in C^{ \infty } ( \Sph^{n-1} ) $ and decompose $ K $ into spherical harmonics via
	\begin{equation*}
		K( x ) = \sum_{ l = 0}^{ \infty } \sum_{ k = 1}^{ b_l }
		a_{ k}^{ (l)} Y_{ k}^{(l)},
	\end{equation*}
	where $ b_l = \mathrm{dim} ( \mathscr{H}_l) $ is the dimension of spherical harmonics of degree $ l $ and $ (Y_k^{(l)})_{k = 1, \ldots, b_l} $ is an orthonormal basis of $ \mathscr{H}_l $ with respect to $ \inner*{f}{g}=\int_{ \Sph^{n-1}} f \overline{g} \dd{ \hm^{ n -1 } } $. Then:
	\begin{enumerate}[label=(\roman*)]
		\item \label{item:dimension} There exists a constant $C_n> 0 $ 
		such that $ b_l \leq C_n l^{ n-2 } $ for all $ l \in \N $.
		\item \label{item:decay}For all $ j \in \N $ there exists $ A_j > 0 $ such that 
		\begin{equation*}
			\abs{a_k^{(l)}} \leq A_j l^{-j}
			\quad
			\text{for all }l \in \N , k \in \{ 1 , \ldots, b_l \}.
		\end{equation*}
	\end{enumerate}
\end{lemma}
\begin{proof}
	We first show \ref{item:dimension}.
	As noted after \cite[Cor.~4.2.2]{stein_weiss_introduction_to_fourier_analysis_on_euclidean_spaces}, we have
	\begin{equation*}
		\mathrm{dim} ( \mathscr{H}_l)
		=
		\binom{n+l-1}{l} - \binom{n+l-3}{l-2}.
	\end{equation*}
	Thus we can estimate
	\begin{equation*}
		b_l = \binom{n+l-3}{l-1} \frac{n+2l-2}{l}
		\leq 
		C_n l^{n-2}.
	\end{equation*}
	We now show \ref{item:decay}, which has been proven in \cite[Thm.~6.5]{calderon_60_integrales_singulares}, but since the source is in Spanish, we recall the proof for the convenience of the reader. First we note that if $ f, g \in C^{ 2 } ( \R^{ n } \setminus \{0\}) $ are $ 0 $-homogeneous, then for the operator $ L \varphi \coloneqq \abs{x}^{ 2 } \Delta \varphi $, we have that
	\begin{equation}
		\label{eq:operator_l_spehre}
		\int_{ \Sph^{ n - 1 } }
		Lf g - Lg f \dd{ \hm^{ n - 1 } }
		=
		0
	\end{equation}
	which follows from the divergence theorem via
	\begin{align*}
		\int_{ \Sph^{ n - 1} }
		Lf g - Lg f \dd{ \hm^{ n - 1 } }
		&=
		\lim_{ \delta \to 0 }
		\frac{1}{2\delta}
		\int_{ B_\delta ( \Sph^{ n -1 } ) }
		Lf g - Lg f \dd{ x }
		\\
		& =
		\lim_{ \delta \to 0 }
		\frac{1}{2 \delta }
		\int_{ \partial B_{ \delta }  ( \Sph^{ n - 1 } ) }
		f \partial_{ \nu } g - g \partial_{ \nu } f 
		\dd{ \hm^{ n -1 } }
		=0,
	\end{align*}
	where the last equality is due to the $ 0 $-homogeneity of $ f $ 
	and $ g $. Moreover, we observe that for the $ 0 $-homogeneous 
	extension of a spherical harmonic $ Y_k^{ (l) } $ of degree $ l 
	\geq 1 $, we have for all $ j \in \N $ that $ L^{ (j ) } Y_{k}^{ 
		(l) } =  l^{ j } ( 2 - n -l )^{j } Y_{ k }^{ (l) } $. To see this, 
	note that if $ P_{ k }^{ (l) } $ denotes the solid spherical 
	harmonic corresponding to $ Y_k^{ (l) } $ via $ P_{ k }^{ (l) } ( x 
	) = \abs{x}^{l} Y_{k }^{ (l) } ( x / \abs{ x } ) $, we have 
	\begin{equation}
		\label{eq:L_on_spehrical_harmonic}
		L Y_{ k }^{ l } 
		= \abs{x}^2 \Delta \frac{P_{ k }^{ (l ) } }{\abs{ x }^{ l } }
		=\abs{x}^{ 2 } \left( \Delta \abs{x}^{-l} P_k^{ (l) } + 2 
		\nabla \abs{x}^{ - l} \cdot \nabla P_{ k }^{(l ) } \right)
		=
		\frac{l ( 2 - n - l )}{ \abs{x}^{ l } }
		P_{ k }^{ (l)}
		=
		l ( 2 - n - l ) Y_{ k }^{ (l ) },
	\end{equation}
	where the third equality follows from $ \Delta \abs{x}^{ - l } = l 
	( l + 2 - n ) \abs{x}^{ -l -2 } $ and $ x \cdot \nabla P_k^{ ( l ) 
	} = l P_{ k }^{ ( l ) } $ due to the $ l $-homogeneity of $ P_k^{ ( 
		l ) } $.
	Combining equations (\ref{eq:operator_l_spehre}) and 
	(\ref{eq:L_on_spehrical_harmonic}) thus yields by the 
	orthonormality of $ ( Y_k^{ ( l ) } )_k $ that
	\begin{equation*}
		a_{ k }^{ ( l ) }
		= 
		\int_{ \Sph^{ n - 1 } } K 
		\overline{Y_{ k }^{ ( l ) }}
		\dd{ \hm^{ n - 1 } }
		=
		l^{ - j } (2 - n -l)^{ -j } \int_{ \Sph^{ n - 1 } } (L^{ (j) } 
		K) 
		\overline{Y_{ k }^{ l }} \dd{ \hm^{ n -1 } }
		\quad
		\text{for all } j \in \N.
	\end{equation*}
	Combined with the smoothness of $ K $, this implies the desired decay of $ a_{ k }^{ (l) } $, finishing the proof.
\end{proof}

\begin{lemma}
	\label{lem:M_not_constant}
	Assume that $ \eltensor $ is isotropic in the sense of 
	(\ref{eq:def_isotropic_elasticity_tensor}). Let $ b, c \in \R^{ 3 } 
	\setminus \{ 0 \} $ be either orthogonal or parallel. Then 
	\begin{align*}
		\varphi \colon \Sph^{ 2} & \to \R 
		\\
		\xi & \mapsto \inner*{ \hat{ M } ( \xi ) b \otimes c }{ b 
			\otimes c }
	\end{align*}
	is nowhere locally constant
\end{lemma}
\begin{proof}
	By analyticity of $ \varphi $, it suffices to find two different 
	values of $ \varphi $. First assume that $ b $ 
	and $ c $ are orthogonal, and that $ b = e_1 $, $ c = e_2 $. 
	Since $ \eltensor $ is isotropic, we obtain
	\begin{equation*}
		L ( e_{ 1 } \otimes e_2 ) =
		\frac{1}{2 } ( e_{ 1 } \otimes e_2 + e_2 \otimes e_1 ).
	\end{equation*}
	We recall from equation (\ref{eq:fourer_M_on_matrices}) how $ 
	\fourier M $ acts on matrices and deduce that
	\begin{equation*}
		\varphi ( \xi )
		=
		C
		\inner*{ \hat{K } ( \xi ) ( (\eltensor ( e_1 \otimes e_2 ) ) 
			\xi )}{ (\eltensor ( e_1 \otimes e_2 ) ) \xi}
		=
		C \inner*{ \hat{ K } ( \xi ) \begin{pmatrix}
				\xi_2 \\ \xi_1 \\ 0 
			\end{pmatrix}
		}
		{\begin{pmatrix}
				\xi_2 \\ \xi_1 \\ 0 
			\end{pmatrix}
		}
	\end{equation*}
	Since we explicitly know $ \hat{ K } $ from equation 
	(\ref{eq:fourier_K_formula}), it follows that there exists some 
	non-zero constant $ C $ such that
	\begin{equation*}
		\frac{1}{C}
		\varphi ( \xi ) 
		= 
		\left( \xi_1^2 + \xi_2^2 - 4 \frac{1+2a}{2+2a} \xi_1^2 \xi_2^2 
		\right).
	\end{equation*}
	Plugging in $ \xi = e_1 $ and $ \xi = e_3 $ yields different 
	values, proving the claim in the case $ e_1 \otimes e_2 $
	A similar computation also applies to $ e_1 \otimes e_1 $, namely 
	$ L ( e_1 \otimes e_1 ) = e_1 \otimes e_1 + a E_3 $, so that
	\begin{align*}
		\frac{1}{C }
		\varphi ( \xi )
		& =
		( 1 + a )^2 \xi_1^2 + a^2 ( \xi_2^2 + \xi_3^2 )
		-
		\frac{1+2a}{2+2a} \inner*{ \xi }{\begin{pmatrix}
				(1+a ) \xi_1 \\ a \xi_2 \\ a \xi_3 \end{pmatrix} }^2
		\\
		& =
		C_a + \xi_1^2 \left( 1 + 2a - \frac{1 + 2a }{2+2a} 2a \right)
		- 
		\xi_1^4 \frac{1 + 2 a}{ 2 + 2a }.
	\end{align*}
	Since $ (1+2a)/(2+2a) $ is non-zero due to $ a > -1/3 $, this 
	function is not constant on $ \Sph^2 $ as well.
	
	Lastly we have to argue why we may reduce the orthogonal case to $ 
	e_1 \otimes e_2 $ and the parallel case to $ e_1 \otimes e_1 $. 
	This follows from the rotational symmetry of $ \diff^{2} K $, see 
	(\ref{eq:rotational_symmetry_K}). 
	In fact, due to equation (\ref{eq:formula_M}), we have $ M = 
	\eltensor \diff^{ 2 } K \eltensor $, and thus
	for any rotation $ R \in \sporth( 3 ) $ and any $ A \in \R^{ 3 
		\times 3 } $ that
	\begin{equation*}
		M ( R x ) ( R A R^{ \top } )
		=
		R M ( x ) ( A ) R^{ \top },
	\end{equation*}
	and the same identity holds for $ \hat{ M } $. 
	If $ a $ and $ b $ are orthogonal and, without loss of generality, 
	of unit length, we find $ R \in \sporth ( 3 ) $ such that $ b = R 
	e_2 $ and $ c = R e_1 $. 
	It follows that
	\begin{equation*}
		\varphi ( R \xi )
		=
		\inner*{ \hat{ M } ( R \xi ) R e_{ 1 } \otimes e_2 R^{ \top } 
		}{ R e_1 \otimes e_2 R^{ \top } }
		=
		\inner*{ \hat{ M } ( \xi ) e_1 \otimes e_2 }{e_1 \otimes e_2 }.
	\end{equation*}
	This finishes the proof.
\end{proof}
\setlength{\emergencystretch}{2em}
		\printbibliography[
		heading=bibintoc,
		title={References}
		]

@article{blin_energy_of_dislocation_in_a_crystal,
	title={Energy of Dislocation in a Crystal},
	author={Jean Blin},
	journal={Acta Metall.},
	year={1955},
	volume={3:199}
}

@article{burgers_39_proceedings_I,
	AUTHOR = {Burgers, Johannes M.},
	TITLE = {Some considerations on the fields of stress connected with dislocations in a regular crystal lattice I},
	JOURNAL = {Proc. Kon. Ned. Akad. Wetenschap.},
	FJOURNAL = {Proceedings of the Koninklijke Nederlandse Akademie van Wetenschappen},
	YEAR = {1939},
	VOLUME = {42},
	PAGES = {293--325},
}

@article{burgers_39_proceedings_II,
	AUTHOR = {Burgers, Johannes M.},
	TITLE = {Some considerations on the fields of stress connected with dislocations in a regular crystal lattice {II}},
	JOURNAL = {Proc. Kon. Ned. Akad. Wetenschap.},
	FJOURNAL = {Proceedings of the Koninklijke Nederlandse Akademie van Wetenschappen},
	YEAR = {1939},
	VOLUME = {42},
	PAGES = {378--399},
}

@article {cermelli_leoni_05_renormalized_energy_and_forces_on_dislocations,
	AUTHOR = {Cermelli, Paolo and Leoni, Giovanni},
	TITLE = {Renormalized energy and forces on dislocations},
	JOURNAL = {SIAM J. Math. Anal.},
	FJOURNAL = {SIAM Journal on Mathematical Analysis},
	VOLUME = {37},
	YEAR = {2005},
	NUMBER = {4},
	PAGES = {1131--1160},
	ISSN = {0036-1410,1095-7154},
	MRCLASS = {74G65 (35J50 49S05 74E15 74G70)},
	MRNUMBER = {2192291},
	MRREVIEWER = {Jean--Jacques\ Marigo},
	DOI = {10.1137/040621636},
	URL = {https://doi.org/10.1137/040621636},
}

@article{conti_garroni_mueller_23_derivation_of_strain_gradient_plasticity_from_a_generalized_peierls_nabarro_model,
	AUTHOR = {Conti, Sergio and Garroni, Adriana and M\"uller, Stefan},
	TITLE = {Derivation of strain-gradient plasticity from a generalized
	{P}eierls-{N}abarro model},
	JOURNAL = {J. Eur. Math. Soc. (JEMS)},
	FJOURNAL = {Journal of the European Mathematical Society (JEMS)},
	VOLUME = {25},
	YEAR = {2023},
	NUMBER = {7},
	PAGES = {2487--2524},
	ISSN = {1435-9855,1435-9863},
	MRCLASS = {49J45 (49Q20 74E15)},
	MRNUMBER = {4612096},
	DOI = {10.4171/jems/1242},
	URL = {https://doi.org/10.4171/jems/1242},
}

@article{conti_garroni_mueller_singular_kernels_multiscale_decomposition_of_microstructure_and_dislocation_models,
	AUTHOR = {Conti, Sergio and Garroni, Adriana and M\"{u}ller, Stefan},
	TITLE = {Singular kernels, multiscale decomposition of microstructure,
	and dislocation models},
	JOURNAL = {Arch. Ration. Mech. Anal.},
	FJOURNAL = {Archive for Rational Mechanics and Analysis},
	VOLUME = {199},
	YEAR = {2011},
	NUMBER = {3},
	PAGES = {779--819},
	ISSN = {0003-9527},
	MRCLASS = {82B26 (49J45 74N05 74Q05)},
	MRNUMBER = {2771667},
	MRREVIEWER = {D. Polisevski},
	DOI = {10.1007/s00205-010-0333-7},
	URL = {https://doi.org/10.1007/s00205-010-0333-7},
}

@article {conti_garroni_ortiz_the_line_tension_approximation_as_the_diluate_limit_of_linear_elastic_dislocations,
	AUTHOR = {Conti, Sergio and Garroni, Adriana and Ortiz, Michael},
	TITLE = {The line-tension approximation as the dilute limit of
	linear-elastic dislocations},
	JOURNAL = {Arch. Ration. Mech. Anal.},
	FJOURNAL = {Archive for Rational Mechanics and Analysis},
	VOLUME = {218},
	YEAR = {2015},
	NUMBER = {2},
	PAGES = {699--755},
	ISSN = {0003-9527},
	MRCLASS = {74B05 (49J45 74N05)},
	MRNUMBER = {3375538},
	MRREVIEWER = {S. M. Giusti},
	DOI = {10.1007/s00205-015-0869-7},
	URL = {https://doi.org/10.1007/s00205-015-0869-7},
}

@article{cuitino_koslowski_ortiz_a_phase_field_theory_of_dislocation_dynamics_strain_hardening_and_hysteresis_in_ductile_single_crystals,
	author = {Marisol Koslowski and Alberto M. Cuitiño and Michael Ortiz},
	doi = {10.1016/S0022-5096(02)00037-6},
	issn = {0022-5096},
	journal ={J. Mech. Phys. Solids},
	fjournal = {Journal of the Mechanics and Physics of Solids},
	number = {12},
	pages = {2597–2635},
	title = {A phase-field theory of dislocation dynamics, strain hardening and hysteresis in ductile single crystals},
	url = {https://www.sciencedirect.com/science/article/pii/S0022509602000376},
	volume = {50},
	year = {2002}
}

@incollection{dash_57_observation_of_dislocations_in_silicon,
	AUTHOR = {Dash, William C.},
	TITLE = {The Observation of Dislocations in Silicon},
	BOOKTITLE =  {Dislocations And Mechanical Properties Of Crystals},
	PUBLISHER = {John Wiley and Sons, Inc., New York},
	YEAR = {1957}, 
	PAGES = {57--68},
}

@article{ewing_rosenhain_1899,
	AUTHOR = {Ewing, James A. and Rosenhain, Walter},
	JOURNAL = {Phil. Trans. Roy. Soc.},
	VOLUME = {A193},
	YEAR = {1899},
	PAGES = {353}
}

@article{firoozye_93_homogenization_on_lattices,
	author = {Firoozye, Nick},
	year = {1993},
	pages = {},
	title = {Homogenization on lattices: Small parameter limits, H-measures, and discrete Wigner measures}
}

@article {fonseca_ginster_wojtowytsch_2021_on_the_motion_of_curved_dislocations_in_3d,
	AUTHOR = {Fonseca, Irene and Ginster, Janusz and Wojtowytsch, Stephan},
	TITLE = {On the motion of curved dislocations in three dimensions:
	simplified linearized elasticity},
	JOURNAL = {SIAM J. Math. Anal.},
	FJOURNAL = {SIAM Journal on Mathematical Analysis},
	VOLUME = {53},
	YEAR = {2021},
	NUMBER = {2},
	PAGES = {2373--2426},
	ISSN = {0036-1410,1095-7154},
	MRCLASS = {35K93 (35Q74 74N05)},
	MRNUMBER = {4246087},
	DOI = {10.1137/20M1325654},
	URL = {https://doi.org/10.1137/20M1325654},
}

@article {fortuna_garroni_25_homogenization_of_line_tension_energies,
	AUTHOR = {Fortuna, Martino and Garroni, Adriana},
	TITLE = {Homogenization of line tension energies},
	JOURNAL = {Nonlinear Anal.},
	FJOURNAL = {Nonlinear Analysis. Theory, Methods \& Applications. An
	International Multidisciplinary Journal},
	VOLUME = {250},
	YEAR = {2025},
	PAGES = {Paper No. 113656, 16},
	ISSN = {0362-546X,1873-5215},
	MRCLASS = {49J45 (49Q20 74Q05)},
	MRNUMBER = {4802830},
	DOI = {10.1016/j.na.2024.113656},
	URL = {https://doi.org/10.1016/j.na.2024.113656},
}

@article {garroni_leoni_ponsiglione_2010_gradient_theory_for_plasticity_via_homogenization_of_discrete_dislocations,
	AUTHOR = {Garroni, Adriana and Leoni, Giovanni and Ponsiglione,
	Marcello},
	TITLE = {Gradient theory for plasticity via homogenization of discrete
	dislocations},
	JOURNAL = {J. Eur. Math. Soc. (JEMS)},
	FJOURNAL = {Journal of the European Mathematical Society (JEMS)},
	VOLUME = {12},
	YEAR = {2010},
	NUMBER = {5},
	PAGES = {1231--1266},
	ISSN = {1435-9855,1435-9863},
	MRCLASS = {35Q74 (35B27 74C05 74E15 74G70 74Q05)},
	MRNUMBER = {2677615},
	DOI = {10.4171/JEMS/228},
	URL = {https://doi.org/10.4171/JEMS/228},
}

@article {garroni_mueller_gamma_limit_of_a_phase_field_model_of_dislocations,
	AUTHOR = {Garroni, Adriana and M\"{u}ller, Stefan},
	TITLE = {{$\Gamma$}-limit of a phase-field model of dislocations},
	JOURNAL = {SIAM J. Math. Anal.},
	FJOURNAL = {SIAM Journal on Mathematical Analysis},
	VOLUME = {36},
	YEAR = {2005},
	NUMBER = {6},
	PAGES = {1943--1964},
	ISSN = {0036-1410},
	MRCLASS = {49J45 (74N05)},
	MRNUMBER = {2178227},
	MRREVIEWER = {G. Alberti},
	DOI = {10.1137/S003614100343768X},
	URL = {https://doi.org/10.1137/S003614100343768X},
}

@article {garroni_mueller_a_variational_model_for_dislocations_in_the_line_tension_limit,
	AUTHOR = {Garroni, Adriana and M\"{u}ller, Stefan},
	TITLE = {A variational model for dislocations in the line tension
	limit},
	JOURNAL = {Arch. Ration. Mech. Anal.},
	FJOURNAL = {Archive for Rational Mechanics and Analysis},
	VOLUME = {181},
	YEAR = {2006},
	NUMBER = {3},
	PAGES = {535--578},
	ISSN = {0003-9527},
	MRCLASS = {49J45 (74A50 74G10 74G65 74N05 74Q05)},
	MRNUMBER = {2231783},
	MRREVIEWER = {G. Alberti},
	DOI = {10.1007/s00205-006-0432-7},
	URL = {https://doi.org/10.1007/s00205-006-0432-7},
}

@article {gerard_markowich_mauser_97_hom_limits_and_wigner_transforms,
	AUTHOR = {G\'erard, Patrick and Markowich, Peter A. and Mauser, Norbert
	J. and Poupaud, Fr\'ed\'eric},
	TITLE = {Homogenization limits and {W}igner transforms},
	JOURNAL = {Comm. Pure Appl. Math.},
	FJOURNAL = {Communications on Pure and Applied Mathematics},
	VOLUME = {50},
	YEAR = {1997},
	NUMBER = {4},
	PAGES = {323--379},
	ISSN = {0010-3640,1097-0312},
	MRCLASS = {35B27 (35L99 35Q40 35S05)},
	MRNUMBER = {1438151},
	MRREVIEWER = {Albert\ J.\ Milani},
	DOI = {10.1002/(sici)1097-0312(199704)50:4<323::aid-cpa4>3.0.co;2-c},
	URL =
	{https://doi.org/10.1002/(sici)1097-0312(199704)50:4<323::aid-cpa4>3.0.co;2-c},
}

@article{gilman_johnston_57_origin_and_growth_of_glide_bands_in_lithium_fluoride_crystals,
	AUTHOR = {Gilman, John J. and Johnston, William G.},
	TITLE = {The Origin and Growth of Glide Bands in Lithium Fluoride Crystals},
	BOOKTITLE =  {Dislocations And Mechanical Properties Of Crystals},
	PUBLISHER = {John Wiley and Sons, Inc., New York},
	YEAR = {1957}, 
	PAGES = {116--163},
}

@article {ginster_19_plasticity_as_Gamma_without_separation,
	AUTHOR = {Ginster, Janusz},
	TITLE = {Plasticity as the {$\Gamma$}-limit of a two-dimensional
	dislocation energy: the critical regime without the assumption
	of well-separateness},
	JOURNAL = {Arch. Ration. Mech. Anal.},
	FJOURNAL = {Archive for Rational Mechanics and Analysis},
	VOLUME = {233},
	YEAR = {2019},
	NUMBER = {3},
	PAGES = {1253--1288},
	ISSN = {0003-9527,1432-0673},
	MRCLASS = {74C05 (35B27 49J45)},
	MRNUMBER = {3961298},
	MRREVIEWER = {Giovanni\ Garcea},
	DOI = {10.1007/s00205-019-01378-5},
	URL = {https://doi.org/10.1007/s00205-019-01378-5},
}

@article {ginster_19_strain_gradient_plast_mixed_growth,
	AUTHOR = {Ginster, Janusz},
	TITLE = {Strain-gradient plasticity as the {$\Gamma$}-limit of a
	nonlinear dislocation energy with mixed growth},
	JOURNAL = {SIAM J. Math. Anal.},
	FJOURNAL = {SIAM Journal on Mathematical Analysis},
	VOLUME = {51},
	YEAR = {2019},
	NUMBER = {4},
	PAGES = {3424--3464},
	ISSN = {0036-1410,1095-7154},
	MRCLASS = {49J45 (35B27 35Q74 58K45 74C05)},
	MRNUMBER = {3995039},
	MRREVIEWER = {Giovanni\ Scilla},
	DOI = {10.1137/18M1176579},
	URL = {https://doi.org/10.1137/18M1176579},
}

@article {james_mueller_internal_variables_and_fine_scale_oscillations_in_micromagnetics,
	AUTHOR = {James, Richard D. and M\"{u}ller, Stefan},
	TITLE = {Internal variables and fine-scale oscillations in
	micromagnetics},
	JOURNAL = {Contin. Mech. Thermodyn.},
	FJOURNAL = {Continuum Mechanics and Thermodynamics},
	VOLUME = {6},
	YEAR = {1994},
	NUMBER = {4},
	PAGES = {291--336},
	ISSN = {0935-1175,1432-0959},
	MRCLASS = {82D40 (49Q20 73B99 78A30)},
	MRNUMBER = {1308877},
	MRREVIEWER = {John\ M.\ Ball},
	DOI = {10.1007/BF01140633},
	URL = {https://doi.org/10.1007/BF01140633},
}

@article{koslowski_ortiz_2004_multi-phase_field_model_of_planar_dislocation_networks,
	doi = {10.1088/0965-0393/12/6/003},
	url = {https://dx.doi.org/10.1088/0965-0393/12/6/003},
	year = {2004},
	publisher = {},
	volume = {12},
	number = {6},
	pages = {1087},
	author = {Marisol Koslowski and Michael Ortiz},
	title = {A multi-phase field model of planar dislocation networks},
	journal = {Model. Simul. Mater. Sc.},
	fjournal = {Modelling and Simulation in Materials Science and Engineering}
}

@article {lions_paul_93_on_wigner_measures,
	AUTHOR = {Lions, Pierre-Louis and Paul, Thierry},
	TITLE = {Sur les mesures de {W}igner},
	JOURNAL = {Rev. Mat. Iberoamericana},
	FJOURNAL = {Revista Matem\'atica Iberoamericana},
	VOLUME = {9},
	YEAR = {1993},
	NUMBER = {3},
	PAGES = {553--618},
	ISSN = {0213-2230},
	MRCLASS = {58G15 (35Q40 47G10 81Q20)},
	MRNUMBER = {1251718},
	MRREVIEWER = {Peter\ N.\ Zhevandrov},
	DOI = {10.4171/RMI/143},
	URL = {https://doi.org/10.4171/RMI/143},
}

@article{muegge_mineralogie,
	AUTHOR = {Mügge, Otto},
	Journal = {Neues Jahrb. Min.},
	Volume = {13},
	Year = {1883},
}

@article {mueller_scardia_zeppieri_14_geometric_rigidity_straing_gradient_pl,
	AUTHOR = {M\"uller, Stefan and Scardia, Lucia and Zeppieri, Caterina
	Ida},
	TITLE = {Geometric rigidity for incompatible fields, and an application
	to strain-gradient plasticity},
	JOURNAL = {Indiana Univ. Math. J.},
	FJOURNAL = {Indiana University Mathematics Journal},
	VOLUME = {63},
	YEAR = {2014},
	NUMBER = {5},
	PAGES = {1365--1396},
	ISSN = {0022-2518,1943-5258},
	MRCLASS = {74C05 (49J45 74B20 74Q05)},
	MRNUMBER = {3283554},
	MRREVIEWER = {Heng\ Xiao},
	DOI = {10.1512/iumj.2014.63.5330},
	URL = {https://doi.org/10.1512/iumj.2014.63.5330},
}

@article{nabarro_47_dislocations,
	title = {Dislocations in a simple cubic lattice},
	volume = {59},
	url = {https://doi.org/10.1088/0959-5309/59/2/309},
	doi = {10.1088/0959-5309/59/2/309},
	pages = {256},
	number = {2},
	journal = {Proc. Phys. Soc.},
	fjournal = {Proceedings of the Physical Society},
	author = {Nabarro, Frank R. N.},
	date = {1947-03},
}

@article{orowan_34_kristallplastizitaetI,
	title = {Zur Kristallplastizität. I},
	volume = {89},
	issn = {0044-3328},
	url = {https://doi.org/10.1007/BF01341478},
	doi = {10.1007/BF01341478},
	pages = {605--613},
	number = {9},
	journal = {Z. Phys. },
	fjournal = {Zeitschrift für Physik},
	author = {Orowan, Egon},
	date = {1934-09-01},
}

@article{orowan_34_kristallplastizitaetIII,
	title = {Zur Kristallplastizität. {III}},
	volume = {89},
	issn = {0044-3328},
	url = {https://doi.org/10.1007/BF01341480},
	doi = {10.1007/BF01341480},
	pages = {634--659},
	number = {9},
	journal = {Z. Phys. },
	fjournal = {Zeitschrift für Physik},
	author = {Orowan, Egon},
	date = {1934-09-01},
}

@article{ortiz_plastic_yielding_as_phase_transition,
	author = {Ortiz, Michael},
	title = {Plastic Yielding as a Phase Transition},
	journal = {J. Appl. Mech.},
	fjournal = {Journal of Applied Mechanics},
	volume = {66},
	number = {2},
	pages = {289-298},
	year = {1999},
	issn = {0021-8936},
	doi = {10.1115/1.2791048},
	url = {https://doi.org/10.1115/1.2791048},
	eprint = {https://asmedigitalcollection.asme.org/appliedmechanics/article-pdf/66/2/289/6075698/289\_1.pdf},
}

@article{peierls_40_size_of_a_dislocation,
	title = {The size of a dislocation},
	volume = {52},
	url = {https://doi.org/10.1088/0959-5309/52/1/305},
	doi = {10.1088/0959-5309/52/1/305},
	pages = {34},
	number = {1},
	journal = {Proc. Phys. Soc.},
	fjournal = {Proceedings of the Physical Society},
	author = {Peierls, Rudolf E.},
	date = {1940-01},
}

@article{polanyi_34_gitterstoerung,
	title = {Über eine Art Gitterstörung, die einen Kristall plastisch machen könnte},
	volume = {89},
	issn = {0044-3328},
	url = {https://doi.org/10.1007/BF01341481},
	doi = {10.1007/BF01341481},
	pages = {660--664},
	number = {9},
	journal = {Z. Phys. },
	fjournal = {Zeitschrift für Physik},
	author = {Polanyi, Michael},
	date = {1934-09-01},
}

@article{ponsiglione_2007_elastic_energy_screw_dislocations_from_discrete_to_continuous,
	AUTHOR = {Ponsiglione, Marcello},
	TITLE = {Elastic energy stored in a crystal induced by screw
	dislocations: from discrete to continuous},
	JOURNAL = {SIAM J. Math. Anal.},
	FJOURNAL = {SIAM Journal on Mathematical Analysis},
	VOLUME = {39},
	YEAR = {2007},
	NUMBER = {2},
	PAGES = {449--469},
	ISSN = {0036-1410,1095-7154},
	MRCLASS = {74E15 (49J10 49J45 74C15 74G65 74N05 74N15)},
	MRNUMBER = {2338415},
	DOI = {10.1137/060657054},
	URL = {https://doi.org/10.1137/060657054},
}

@article {scardia_zeppieri_2012_line_tension_model_for_plasticity_as_gamma_limit_of_nonlinear_dislocation_energy,
	AUTHOR = {Scardia, Lucia and Zeppieri, Caterina Ida},
	TITLE = {Line-tension model for plasticity as the {$\Gamma$}-limit of a
	nonlinear dislocation energy},
	JOURNAL = {SIAM J. Math. Anal.},
	FJOURNAL = {SIAM Journal on Mathematical Analysis},
	VOLUME = {44},
	YEAR = {2012},
	NUMBER = {4},
	PAGES = {2372--2400},
	ISSN = {0036-1410,1095-7154},
	MRCLASS = {74C05 (49J45)},
	MRNUMBER = {3023380},
	MRREVIEWER = {Jos\'e\ Carlos Pedro Cardoso Matias},
	DOI = {10.1137/110824851},
	URL = {https://doi.org/10.1137/110824851},
}

@article {tartar_H_measures_a_new_approach,
	AUTHOR = {Tartar, Luc},
	TITLE = {{$H$}-measures, a new approach for studying homogenisation,
	oscillations and concentration effects in partial differential
	equations},
	JOURNAL = {Proc. Roy. Soc. Edinburgh Sect. A},
	FJOURNAL = {Proceedings of the Royal Society of Edinburgh. Section A.
	Mathematics},
	VOLUME = {115},
	YEAR = {1990},
	NUMBER = {3-4},
	PAGES = {193--230},
	ISSN = {0308-2105,1473-7124},
	MRCLASS = {35B27 (35B05)},
	MRNUMBER = {1069518},
	MRREVIEWER = {Pierre-Louis\ Lions},
	DOI = {10.1017/S0308210500020606},
	URL = {https://doi.org/10.1017/S0308210500020606},
}

@article{taylor_34_mechanism_of_plastic_deformation_of_crystals,
	author = {Taylor, Geoffrey I.},
	title = {The mechanism of plastic deformation of crystals. Part I.—Theoretical},
	journal = {Proc. Roy. Soc.},
	volume = {A145},
	number = {855},
	pages = {362-387},
	year = {1934},
	doi = {10.1098/rspa.1934.0106},
	URL = {https://royalsocietypublishing.org/doi/abs/10.1098/rspa.1934.0106},
	eprint = {https://royalsocietypublishing.org/doi/pdf/10.1098/rspa.1934.0106},
}

@article {young_42_generalized_surfaces_in_calcvar,
	AUTHOR = {Young, Laurence C.},
	TITLE = {Generalized surfaces in the calculus of variations},
	JOURNAL = {Ann. of Math. (2)},
	FJOURNAL = {Annals of Mathematics. Second Series},
	VOLUME = {43},
	YEAR = {1942},
	PAGES = {84--103},
	ISSN = {0003-486X},
	MRCLASS = {49.0X},
	MRNUMBER = {6023},
	MRREVIEWER = {E.\ J.\ McShane},
	DOI = {10.2307/1968882},
	URL = {https://doi.org/10.2307/1968882},
}

@article {volterra_07_equilibrium_of_elastic_bodies_with_multiple_components,
	AUTHOR = {Volterra, Vito},
	TITLE = {Sur l'\'equilibre des corps \'elastiques multiplement
	connexes},
	JOURNAL = {Ann. Sci. \'Ecole Norm. Sup. (3)},
	FJOURNAL = {Annales Scientifiques de l'\'Ecole Normale Sup\'erieure.
	Troisi\`eme S\'erie},
	VOLUME = {24},
	YEAR = {1907},
	PAGES = {401--517},
	ISSN = {0012-9593},
	MRCLASS = {99-04},
	MRNUMBER = {1509085},
	URL = {http://www.numdam.org/item?id=ASENS_1907_3_24__401_0},
}

@article{wigner_1932_quantum_correction,
	author  = {Wigner, Eugene P.},
	title   = {On the quantum correction for thermodynamic equilibrium},
	journal = {Phys. Rev.},
	volume  = {40},
	number  = {5},
	pages   = {749--759},
	year    = {1932},
	doi     = {10.1103/PhysRev.40.749}
}

@book {abramowitz_stegun_64_handbook_of_math_functions,
	AUTHOR = {Abramowitz, Milton and Stegun, Irene A.},
	TITLE = {Handbook of mathematical functions with formulas, graphs, and
	mathematical tables},
	SERIES = {National Bureau of Standards Applied Mathematics Series},
	VOLUME = {No. 55},
	PUBLISHER = {U. S. Government Printing Office, Washington, DC},
	YEAR = {1964},
	PAGES = {xiv+1046},
	MRCLASS = {33.00 (65.05)},
	MRNUMBER = {167642},
	MRREVIEWER = {D.\ H.\ Lehmer},
}

@book {ambrosio_fusco_pallara_functions_of_bv_and_free_discontinuity_problems,
	AUTHOR = {Ambrosio, Luigi and Fusco, Nicola and Pallara, Diego},
	TITLE = {Functions of bounded variation and free discontinuity
	problems},
	SERIES = {Oxford Mathematical Monographs},
	PUBLISHER = {The Clarendon Press, Oxford University Press, New York},
	YEAR = {2000},
	PAGES = {xviii+434},
	ISBN = {0-19-850245-1},
	MRCLASS = {49-02 (49J45 49K10 49Qxx)},
	MRNUMBER = {1857292},
	MRREVIEWER = {J. E. Brothers},
}

@book{anderson_hirth_lothe_theory_of_dislocations,
	title={Theory of Dislocations},
	author={Anderson, Peter M. and Hirth, John P. and Lothe, Jens},
	isbn={9780521864367},
	lccn={2016024204},
	url={https://books.google.de/books?id=LK7DDQAAQBAJ},
	year={2017},
	publisher={Cambridge University Press}
}

@book {bochner_55_harmonic_analysis_and_probability,
	AUTHOR = {Bochner, Salomon},
	TITLE = {Harmonic analysis and the theory of probability},
	PUBLISHER = {University of California Press, Berkeley-Los Angeles, Calif.},
	YEAR = {1955},
	PAGES = {viii+176},
	MRCLASS = {60.0X},
	MRNUMBER = {72370},
	MRREVIEWER = {J.\ L.\ Doob},
}

@book {calderon_60_integrales_singulares,
	AUTHOR = {Calder\'on, Alberto-P.},
	TITLE = {Integrales singulares y sus aplicaciones a ecuaciones
	diferenciales hiperbolicas},
	SERIES = {Cursos y Seminarios de Matem\'atica},
	VOLUME = {Fasc. 3},
	PUBLISHER = {Universidad de Buenos Aires, Buenos Aires},
	YEAR = {1960},
	PAGES = {121},
	MRCLASS = {35.53 (42.00)},
	MRNUMBER = {123834},
}

@book {rindler_18_calcvar,
	AUTHOR = {Rindler, Filip},
	TITLE = {Calculus of variations},
	SERIES = {Universitext},
	PUBLISHER = {Springer, Cham},
	YEAR = {2018},
	PAGES = {xii+444},
	ISBN = {978-3-319-77636-1},
	MRCLASS = {49-01},
	MRNUMBER = {3821514},
	DOI = {10.1007/978-3-319-77637-8},
	URL = {https://doi.org/10.1007/978-3-319-77637-8},
}

@book {stein_singular_integrals_and_differentiability_properties_of_functions,
	AUTHOR = {Stein, Elias M.},
	TITLE = {Singular integrals and differentiability properties of
	functions},
	SERIES = {Princeton Mathematical Series},
	VOLUME = {No. 30},
	PUBLISHER = {Princeton University Press, Princeton, NJ},
	YEAR = {1970},
	PAGES = {xiv+290},
	MRCLASS = {46.38 (26.00)},
	MRNUMBER = {290095},
	MRREVIEWER = {R.\ E.\ Edwards},
}

@book{stein_weiss_introduction_to_fourier_analysis_on_euclidean_spaces,
	AUTHOR = {Stein, Elias M. and Weiss, Guido},
	TITLE = {Introduction to {F}ourier analysis on {E}uclidean spaces},
	SERIES = {Princeton Mathematical Series, No. 32},
	PUBLISHER = {Princeton University Press, Princeton, N.J.},
	YEAR = {1971},
	PAGES = {x+297},
	MRCLASS = {42A92 (31B99 32A99 46F99 47G05)},
	MRNUMBER = {0304972},
	MRREVIEWER = {Edwin Hewitt},
}

@article {wainger_65_special_trigonometric_series_in_kd,
	AUTHOR = {Wainger, Stephen},
	TITLE = {Special trigonometric series in {$k$}-dimensions},
	JOURNAL = {Mem. Amer. Math. Soc.},
	FJOURNAL = {Memoirs of the American Mathematical Society},
	VOLUME = {59},
	YEAR = {1965},
	PAGES = {102},
	ISSN = {0065-9266,1947-6221},
	MRCLASS = {42.40},
	MRNUMBER = {182838},
	MRREVIEWER = {L.\ Berg},
}
	\end{document}